\documentclass[12pt, reqno]{amsart}
\usepackage{mathrsfs}
\usepackage{amsmath}
\usepackage[height=22cm, width=16.5cm, hmarginratio={1:1}]{geometry}
\usepackage{diagbox}
\usepackage{tikz}
\usepackage{tikz-cd}
\usetikzlibrary{positioning, shapes.geometric}
\usepackage{newclude}
\usepackage{appendix}
\usepackage{extarrows}
\usepackage{pifont}
\usepackage{array}
\usepackage[hidelinks]{hyperref}

\author{Shuai Guo}
\address{School of Mathematical Sciences, Peking University,
	No 5. Yiheyuan Road, Beijing, 100871, China}
\email{guoshuai@math.pku.edu.cn}

\author{Qingsheng Zhang}
\address{School of Sciences, Great Bay University,
	Great Bay Institute for Advanced Study,
	Dongguan, 523000, China}
\email{zhangqingsheng@gbu.edu.cn}

\author{Yang Zhou}
\address{Shanghai Center for Mathematical Sciences, Fudan University, 
	Shanghai, 200433, China}
\email{{\nolinkurl{y_zhou@fudan.edu.cn}}}

\newcommand{\lh}{{\rm h}}
\newcommand{\hh}{{\rm H}}

\newcommand{\cX}{{\mathcal X}}
\newcommand{\Frob}{H}
\newcommand{\twqp}{\star}

\newcommand{\<}{\langle}
\renewcommand{\>}{\rangle}
\newcommand{\git}{/\!\!/}

 \def\fanoind{c_X}

\newcommand{\corr}[1]{\langle\!\langle {#1} \rangle\!\rangle}

\def\sspan{{\mathop{\rm span}}}

\def\cD{{\mathcal D}}

\def\cG{{\mathcal G}}

\def\E{{\mathcal E}}

\newcolumntype{C}[1]{>{\centering\arraybackslash$}m{#1}<{$}}
\newlength{\mycolwd}                                         
\newlength{\mycolwdm}                                         
\newlength{\mycolwda}                                         
\newlength{\mycolwdb}                                         
\newlength{\mycolwdc}                                         
\newlength{\mycolwdd}                                         
\newlength{\mycolwddd}                                         
\newlength{\mycolwddc}                                         
\newcommand{\id}{\hspace{0pt}\text{I}}

\usepackage {xcolor,graphicx,psfrag, verbatim,amssymb,amscd,enumerate,subfigure}

\def\End{\mathrm{End}}
\def\bt{\mathbf t}

\newcommand{\tw}{\mathrm{tw}}

\def\PF{\mathcal{PF}}

\def\beq{\begin{equation}}
\def\eeq{\end{equation}}

 \def\pd{\partial}

 \newcommand{\cF}{\mathcal{F}}

 \DeclareMathOperator{\diag}{diag}

\newcommand{\tr}{\mathrm{tr}}
\newcommand{\str}{\mathrm{str}}

\newtheorem{dummy}{}[section]
\newtheorem{lemma}[dummy]{Lemma}
\newtheorem{proposition}[dummy]{Proposition}
\newtheorem{theorem}[dummy]{Theorem}
\newtheorem{corollary}[dummy]{Corollary}

\theoremstyle{definition}

\newtheorem{definition}[dummy]{Definition}
\newtheorem{example}[dummy]{Example}

\newtheorem{remark}[dummy]{Remark}

\newtheorem{thm}{Theorem}

\usepackage{graphicx}

\begin{document}
\title[Wall-crossing formula and Virasoro constraints]
{Wall-crossing formula and genus-one Virasoro conjecture for Fano complete intersections}

\begin{abstract}
  The Virasoro conjecture predicts a set of universal relations among all genera Gromov--Witten invariants of any smooth projective variety.
  The conjecture is well understood for semisimple theories, but remains largely open in the non-semisimple setting.
  
  We prove the genus-one Virasoro conjecture on the ambient state space of smooth Fano complete intersections in projective space. 
  For most of these complete intersections, the big quantum cohomology is nowhere semisimple. 
  We also generalize the wall-crossing formula for quasimap invariants with weighted markings to the equivariant twisted setting, allowing descendant insertions at light markings. 
  Together with genus-one quantum Lefschetz for quasimaps with light markings, this wall-crossing formula provides the key bridge from the Gromov--Witten theory of the complete intersection to the semisimple equivariant twisted theory of the projective space.
\end{abstract}

\maketitle

\setcounter{section}{0}
\setcounter{tocdepth}{2}

\tableofcontents

\section{Introduction}
\subsection{Virasoro conjecture}

The Virasoro conjecture in Gromov--Witten theory, proposed by Eguchi--Hori--Xiong and later refined in Katz's formulation \cite{EHX97,EJX98}, is one of the basic conjectural structures governing the total descendant potential. For a smooth projective variety, it predicts that the total descendant potential $\mathcal D({\bf t};\hbar)$ is annihilated by a sequence of differential operators $\{L_m\}_{m\geq -1}$ which form a representation of the Virasoro Lie algebra in the sense that \[ [L_m,L_n]=(m-n)L_{m+n}, \qquad m,n\geq -1 . \]

More precisely, if $\cF_g({\bf t})$ is the generating series of genus-$g$ GW
invariants, then $\cD({\bf t};\hbar) = \mathrm{exp}(\sum_{g \ge 0}
\hbar^{2g-2} \cF_g({\bf t}))$. And the conjecture says $L_m\cD({\bf t};\hbar)
= 0$ for all $m\geq -1$. The conjecture can be studied genus-by-genus. Indeed,
we always have the genus expansion
\beq\label{def:Lgm-D}\textstyle
(L_m\cD({\bf t};\hbar))/\cD({\bf t};\hbar)=\sum_{g\geq 0}\hbar^{2g-2}\mathscr L_{g,m}({\bf t}).
\eeq
Then the Virasoro conjecture is equivalent to $\mathscr L_{g,m}({\bf t})=0$ for all $g\geq 0$ and $m\geq -1$.
We refer to $\mathscr{L}_{g,m}({\bf t}) = 0$ as the \textit{genus-$g$ $L_m$-constraint}. 
For fixed $g$ and all $m$, they are called the \textit{genus-$g$ Virasoro constraints}.  

For the simplest case, the Gromov--Witten theory of a point, the Virasoro
conjecture coincides with the celebrated Witten conjecture~\cite{Wit91}, proved
by Kontsevich~\cite{Kont92}. For manifolds with semisimple quantum product, the
lower-genus part of the conjecture was developed through a series of
works~\cite{DZ98,DZ99,Get99,Get04,LT98,Liu01,Liu02}, and the full conjecture was
eventually established by Teleman~\cite{Tel12} based on Givental’s quantization
formalism~\cite{Giv01a,Giv01b}. Without assuming semisimplicity, the Virasoro
conjecture has been verified for Calabi--Yau threefolds (due to the vanishing of
the first Chern class and dimensional reasons), for the Gromov--Witten theory of
smooth curves~\cite{OP06}, and, at genus zero, by Liu and
Tian~\cite{LT98} (see also~\cite{DZ99,Get99,Giv04}). The higher-genus part of
the conjecture remains largely open in the general non-semisimple case. In
particular, proving the genus-one Virasoro constraints is of central importance,
as it provides the first nontrivial check beyond genus zero and often serves as
a crucial step toward understanding the full higher-genus structure.

  Most Fano complete  intersections have non-semisimple quantum
  cohomology, and therefore provide a natural setting for the study of the
  Virasoro conjecture beyond the semisimple case. Let 
  \[
    X=X_{m_1,\ldots,m_r}\subset\mathbb P^{N-1}
  \]
  be a smooth Fano complete intersection of multidegree $(m_1,\ldots,m_r)$.
  After eliminating linear equations, we may assume that $m_i\geq2$ for every $i$.
  We
  call $X$ \emph{exceptional} if it falls in one of the following three cases: a
  quadric hypersurface, an even-dimensional intersection of two quadrics, or a cubic surface.
  The big quantum cohomology of a Fano complete intersection $X$ is generically
  semisimple if and only if it is exceptional~\cite{BM04,CMP10,Hu21,hu2015big}.
  Note that since $m_i \geq 2$, the Fano
  surface cases are all exceptional. Hence we focus on the case $\dim X \geq 3$.

We define \[ H^{\rm amb}(X, \mathbb Q):=\operatorname{Im}\bigl(j^*:H^*(\mathbb P^{N-1}, \mathbb
Q)\to H^*(X, \mathbb Q)\bigr), \]
where $j:X\to\mathbb P^{N-1}$ is the inclusion. We call elements of $H^{\rm
amb}(X, \mathbb Q)$ ambient classes, and we call $H^{\rm amb}(X, \mathbb Q)$ the ambient state space.
The first main result is the following theorem.
\begin{thm}\label{mainthm:genus-one-virasoro}
	Let $X\subset\mathbb P^{N-1}$ be a smooth Fano complete intersection.
	Then the genus-one Virasoro conjecture holds on the ambient state space. 
	Namely, $\mathscr L_{1,m}(\mathbf t)=0$ for $m\geq-1$ whenever $\mathbf t$ is restricted to $H^{\rm amb}(X,\mathbb Q)$.
\end{thm}
\begin{remark}
	There are two related reconstruction approaches to the genus-one Gromov--Witten invariants of Fano complete intersections: the nodal approach of Argüz--Bousseau--Pandharipande--Zvonkine~\cite{ABPZ23} and the monodromy approach of Hu~\cite{Hu22}.
	It is not immediate, however, how to derive the Virasoro constraints from these reconstructions.
\end{remark}
\subsection{The strategy}
Our main strategy is to use quasimap wall-crossing to compare the Gromov--Witten
theory of $X$ and the equivariant twisted theory of $\mathbb P^{N-1}$. Suppose
$X \subset \mathbb P^{N-1}$ is defined as the zero locus of a generic section
$\sigma$ of $\mathbb E$, where
\begin{equation}
  \label{eq:E}
  \mathbb E = \bigoplus_{i = 1}^{r}\mathcal O_{\mathbb P^{N-1}}(m_i), \quad m_i>0.
\end{equation}
The equivariant Gromov--Witten theory of $\mathbb P^{N-1}$ twisted by $\mathbb E$ is semisimple after localization, 
and hence can be studied effectively using Givental's reconstruction formalism \cite{Giv01a,Giv01b,Tel12}. 
In genus zero, its
non-equivariant limit equals the Gromov--Witten theory of $X$. This is referred
to as the Quantum Lefschetz Hyperplane Principle \cite{lee2001quantum,coates2007quantum}. However, that
principle fails in genus one, for the following geometric reason. For any stable
map
\[
  u: C \longrightarrow \mathbb P^{N-1},
\]
$u$ factors through $X$ if and only if $u^*\sigma \in H^0(C, u^*\mathbb E)$
vanishes. In genus zero, we always have $H^1(C, u^*\mathbb E) = 0$. This implies
that, as $C$ and $u$ vary, $H^0(C, u^*\mathbb E)$ forms a vector bundle
$\mathcal E$ on the moduli space of stable maps to $\mathbb P^{N-1}$. And
$u^*\sigma$ forms a section of $\mathcal E$, whose vanishing locus is precisely
the subspace of maps into $X$. However, in genus one, there may be some stable
map contracting some genus-one subcurve $C^\prime \subset C$ to a point. Such a
component is referred to as a ghost component in \cite{vakil-zinger}. In that
case we have $H^1(C, u^*\mathbb E) \neq 0$ so the discussion above fails and the
non-equivariant limit may not exist.

This can be fixed via quasimap wall-crossing. For $\epsilon = 0^+$ stable
quasimaps of with light markings, which will be explained in the next
subsection, the Quantum Lefschetz Hyperplane Principle always holds, as long as
the curve class is nonzero. This theory is related to the Gromov--Witten theory
via the wall-crossing formula. The generating functions of the two theories are
related by an explicit change of coordinates. And the change of coordinates is
invertible when restricted to the ambient state space $H^{\rm amb}(X, \mathbb Q)$.

Thus, at least in principle, the Gromov--Witten theory of $X$ with ambient
insertions can be recovered from the equivariant twisted theory of $\mathbb
P^{N-1}$ by combining wall-crossing, the genus-one quantum Lefschetz principle
for all-light quasimaps, and the non-equivariant limit. We will not try to
compute the resulting invariants by a direct localization calculation. Instead,
we use the diagram below as a way of transporting structural identities, in
particular the genus-one Virasoro constraints.
\begin{equation}
  \label{eq:strategy}
  \begin{tikzcd}
    \boxed{\mathrm{GW}(X)} \ar[rrr, "\epsilon\text{-wall-crossing and}"] & &
    &\ar[lll, "\text{heavy-light wall-crossing}"]
    \boxed{\mathrm{Qmap}^{\mathrm{light}}(X)} \\
    \boxed{\mathrm{GW}(\mathbb P^{N-1})^{\mathrm{tw}}} \ar[rrr,
    "\epsilon\text{-wall-crossing and}"] \ar[u,dashed] & & & \ar[lll,
    "\text{heavy-light wall-crossing}"] \ar[u, "\substack{\text{non-equivariant}
      \\ \text{limit}}"'] \boxed{\mathrm{Qmap}^{\mathrm{light}}(\mathbb
      P^{N-1})^{\mathrm{tw}}}
  \end{tikzcd}
\end{equation}
Here ``GW'' denotes Gromov--Witten theory, ``$\mathrm{Qmap}^{\mathrm{light}}$''
denotes the $\epsilon=0^+$ quasimap theory with all markings light, and ``tw''
denotes the equivariant theory of $\mathbb P^{N-1}$ twisted by the bundle
defining $X$. The lower-left theory is equivariant and semisimple, so Givental's
quantization formalism can be applied to it. We then pass through the lower
horizontal comparison, take the non-equivariant limit on the all-light side, and
use the upper horizontal comparison to obtain the ambient genus-one Virasoro
constraints for $\mathrm{GW}(X)$.

Finally, even when we restrict the $\mathscr L_{g,m}(\mathbf t)$ to the ambient
state space $H^{\rm amb}(X, \mathbb Q)$, its definition still involves some invariant with
from the primitive cohomology $H^{\rm prim}(X,\mathbb Q)$, coming from separating a node.
Fortunately, in the $g = 1$ case, it only involves genus $0$ invariants with
non-ambient insertions, which have been computed in \cite{hu2015big}.

More strategies are introduced to actually prove the theorem using the diagram above.
To further simplify the computation, before wall-crossing,
we use Getzler's relation, together with a reconstruction argument for Fano complete
intersections, to reduce the ambient genus-one Virasoro constraints to finitely
many one-point identities.
We then apply wall-crossing to those identities, reducing most of the problem
to computations in the equivariant twisted theory. The only remaining
difficulty is that even though we only consider ambient insertions for genus-one invariants, their Virasoro constraints will involve genus-zero invariants
with a pair of primitive insertions. By definition, light marking cannot take
primitive insertions. We get around this issue by revoking the explicit
computation of genus-zero invariants from \cite{hu2015big}, which turns out to
perfectly match the expectation from genus-one Virasoro constraints.

\subsection{Wall-crossing formula}
Quasimap wall-crossing provides an effective way of computing lower genus
Gromov-Witten invariants in many situations. In this section, we roughly state the wall-crossing
results and explain why they are helpful for our computation. For the precise
statements and more details, we refer the readers to Section~\ref{sec:wall-crossing}.

The target space
for quasimaps can be a large class of GIT quotients $Y = V \git_{\theta} G$,
where $G$ is a reductive group acting on an affine scheme $V$, and $\theta$ is a
character of $G$, satisfying certain condition. A quasimap
to $Y$ consists of a pointed nodal curve $(C, x_1 ,\ldots, x_n)$ and a morphisms
$u: C \to [V/G]$ to the Artin stack quotient, such that $u^{-1}([V^{\mathrm{un}}(\theta)/G])$ does not
contain any irreducible component of $C$, where $V^{\mathrm{un}}(\theta) \subset
V$ is the $\theta$-unstable locus. The points in $u^{-1}([V^{\mathrm{un}}/G])$ are
called base points of the quasimap.

\begin{example}
  \label{eg:qmap}
    When $Y$ is a complete intersection in projective space $\mathbb
    P^{N-1}$, we write $Y = V \git_{\theta} G$, where $V \subset \mathbb C^{N}$ is
    the affine cone of $Y$, $G = \mathbb C^*$ acts on $\mathbb C^{N}$ by the scalar
    multiplication, and $\theta: G \to \mathbb C^*$ is the identity map.
    Thus the unstable locus $V^{\mathrm{un}}(\theta)$ only consists of the
    origin and $Y = V\git_{\theta}G$. 
    In general, a map $C \to [V/G]$ is equivalent to a principal $G$-bundle $P$ on $C$,
    together with a section of the associated $V$-bundle $P\times_{G}V
    \to C$.
    Thus, in this case,  a quasimap to $Y$
    consists of $(L, \varphi_1 ,\ldots, \varphi_N)$, where $L$
    is a line bundle on $C$, and $\varphi_1 ,\ldots, \varphi_N \in H^0(C, L)$
    satisfy all the defining equations of $Y$. The base points are  the
    common zeros of all those $\varphi_i$'s.

    Let $\mathcal O_{[\mathbb C/ N]}(1)$ be the line bundle on $[\mathbb C^N/
    \mathbb C^*]$ induced by the standard representation of
    $\mathbb C^*$. The
    restriction of $\mathcal O_{[\mathbb C/ N]}(1)$ to  $\mathbb P^{N-1} \subset
    [\mathbb C^N/ \mathbb C^*]$ is equal to $\mathcal O_{\mathbb P^{N-1}}(1)$.
    Then under the above correspondence, we have $u^* \mathcal O_{[\mathbb C^N/
      \mathbb C^*]}(1) = L$, and $\varphi_1 ,\ldots, \varphi_N$ are the pullback
    of the coordinate function on $V = \mathbb C^N$.
  \end{example}

  For any rational number $\epsilon > 0$, Ciocan-Fontanine--Kim--Maulik defined
$\epsilon$-stable quasimaps. They showed that 
the moduli $Q^{\epsilon}_{g,n}(Y, \beta)$ of genus-$g$, $n$-pointed,
$\epsilon$-stable quasimaps of any curve class $\beta$ is a proper
Deligne--Mumford stack with a canonical perfect obstruction theory, inducing a virtual cycle
$[Q^{\epsilon}_{g,n}(Y, \beta)]^{\mathrm{vir}} \in H_*(Q^{\epsilon}_{g,n}(Y,
\beta), \mathbb Q)$. Thus, quasimap invariants are defined in a similar way to the
Gromov--Witten theory.
For $\epsilon \to \infty$, an
$\epsilon$-stable quasimap is the same as a stable map to $Y$, and the
invariants are the Gromov--Witten invariants.

The space of stability conditions $\mathbb Q_{>0}$ has a wall-and-chamber
structure and there are
only finitely many walls for each fixed curve class $\beta$.
Let $\omega_{C}$ be the dualizing sheaf of $C$ and $\omega_{C}^{\log} =
\omega_{C}(x_1 + \cdots + x_n)$, where $x_1 ,\ldots, x_n$ are the marked
points.
We call an irreducible component $C^\prime$ a
rational tail if $\deg(\omega_{C}^{\log}|_{C^\prime}) = -1$, namely,
$C^\prime$ is a smooth rational component with one special point (node or
marking).
Roughly speaking, as $\epsilon$ decreases, rational tails are
contracted to base points and the domain curve becomes simpler.
For $\epsilon \to 0^+$,
the stability amounts to that the domain curve $C$ does not have any
irreducible component $C^\prime$ such that
$\deg(\omega_{C}^{\log}|_{C^\prime}) < 0$ and in case
$\deg(\omega_{C}^{\log}|_{C^\prime}) = 0$, $u|_{C^\prime}$ must not be a
constant regular map into $Y$.

Ciocan-Fontanine--Kim conjectured a wall-crossing formula which says that the
quasimap invariants of different chambers are equivalent \cite{ciocan2017higher, ciocan2020quasimap}. More precisely, those
invariants are computable from each other via an explicit wall-crossing formula once the small
$I$-function is known. The formula was proved in many cases
\cite{ciocan2014wall, cheong2015orbifold, ciocan2017higher, ciocan2020quasimap,
  Clader2024HigherGenus, Wang_2025}
and in full generality
by \cite{zhou2022quasimap}.

Since $\epsilon = 0^+$ stability forbids all rational tails,
 there are no ghost components for genus $1$,
$\epsilon = 0^+$ stable quasimap if there are no markings. Kim--Lho used this
to reprove the mirror theorem for elliptic Gromov--Witten for smooth Calabi-Yau
complete intersections in projective spaces \cite{kim2018mirror,
  zinger2009reduced, popa2013genus, li2009genus}. 
However, markings are crucial for our purpose. When there are many markings, 
the domain curve of an $\epsilon = 0^+$ stable quasimap
may have many rational components. Thus, ghost components may still exist.
The domain curve can be further simplified by introducing the notion of
weighted markings, similar to \cite{hassett2003moduli}. It was introduced by
\cite{ciocan-fontanine2016}. For simplicity, in this paper we only consider an extreme
case where $\epsilon = 0^+$, and the markings are either of weight $1$ (heavy markings), or of weight
$0^+$ (light markings).Heavy markings are just like the usual markings as
before. While light markings can collide with each
other and collide with base points. On the other hand, a smooth rational
component $C^\prime \subset C$ containing with one node of $C$ and only light
markings will not be stable. See Definition~\ref{def:stability-general} for the
full definition. 

Since light markings can be mapped to the unstable locus, the
evaluation maps at the light markings land in $[V/G]$. Thus when defining the
quasimap invariant with light markings, the light markings can only accept
insertions from $H^*([V/G], \mathbb Q)$. That is the reason why we restrict the study to
ambient cohomology.

  When all markings are heavy, we recover the previously introduced $\epsilon = 0^+$ stable quasimaps.
  By heavy-light wall-crossing we mean replacing heavy markings by light
  markings.
  The domain curves get further simplified as
  more markings become light. 
  The wall-crossing formula was proved in the genus $0$ case for target
  spaces with a good torus action in \cite{ciocan-fontanine2016}. The
  last author introduced a master space that geometrically realized the
  heavy-light wall-crossing for weighted FJRW theory
  \cite{zhou2020higher-with-pg-no}. The same technique is later used in the
  appendix of \cite{clader2017higherLG}, which is also written by the last
  author. Then Pinharry used the construction to prove the heavy-light
  wall-crossing for quasimaps in her thesis \cite{pinharry2020weighted}.

  In this paper, we need to generalize all the wall-crossing formulas to
  equivariant version and twisted version.
  The geometric constructions for proving the
  wall-crossing formula, namely the master space construction, remains the same.
  After briefly introducing that construction, 
  we will explain why the localization technique still 
  produces a wall-crossing formula in the equivariant and twisted setting.

  The wall-crossing procedures used in this paper can thus be summarized as
  
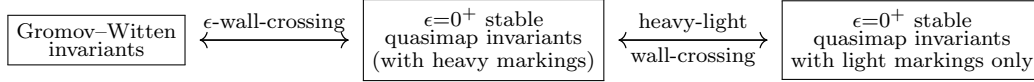
\begin{equation*}
    \begin{tikzcd}
      \boxed{\substack{\text{Gromov--Witten} \\ \text{invariants}}}
      \ar[rr, "\epsilon\text{-wall-crossing}"]
      & & \ar[ll]
      \boxed{
        \substack{
        \epsilon = 0^+\text{ stable}
        \\
        \text{quasimap invariants}
        \\\text{(with heavy markings)}
        }
      } 
      \ar[rr, "\text{heavy-light}"]
      & &
      \ar[ll, "\text{wall-crossing}"]
      \boxed{
        \substack{
        \epsilon = 0^+\text{ stable}
        \\
        \text{quasimap invariants}
        \\\text{with light markings only}
        }
      } 
    \end{tikzcd}
  \end{equation*}

  Combining the two wall-crossings, we get the following theorem, which
  generalizes \cite[Theorem 1.2.1]{pinharry2020weighted} to the twisted and
  equivariant case, and also allowing $\psi$-classes at light markings.
  \begin{theorem}[Theorem~\ref{thm:wc-main}]
    \label{thm:wc-main-intro}
    Let $\mathcal F_g^{\infty,\tw}(\bt)$  be the potential function of
    (twisted, equivariant) Gromov--Witten theory, and let
    $\mathcal F_g^{0+,\tw}(\bt;
    \mathbf s)$ be the  potential function of
    $\epsilon = 0^+$ stable
    quasimap theory with weighted markings, then
    \begin{equation}  \label{eqn:wc-intro}
      \mathcal F_g^{0+,\tw}(\bt;  \mathbf s)=
      \mathcal F_g^{\infty,\tw}(\bt + [\mathbb I(\mathbf s,q,z)-z]_{+}  ) .
    \end{equation}
  \end{theorem}
  Here, the big $\mathbb I$-function is some generating series of genus $0$
  quasimap invariants with parametrized domains.
  We refer the reader to Section~\ref{sec:wall-crossing} for the precise
  definitions of those objects involved in the theorem.

  We will apply the above theorem in two situations.
  First we take $Y = V\git_{\theta} G$ to be the Fano complete intersection $X$
  and the twisting is trivial. This give the upper arrow in \eqref{eq:strategy}.
  Then we take 
  $Y = \mathbb P^{N-1}$ and consider the theory twisted by
  $\mathbb E$. This gives the lower arrow in \eqref{eq:strategy}.
  We will take $\mathbf t = 0$ and get
  \begin{equation}
    \label{eq:wc-t=0}
    \mathcal F_g^{0+,\tw}(0 ;  \mathbf s)=
      \mathcal F_g^{\infty,\tw}([\mathbb I(\mathbf s,q,z)-z]_{+}  ).
  \end{equation}
  This relation is still invertible since
  \[
    [\mathbb I(\mathbf s,q,z)-z]_{+} =
    \mathbf s + O(q),
  \]
  where $q$ is Novikov variable keeping track of the degrees.

  The left hand side of \eqref{eq:wc-t=0} is the potential function of
  $\epsilon = 0^+$ stable quasimaps with light markings only. For those
  invariants we have the Quantum Lefschetz Hyperplane Principle in genus $1$.
  More precisely, we need to first extend the $\mathbb E$ in \eqref{eq:E} to a
  bundle on $[\mathbb
  C^{N}/ \mathbb C^*]$. By abuse of notation, we still denote the unique
  extension by $\mathbb E$.
  Suppose we have quasimap $u : C \to [\mathbb
  C^{N}/ \mathbb C^*]$ given by a line
  bundle $L$ with sections $\varphi_1 ,\ldots, \varphi_N$ (c.f.\
  Example~\ref{eg:qmap}). Then
  \[
    u^*(\mathbb E) = \bigoplus_{i = 1}^{r}L^{\otimes {m_i}}. 
  \]
  If $C$ has genus $1$, the quasimap is $\epsilon = 0^+$ stable with light markings
  only, and the curve class is nonzero, then $C$ is either irreducible of (arithmetic genus $1$), or a loop of
  smooth rational curves. In either case, there are no ghost components. Since $L$ has nonnegative degree on
  each irreducible component, we have 
  \[
    H^1(C,u^*(\mathbb E))= 0.
  \]
  Thus, for that stability, the non-equivariant limit of $\mathbb E$-twisted
  theory of $\mathbb P^{N-1}$ exits and is equal to the untwisted theory of $X$.
  This gives the vertical arrow on the right in \eqref{eq:strategy}.

\subsection{Plan of the paper} 
This paper is organized as follows.
In \S \ref{sec:wall-crossing}, we establish the equivariant twisted wall-crossing formula and specialize it to one-point invariants of Fano complete intersections. 
In \S \ref{sec:virasoro}, we prove the genus-one ambient Virasoro conjecture for Fano complete intersections.
Appendix~\ref{sec:CohFT-comp-tw} contains the twisted-theory computations needed in the proof.

\par
\vspace{1.2\baselineskip}
\noindent\textbf{Acknowledgements.}
Parts of this work were presented at the Tianyuan Mathematical International Exchange Center (2023), Sun Yat-sen University (2023), and KIAS (2024). We thank the organizers for the opportunities to present and discuss this work.

The work of S.~Guo was supported by the National Natural Science Foundation of China under Grant No. 12225101.
The work of Q.~Zhang was supported by the Guangdong Basic and Applied Basic Research Foundation under Grant No. 2025A1515110325.
The work of Y.~Zhou was partially supported by
Shanghai Pilot Program for Basic Research-Fudan Univ.\ 21TQ1400100 (22TQ001),
NSFC for Innovative Research Groups (No. 12121001).
Y.~Zhou is a member of the Key Laboratory of Mathematics for Nonlinear Sciences,
Fudan University. Y.~Zhou would also like to thank the
support of Alibaba Group as a DAMO Academy Young Fellow, and would like to
thank the support of Xiaomi Corporation as a Xiaomi Young Fellow.

\section{Wall-crossing formula for GIT quotients}
\label{sec:wall-crossing}

\subsection{The setup} \label{sec:setup}

  Let $Y = V \git_{\theta} G$ be a
  GIT quotient where
  \begin{enumerate}
  \item
    $G$ is a reductive group acting an affine scheme $V$;
  \item
    $\theta$ is a character of $G$ such that the semistable locus
    $V^{\mathrm{s}}(\theta) = V^{\mathrm{ss}}(\theta) \neq \emptyset$, where
    $V^{\mathrm{ss}}(\theta)$ (resp.~$V^{\mathrm{s}}(\theta)$) is the GIT
    semistable (resp.~stable) locus;
  \item
    $V^{\mathrm{s}}(\theta)$ is smooth and $V$ has at worst lci singularities;
  \item
    for simplicity we further assume that $G$ acts on $V^{\mathrm{s}}(\theta)$
    freely and the affine quotient $V \git_0 G$ is a point.
  \end{enumerate}

  In addition we consider equivariant theory and twisted theory. For this, we consider an
  additional torus $T = (\mathbb C^{*})^N$ acting on $V$ that commutes with the $G$-action, and let
  $\mathbb E$ be an
  $G\times T$-equivariant bundle on $V$. It descends to a $T$-equivariant bundle
  on $[V/G]$, which we also denote by $\mathbb E$ by abuse of notation.
  In order make sure that the twisted theory is well-defined, we consider an
  additional copy of $\mathbb C^*$, denoted by $\mathbb C^*_{\rho}$
  acting only on the fibers of $\mathbb E$ by scalar multiplication.

  We denote the equivariant cohomology of a point by
  \[
    H^*_{T \times \mathbb C^*_{\rho}}(\mathrm{point}, \mathbb Q) = \mathbb Q[\lambda_1 ,\ldots,
    \lambda_N, \rho],
  \]
 where we use the convention that $-\lambda_i$ is the equivariant
 	first Chern class of the $i$-th standard character of
 	$T=(\mathbb C^*)^N$, while $-\rho$ is the equivariant first Chern
 	class of the standard character of $\mathbb C_\rho^*$ acting by
 	scalar multiplication on the fibers of $\mathbb E$.
  We abbreviate $\lambda_1 ,\ldots, \lambda_N$ as $\lambda$, thus
  \[
    \mathbb Q[\lambda] = \mathbb Q[\lambda_1 ,\ldots, \lambda_N], \quad \mathbb
    F = \mathbb Q(\lambda) = \mathbb Q(\lambda_1 ,\ldots, \lambda_N).
  \]
  and so on.
\subsection{Quasimaps invariants with weighted markings}
In this subsection, we recall from \cite{ciocan2014stable, ciocan-fontanine2016}
the definition and some basic facts of quasimap invariants with weighted markings. For simplicity, we
focus on the special cases when the stability parameter $\epsilon = 0^+$, 
and the weights of the markings are either $1$ (heavy markings) or $0^+$ (light markings).

\begin{definition}
  A weighted-pointed curve with $m$ heavy markings
  and $n$ light markings consists of
 \[
   (C, x_1 ,\ldots, x_m, y_1 ,\ldots, y_n),
 \]
 where $C$ is a projective nodal curve with marked points $x_1 ,\ldots, x_m, y_1
 ,\ldots, y_n$, such that
 \begin{itemize}
 \item
   all marked points are in the smooth locus of $C$;
 \item
   $x_1 ,\ldots, x_m$ are pairwise distinct;
 \item
   $y_i \neq x_j$ for any $i$ and $j$.
 \end{itemize}
 The points $x_1 ,\ldots, x_m$ are called heavy markings and the points $y_1
 ,\ldots, y_n$ are called light markings.
\end{definition}

\begin{definition}[\cite{ciocan-fontanine2016}]
  Let $(C, x_1 ,\ldots, x_m, y_1 ,\ldots, y_n)$ be a weighted-pointed curve with
  $m$ heavy markings and $n$ light markings,
  A prestable quasimap from $C$ to $Y$ is a map
  \[
    u: C \longrightarrow [V / G]
  \]
  such that the heavy markings, the nodes and the generic points of each
  irreducible component of $C$ are mapped into $Y =
  [V^{\mathrm{ss}}(\theta)/G]$.
  The subset $u^{-1}([V^{\mathrm{us}}(\theta)/ G])$ is called the base locus and
  consists of the base points.
\end{definition}

More concretely, $u$ is equivalent to a principal $G$-bundle $P$ on $C$,
together with a section $\sigma$ of the associated $V$-bundle $P\times_{G}V
\to C$, and the base locus is $\sigma^{-1}(P
\times_{G}V^{\mathrm{us}}(\theta))$.

Set $\omega_{C}^{\mathrm{log}} = \omega_{C}(x_1 + \cdots + x_m)$.

\begin{definition}[\cite{ciocan-fontanine2016}]
  \label{def:stability-general}
  The prestable quasimap above is said to be $\epsilon = 0^+$-stable if 
  \begin{itemize}
  \item
    $C$ does not have any
    irreducible component $C^\prime$ such that
    $\deg(\omega_{C}^{\log}|_{C^\prime}) < 0$;
  \item
    if $C^\prime \subset C$ is an irreducible component with 
    $\deg(\omega_{C}^{\log}|_{C^\prime}) = 0$, then 
    either $C^\prime$ contains some
    light markings $y_1 ,\ldots, y_n$, or  $u|_{C^\prime}$ is \textit{not} a
    constant regular map into $Y$.
  \end{itemize}
\end{definition}

  Let $Q_{g,m\mid
    n}(Y, \beta)$ be the moduli of genus $g$, $\epsilon=0^+$ stable quasimaps to $Y$, of curve
  class $\beta$, with $m$ heavy markings and $n$ light markings. By
  \cite{ciocan-fontanine2016}, it is a proper
  Deligne--Mumford stack with a virtual fundamental class $[Q_{g,m\mid n}(Y, \beta)]^{\mathrm{vir}}$.
  Here, the curve class $\beta$ of a map $u: C \to [V/G]$ is defined to be 
  the group homomorphism
  \[
    \beta: \mathrm{Pic}([V/G]) \longrightarrow \mathbb Z, \quad \beta(L) := \deg_{C}(f^*L),
  \]
  where $\deg_{C}(\cdot)$ denotes the degree of a line bundle on $C$.

  Since by assumption the heavy markings are not base points, restricting the
  quasimap to a heavy marking $x_i$ defines the evaluation map
  \[
    \mathrm{ev}_{x_i}: Q_{g,m|n}(Y, \beta) \longrightarrow  Y.
  \]
  In contrast, the light markings may be base points.
  Hence restricting the quasimap to a light marking $y_j$ defines the evaluation
  map
  \[
    \widehat {\mathrm{ev}}_{y_j}: Q_{g,m|n}(Y, \beta) \longrightarrow [V/G].
  \]

  Let $\pi:\mathcal C \to Q_{g,m\mid n}(Y, \beta)$ be the universal curve and
  $u: \mathcal C \to [V/G]$ be the universal quasimap.
  \subsection{Heavy-light wall-crossing}

  For $a_1 ,\ldots, a_m, b_1 ,\ldots, b_n \geq 0$, and equivariant cohomology classes
  \[
    \gamma_1 ,\ldots, \gamma_m \in H^*_{T}(Y, \mathbb Q), \quad \eta_1 ,\ldots, \eta_n \in
  H^*_T([V/G], \mathbb Q) = H^*_{G\times T}(V, \mathbb Q),
  \]
  the equivariant twisted quasimap invariants with descendants
  \begin{equation}
    \label{eq:def-of-invariants}
    \langle \gamma_1 \psi^{a_1}, \ldots, \gamma_m \psi^{a_m}\mid
      \eta_1\psi^{b_1} ,\ldots, \eta_n \psi^{b_n}
      \rangle^{0^+, \mathrm{tw}}_{g,m|n, \beta} \in \mathbb Q(\lambda, \rho)
  \end{equation}
  are defined to be
  \[
    \begin{aligned}
      \int_{[Q_{g,m|n}(Y, \beta)]^{\mathrm{vir}}}
      e_{T \times \mathbb C^*_{\rho}}(R\pi_*u^*\mathbb E) \cdot
      \prod_{i=1}^{m}\mathrm{ev}_{x_i}(\gamma_i)\psi_{x_i}^{a_i} \cdot
      \prod_{j=1}^{n}\widehat{\mathrm{ev}}_{y_j}(\eta_j) \psi_{y_j}^{b_j},
    \end{aligned}
  \]
  where $\psi_{x_1} ,\ldots,\psi_{x_m}, \psi_{y_1} ,\ldots, \psi_{y_n}$ are the
  $\psi$-classes at each marking, and $e_{T \times \mathbb C^*_{\rho}}(\cdot)$ denote
  the equivariant Euler class with respect to the $T \times \mathbb
  C^*_{\rho}$-action.

  For $n = 0$, we write
  \[
    \langle \gamma_1 \psi^{a_1}, \ldots, \gamma_m \psi^{a_m}
    \rangle^{0^+, \mathrm{tw}}_{g,m, \beta} 
    =
    \langle \gamma_1 \psi^{a_1}, \ldots, \gamma_m \psi^{a_m}\mid \emptyset
    \rangle^{0^+, \mathrm{tw}}_{g,m|0, \beta}.
  \]
  This is the usual (twisted) $\epsilon = 0^+$ stable quasimaps in \cite{ciocan2014stable}.
  \begin{remark}
    \label{eq:polynomiality-rho-expansion}
    Since $\mathbb C^*_{\rho}$
    only acts on the fibers of $\mathbb E$ and acts nontrivially, the
    equivariant Euler class becomes invertible after we invert $\rho$ and take
    the $\rho^{-1}$-adic completion. Since $Q_{g,m\mid n}(Y, \beta)$ is proper,
    we see that \eqref{eq:def-of-invariants} also lies in
    $\mathbb Q[\lambda][\rho, \rho^{-1}]\!]$.
  \end{remark}

  The wall-crossing formula to be stated will express the difference
  \[
    \begin{aligned}
      \langle \gamma_1 \psi^{a_1}, \ldots,  \gamma_m \psi^{a_m}\mid
      \eta_1\psi^{b_1} &,\ldots, \eta_n \psi^{b_n}
                               \rangle^{0^+, \mathrm{tw}}_{g,m|n, \beta} \\
      - \quad & \langle \gamma_1 \psi^{a_1}, \ldots,  \gamma_m \psi^{a_m}, \overline{\eta}_1\psi^{b_1}\mid
           \eta_2\psi^{b_2} ,\ldots, \eta_n \psi^{b_n}
                                    \rangle^{0^+, \mathrm{tw}}_{g,m + 1|n - 1, \beta}
    \end{aligned}
  \]
  as a summation of some similar invariants, but with fewer markings or of
  curve classes of lower degree. Here $\overline{\eta}_1$ is the restriction of
  $\eta_1$ to $Y$.

  To state the wall-crossing formula, 
  following \cite{ciocan-fontanine2016}. 
  Consider the so-called loop space 
  \begin{equation}
    \label{eq:loop-space}
    Q(\mathbb P^1, Y, \beta)
  \end{equation}
  of quasimaps from a fixed $\mathbb P^1$ to $Y$ of curve class $\beta$. 
  It carries a $\mathbb C^*$-action induced by the action on $\mathbb P^1$:
  \[
    t\cdot [x, y] = [tx, y], \quad \text{for } t \in \mathbb C^*,
  \]
  where $x, y$ are the homogeneous coordinates on $\mathbb P^1$.
  We denote this copy of $\mathbb C^*$ by $\mathbb C^*_z$.
  We use the convention that the equivariant first Chern class of its standard representation is $-z$.
  We refer to the
  point $[0, 1]$ as $\infty$ so that the  tangent space to $\infty \in
  \mathbb P^1$ is the standard representation of $\mathbb C^*_z$.
  Let $F_{\beta} \subset Q(\mathbb P^1, Y, \beta)$ be the fixed locus where $0$
  is the only base point.

  As in \cite{ciocan-fontanine2016}, $Q(\mathbb P^1, Y, \beta)$ has a canonical
  $T\times \mathbb C^*_z$-equivariant perfect obstruction theory, inducing the virtual
  fundamental class $[F_{\beta}]^{\mathrm{vir}}$ and the virtual
  normal bundle $N^{\mathrm{vir}}_{F_\beta/ Q(\mathbb P^1, Y, \beta)}$ of $F_{\beta}$. Let $\pi_{\beta}: F_{\beta} \times
  \mathbb P^1 \to F_{\beta}$ be the universal curve over $F_{\beta}$ and
  $u_{\beta}: F_{\beta} \times \mathbb P^1 \to [V/G]$ be the universal map.
  Consider the twisted virtual cycle
  \begin{equation}
    \label{eq:twisted-vir}
    [F_{\beta}]^{\mathrm{vir}, \mathrm{tw}} =
    e_{T\times \mathbb C^*_{\rho} \times \mathbb C^*_{z}}(R\pi_{\beta*}u_{\beta}^*\mathbb
    E)\cap [F_{\beta}]^{\mathrm{vir}}.
  \end{equation}
  In order to define the big $\mathbb I$-function, which will appear in the
  wall-crossing formula, 
  we think of $\infty$ as a heavy marking and $0$ as the accumulation of
  several light markings. Thus we consider
   two evaluation maps
  \[
    \widehat{\mathrm{ev}}_{0}: F_{\beta} \longrightarrow [V/G] \quad \text{and} \quad 
    \mathrm{ev}_{\infty}: F_{\beta} \longrightarrow Y,
  \]
  defined by evaluating the quasimap at $0$ and $\infty$ respectively.

  To state and prove the wall-crossing formula, first define the $\mathbb Q[\lambda]$-linear operator
  \[
    \Phi_{\beta}: H_{T}^*([V/G], \mathbb Q) \longrightarrow H^*_T(Y, \mathbb Q)
    [z, z^{-1}]\!] [\rho, \rho^{-1}]\!]
  \]
  by
  \[
    \Phi_{\beta}(\gamma) =   \frac{1}{e_{T\times \mathbb C^*_{\rho}}(\mathbb E)}(\mathrm{ev}_{\infty})_*
    \Big(\frac{\widehat{\mathrm{ev}}_0^*\gamma  
       \cap [F_{\beta}]^{\mathrm{vir, tw}}}{e_{T\times \mathbb
        C^*_{\rho} \times \mathbb C^*_z}(N^{\mathrm{vir}}_{F_{\beta}}
      )}\Big).
  \]
  Note that in particular $\Phi_{0}$ is the Kirwan map.

  \begin{example}
    \label{eg:proj-space-first}
    Let $T = (\mathbb C^*)^{N}$ act on $V =\mathbb C^{N}$ by
    \[
      (t_1 ,\ldots, t_N) \cdot [x_1 ,\ldots, x_N] = [t_1x_1 ,\ldots, t_Nx_N],
      \quad \text{ for } (t_1 ,\ldots, t_N) \in T, (x_1 ,\ldots, x_N) \in V.
    \]
    Consider the twisted theory of $\mathbb P^{N-1}$, twisted by $\mathbb E =
    \bigoplus^r_{i=1} \mathcal O_{[\mathbb C^{N} / \mathbb C^*]}(m_i)$, $m_i >
    0$.
    The $T$-action on $V$ induces an action on $[\mathbb C^{N}/ \mathbb C^*]$
    and lifts to a canonical linearization on $\mathbb E$.

    Define the hyperplane classes
    \[
      \mathrm{h} = c_1(\mathcal O_{[\mathbb C^{N} / \mathbb C^*]}(1)) \in H^*_T ([\mathbb
      C^{N}/\mathbb C^*], \mathbb Q), \quad \mathrm{H} = \mathrm{h}|_{\mathbb P^{N-1}} = c_1(\mathcal
      O_{\mathbb P^{N-1}}(1))
      \in H^*_T(\mathbb P^{N-1}, \mathbb Q).
    \]
    and $z$ be the equivariant 
    parameter of the $\mathbb C^*$ acting on the domain $\mathbb P^1$.
    We identify the curve class $\beta$ with its degree $d$.
    Then for $a
    \geq 0$, 
    \[
      \Phi_{d}(\mathrm{h}^a) = (\mathrm{H} + dz)^a \frac{\prod_{i = 1}^r \prod_{k = 1}^{m_id}(m_i \mathrm{H}
        + kz - \rho)}
      {\prod_{k=1}^d\prod_{i = 1}^N (\mathrm{H} +  kz - \lambda_i)}.
    \]
    If we set $\rho = 0$, we get the formula of the operator for the un-twisted theory of the
    complete intersection defined by a generic section of $\mathbb E$.
  \end{example}
  \begin{remark}
    We have abused the notation here. For simplicity of the explanation, let us
    consider the case without the $T$ action on $Y$.
    To form the pullback $\widehat{\mathrm{ev}}_0^*:
    H^*([V/G]) \to H^*_{\mathbb C^*}(F_{\beta}) \cong H^*(F_{\beta})[z,
    z^{-1}]$, we should view $\widehat{\mathrm{ev}}_0$ as a map
    \[
      [F_{\beta} / \mathbb C^*] \cong F_{\beta} \times B \mathbb C^*
      \longrightarrow [V / G].
    \]
    In other words, the map $F_{\beta} \to [V/G]$ is not only $\mathbb
    C^*$-invariant but has extra $\mathbb
    C^*$-equivariant structure.
    In concrete terms, recall that a map $F_{\beta} \to [V/G]$ is a principal
    $G$-bundle $P \to F_{\beta}$ together with a $G$-equivariant map $P \to V$.
    Then the equivariant structure is a $\mathbb C^*$-action on $P$ such that
    $P\to F_{\beta}$ and $P \to V$ are both invariant.
    This $\mathbb C^*$-equivariant structure is determined as the pullback of the
    $\mathbb C^*$-equivariant structure of the universal quasimap over $F_{\beta}$
    \[
      u_{\beta}: F_{\beta} \times \mathbb P^1 \longrightarrow [V/G].
    \]
    The key is that the generic point of $\mathbb P^1$ is assumed to the mapped into
    the stable locus $Y \subset [V/G]$, which is a scheme. Hence, for $t\in \mathbb C^*$, the
    isomorphism between $t^*u_{\beta}$ and $u_{\beta}$ must be uniquely determined. 
    This is implicitly used in the computation in \cite{ciocan-fontanine2016}.
    In \cite{ciocan-fontanine2016},  it was
    shown that for a line bundle $L$ on $[V/G]$,
    \[
      \widehat{\mathrm{ev}}_0^*(c_1(L)) = {\mathrm{ev}}_\infty^*(c_1(L)) + \beta(L) z.
    \]
    Indeed, the pullback of $L$ to $F_{\beta} \times \mathbb P^1$ is family of
    equivariant line bundles on $\mathbb P^1$ of degree $\beta(L)$ and the
    $\mathbb C^*$ action on the fiber over $\infty$ is trivial. This determines
    the $\mathbb C^*$ equivariant structure on the fibers of $F_{\beta} \times
    \mathbb P^1 \to F_{\beta}$. Thus we get the
    above formula.
    In particular, for complete intersections in projective space, 
     $\widehat{\mathrm{ev}}_0^*(h) = \mathrm{ev}_\infty^*H + dz$.
  \end{remark}

  The operators $\Phi_{\beta}$, combined with some truncation operator, will give the so-called correction terms in the
  wall-crossing formula. Recall $\mathbb F = \mathbb Q(\lambda)$. For any
  $\mathbb F$-algebra $A$, define the formal residue operator 
  \[
    \mathrm{Res}: A[z, z^{-1}]\!][\rho, \rho^{-1}]\!] \longrightarrow A[\rho,
    \rho^{-1}]\!], 
  \]
  by taking the terms containing $z^{-1}$. Namely, 
  \[
    \quad \mathrm{Res}(\sum_{i,j} a_{ij}z^{i}\rho^{j}) = \sum_{j} 
    a_{(-1)j} \rho^j.
  \]
  And we define the truncation operator
  \begin{equation}
    \label{eq:truncation-def}
    [\ \cdot\ ]_{+}: A[z, z^{-1}]\!][\rho, \rho^{-1}]\!] \longrightarrow A[z][\rho, \rho^{-1}]\!]
  \end{equation}
  to be the truncation keeping non-negative powers of $z$. Namely, 
  \[
    \Big[\sum_{i,j}a_{ij}z^{i}\rho^{j}\Big]_{+} = \sum_{j} \sum_{i\geq 0}a_{ij}z^{i}\rho^{j}.
  \]
  Then for $a \in A$, $f(z, \rho) \in A[z, z^{-1} ]\!][\rho, \rho^{-1} ]\!]$,
  \begin{equation}
   \label{eq:res-and-plus} 
    \mathrm{Res} \left(\frac{f(z,\rho)}{z - a}\right) = ([ f(z,\rho) ]_{+})|_{z = a}.
  \end{equation}

  For $f \in A \otimes_{\mathbb F} \mathbb F(z, \rho)$, we define
  $\mathrm{Res}(f)$ and $[f]_+$
  via the sequence of inclusions 
  \[
    A \otimes_{\mathbb F} \mathbb F(z, \rho) \subset
    A \otimes_{\mathbb F}\mathbb F(z) [\rho,\rho^{-1}]\!]
    \subset A[z, z^{-1}]\!][\rho,
    \rho^{-1}]\!].
  \]
  In other words, we first expand it as a formal Laurent series in $\rho^{-1}$,
  then we expand the coefficients of each $\rho^{j}$ as a formal Laurent series in
  $z^{-1}$, and finally we collect all the terms containing $z^{-1}$. 
  \begin{example}
    We take $A = \mathbb F$ and consider $f = \frac{1}{z + \rho}$. There are two
    natural ways to expand
    \[
      f = \sum_{i = 0}^{\infty} (-1)^{i} z^{i} \rho^{-i-1} \in \mathbb
      F[z, z^{-1}]\!][\rho, \rho^{-1}]\!] \quad \text{or}
      \quad f =
      \sum_{j = 0}^{\infty} (-1)^{j} z^{-j - 1} \rho^{j} \in \mathbb
      F[\rho, \rho^{-1}]\!][z, z^{-1}]\!].
    \]
    In the two expansions, the coefficients of the $z^{-1}$ terms are $0$ and $1$,
    respectively. Note that they are not equal. In our definition, we use
    the first one and get $\mathrm{Res}(f) = 0$.
  \end{example}
  \begin{lemma}
    \label{lem:formal-residue}
    Let $A$ be an $\mathbb F$-algebra.
    \begin{enumerate}
    \item 
      The map $\mathrm{Res}:  A[z,z^{-1} ]\!] [\rho,\rho^{-1} ]\!] \to A
      [\rho,\rho^{-1} ]\!]$ is linear over $A[\rho
      , \rho^{-1}]\!]$. 
    \item 
      If $f \in A[z] [\rho,\rho^{-1} ]\!]$, then $\mathrm{Res}(f) = 0$.
    \item
      If $B$ is another $\mathbb F$-algebra and $\int: A \to B$ is any $\mathbb
      F$-linear map, then the induced diagram commutes
      \[
        \begin{tikzcd}
          A[z, z^{-1} ]\!] [\rho ,\rho^{-1}]\!] \ar[r, "{\int}"] \ar[d,
          "{\mathrm{Res}}"] & B[z, z^{-1}]\!] [\rho ,\rho^{-1}]\!] \ar[d,
          "{\mathrm{Res}}"]\\
          A[\rho ,\rho^{-1}] \ar[r, "{\int}"] & B[\rho,\rho^{-1} ] 
        \end{tikzcd}
      \]
    \end{enumerate}
  \end{lemma}
  \begin{proof}
    These are straightforward from definition.
  \end{proof}

  In the following Proposition and its proof, we will take 
  $A$ to be the $T$-equivariant cohomology of various spaces.
  \begin{proposition}[Basic wall-crossing formula]
    \label{prop:basic-wc}
    For any $\gamma_1 ,\ldots, \gamma_m \in H^*_T(Y, \mathbb Q)$,  $\eta_1 ,\ldots,
    \eta_n \in H^*_T([V/G], \mathbb Q)$, suppose $2 g - 2 + m \geq 0$ and $(2 g
    - 2 + m, n, \beta) \neq (0,0,0)$, we have
    \begin{equation}
      \label{eq:basic-wc}
      \begin{aligned}
 &   \langle \gamma_1 \psi^{a_1}, \ldots,  \gamma_m \psi^{a_m}\mid
        \eta_1\psi^{b_1} ,\ldots, \eta_n \psi^{b_n}
                                 \rangle^{0^+, \mathrm{tw}}_{g,m|n, \beta} \\
        = & \sum_{\beta^\prime, J}
            \langle \gamma_1 \psi^{a_1}, \ldots,  \gamma_m \psi^{a_m},
            [
            (-z)^{b_J} z^{1 - |J|}
            \textstyle \Phi_{\beta^\prime}(\prod_{j\in
            J}\eta_j)
            ]_{+} \big|_{z = -\psi}
           \mid \{ \eta_j \psi^{b_j}\}_{j\not \in J}
            \rangle^{0^+, \mathrm{tw}}_{g,m + 1|n - |J|, \beta - \beta^\prime}
      \end{aligned}
      \end{equation}
    where $\overline{\eta}_1$ is the restriction of
  $\eta_1$ to $Y$, and the summation is over all pairs $(\beta^\prime, J)$ such that
    \begin{itemize}
    \item
      both $\beta^\prime, \beta - \beta^\prime$ are effective curve classes;
    \item
      $1 \in J \subset \{1 ,\ldots, n\}$.
    \end{itemize}
    For such a $J$, we set $b_J = \sum_{j\in J}b_j$, and $|J|$
    is the cardinality of $J$;
  \end{proposition}
  \begin{proof}
    The proof is parallel to many previous works and we only sketch it. We first
    explain the differences in equivariant twisted theory, and then we outline 
    the computation.

    The key is to construct a master space $\mathcal M$ which is proper with a
    perfect obstruction theory, 
    and has
    a $\mathbb C_{z}^*$-action with fixed loci
    \[
      F_{\star}, \quad \text{where }\star = 0, \infty, \text{ or $(\beta^\prime, J)$ as in
        the proposition.} 
    \]
    The master space is constructed by merging the stability conditions on both
    sides of the wall and introducing an additional
    \[
      v \in T_{y_1}C \cup \{\infty\},
    \]
    where $T_{y_1}C$ is the tangent space to the domain curve at the marking $y_1$,
    which naturally form a line bundle on the moduli.
    It was first introduced in \cite{zhou2020higher-with-pg-no} and we refer the
    reader to \cite{pinharry2020weighted} for
    the version in the context of quasimaps.

    Before describing each $F_{\star}$ in details, we prove a general
    localization formula.
    The $\mathbb C^*_z = \mathbb C^*$ acts on $\mathcal M$ by scaling $v$. Hence it
    commutes with the $T$-action on the target space and we get a $T\times
    \mathbb C_z^*$-action on $\mathcal M$.
    Moreover, $\mathcal M$ has an equivariant perfect obstruction theory, inducing virtual
    cycles $[\mathcal M]^{\mathrm{vir}}$ and $[F_{\star}]^{\mathrm{vir}}$. We
    also define the twisted virtual cycles $[\mathcal M]^{\mathrm{vir, tw}},
    [F_{\star}]^{\mathrm{vir, tw}}$ similar to \eqref{eq:twisted-vir}.
    Now take any $\alpha \in H^*_{T\times \mathbb C_{z}^*}(\mathcal M, \mathbb Q)$, and we
    claim that
    \begin{equation}
      \label{eq:localization}
      \int_{[\mathcal M]^{\mathrm{vir, tw}}} \alpha =
      \sum_{\star = 0, \infty, (\beta^\prime, J)}
      \int_{[F_{\star}]^{\mathrm{vir, tw}}} \frac{\alpha|_{F_{\star}}}{e_{T\times \mathbb
          C_z^*}(N^{\mathrm{vir}}_{F_{\star}/\mathcal M})}
      \quad \text{ in } \mathbb F(z)[\rho, \rho^{-1}]\!].
    \end{equation}
    To see this, we first write
    \begin{equation}
      \label{eq:twisted-cycle-decomposition}
      [\mathcal M]^{\mathrm{vir, tw}} = \sum_{i = -\infty}^{K} e_i \rho^{-i} [\mathcal M]^{\mathrm{vir}},
      \text{ for some }K \in \mathbb Z \text{ and } e_i \in H^*_{T\times \mathbb C_z^*}(\mathcal M, \mathbb Q).
    \end{equation}
    Then applying virtual localization to each $\int_{[\mathcal
      M]^{\mathrm{vir}}}e_i\alpha$, with respect to the $T \times \mathbb C_z^*$-action,
    we get
    \begin{equation}
      \label{eq:localization-1}
      \begin{aligned}
        \int_{[\mathcal M]^{\mathrm{vir, tw}}} \alpha
        = &  \sum_{i = -\infty}^K \big(\int_{[\mathcal M]^{\mathrm{vir}}} e_i \alpha) \rho^{i} \\
        = & \sum_{i = -\infty}^K \big(\sum_{\star = 0, \infty, (\beta^\prime,
            J)} \sum_{j} \int_{[F_{\star, j}]^{\mathrm{vir}}}
            \frac{(e_i\alpha)|_{F_{\star, j}}}{e_{T\times \mathbb C_z^*}(N_{F_{\star,
            j}/\mathcal M}^{\mathrm{vir}})}) \rho^{i},
      \end{aligned}
    \end{equation}
    where the $F_{\star, j}$'s are the connected components of the $T$-fixed locus
    $(F_{\star})^{T}$. Note that 
    \[
      [F_{\star}]^{\mathrm{vir, tw}} = \sum_{i = -\infty}^{K} (e_i|_{F_\star})
      \rho^{-i}[F_{\star}]^{\mathrm{vir}} \quad \text{and} \quad 
     e_{T\times \mathbb C_z^*}(N_{F_{\star,
          j}/ \mathcal M}^{\mathrm{vir}})= e_{T\times
        \mathbb C_z^*}(N_{F_{\star, j}/F_{\star}}^{\mathrm{vir}}) \Big(e_{T\times
        \mathbb C_z^*}(N^{\mathrm{vir}}_{F_\star/\mathcal M})\Big) \Big|_{F_{\star, j}} .
    \]
    Applying the virtual localization formula to the $T\times \mathbb
    C^*_z$-action on each $F_\star$, where the $\mathbb C^*_z$ factor acts trivially, we get
    \begin{equation}
      \label{eq:localization-2}
      \int_{[F_\star]^{\mathrm{vir}}} \frac{(e_i\alpha)|_{F_{\star}}}{e_{T\times \mathbb
          C_z^*}(N^{\mathrm{vir}}_{F_\star /\mathcal M})} = \sum_{j} \int_{[F_{\star, j}]^{\mathrm{vir}}}
      \frac{(e_i\alpha)|_{F_{\star, j}}}{e_{T\times \mathbb
          C_z^*}(N_{F_{\star, j}/ \mathcal M}^{\mathrm{vir}})}.
    \end{equation}
    Combining \eqref{eq:localization-1} and \eqref{eq:localization-2}, we get \eqref{eq:localization}.

    Now in \eqref{eq:localization} we take 
    \[
      \alpha =
      \prod_{i=1}^{m} \mathrm{ev}^*_{x_i}\gamma_i \psi_{x_i}^{a_i} \cdot 
      \prod_{j = 1}^{n}
      \widehat{\mathrm{ev}}^*_{y_j}\eta_j \psi_{y_j}^{b_j}.
    \]
    By \eqref{eq:twisted-cycle-decomposition}, the left hand side of
    \eqref{eq:localization} lies in $\mathbb F[z][\rho, \rho^{-1}]\!]$. Applying
    the formal residue operator to both sides of \eqref{eq:localization}, by
    Lemma~\ref{lem:formal-residue} we get
    \begin{equation}
      \label{eq:sum-of-res}
      0 = \sum_{\star = 0, \infty, (\beta^\prime, J)}
      \int_{[F_{\star}]^{\mathrm{vir, tw}}} \mathrm{Res}\Big(\frac{\alpha|_{F_{\star}}}{e_{T\times \mathbb
          C^*_{z}}(N^{\mathrm{vir}}_{F_{\star}/\mathcal M})}\Big) \in \mathbb F[\rho, \rho^{-1}]\!].
    \end{equation}
    It remains to compute the residue on each $F_{\star}$.
    Recall that our convention is that the standard representation of $\mathbb C^*$ has
    first Chern class $- z$. This is the same as \cite{zhou2022quasimap}, but different from
    \cite{zhou2020higher-with-pg-no,pinharry2020weighted}.

    The computation is almost parallel to those in
    \cite{zhou2020higher-with-pg-no, zhou2022quasimap, pinharry2020weighted}. We
    sketch the key steps of the computation and highlight the differences.
    \begin{itemize}
    \item
      $F_{0} \cong Q_{g, m + 1 \mid n - 1}(Y, \beta)$ is the locus where $v =
      0$, and we have
      \[
        \frac{1}{e_{T\times \mathbb C_z^*}(N^{\mathrm{vir}}_{F_{0}/ \mathcal M})} = \frac{1}{- z - \psi_{y_1}}.
      \]
      And $\alpha$ restricts to
      \[
        \prod_{i=1}^{m} \mathrm{ev}^*_{x_i}\gamma_i \psi_{x_i}^{a_i} \cdot 
        \prod_{j = 1}^{n}
        \widehat{\mathrm{ev}}^*_{y_j}\eta_j \psi_{y_j}^{b_j}.
      \]
      Note that $y_1$ is a heavy marking here for $Q_{g, m + 1 \mid n - 1}(Y,
      \beta)$ and we use $\widehat{\mathrm{ev}}_{y_1}$ to denote the
      evaluation map at $y_1$ composed with the inclusion map $X \to [V/G]$.

      The equivariant intergration is $\mathbb F$-linear, taking $A =
      H^*_{T}(F_0, \mathbb Q)$, we have $\alpha|_{F_0} \in A$. By
      Lemma~\ref{lem:formal-residue} and \eqref{eq:res-and-plus},
      \[
        \begin{aligned}
          \mathrm{Res}\Big(\int_{[F_0]^{\mathrm{vir, tw}}} \frac{\alpha|_{F_0}}{-z - \psi_{y_1}}\Big)
          = & \int_{[F_0]^{\mathrm{vir, tw}}}
              \mathrm{Res}\big(\frac{\alpha|_{F_0}}{-z - \psi_{y_1}}\big) \\
          = & - \int_{[F_0]^{\mathrm{vir, tw}}} \alpha|_{F_0} \\
          = & - \int_{[Q_{g, m + 1 \mid n - 1}(Y, \beta)]^{\mathrm{vir, tw}}} 
              \prod_{i=1}^{m} \mathrm{ev}^*_{x_i}\gamma_i \psi_{x_i}^{a_i} \cdot 
        \prod_{j = 1}^{n}
        \widehat{\mathrm{ev}}^*_{y_j}\eta_j \psi_{y_j}^{b_j}.
        \end{aligned}
      \]
      This gives $(\beta^\prime, J) = (0, \{1\})$ term, up to a sign, in the wall-crossing
      formula \eqref{eq:basic-wc}.
    \item
      $F_{\infty} \cong Q_{g, m \mid n }(Y, \beta)$ is the locus where $v =
      \infty$, and we have
      \[
        \frac{1}{e_{T\times \mathbb C^*_z}(N^{\mathrm{vir}}_{F_{\infty}/
            \mathcal M})} = \frac{1}{ z + \psi_{y_1}}.
      \]
      And $\alpha$ restricts to
      \[
        \prod_{i=1}^{m} \mathrm{ev}^*_{x_i}\gamma_i \psi_{x_i}^{a_i} \cdot 
        \prod_{j = 1}^{n}
        \widehat{\mathrm{ev}}^*_{y_j}\eta_j \psi_{y_j}^{b_j}.
      \]
      And similarly
      \[
        \begin{aligned}
          \mathrm{Res}\Big(\int_{[F_\infty]^{\mathrm{vir, tw}}} \frac{\alpha|_{F_\infty}}{z + \psi_{y_1}}\Big)
          = & \int_{[Q_{g, m \mid n }(Y, \beta)]^{\mathrm{vir, tw}}}
              \prod_{i=1}^{m} \mathrm{ev}^*_{x_i}\gamma_i \psi_{x_i}^{a_i} \cdot 
              \prod_{j = 1}^{n}
              \widehat{\mathrm{ev}}^*_{y_j}\eta_j \psi_{y_j}^{b_j},
        \end{aligned}
      \]
      giving the term on the left hand side in the wall-crossing
      formula \eqref{eq:basic-wc}.
    \item
      $F_{\beta^\prime, J} \cong Q_{g, m + 1\mid n - |J|}(Y, \beta
      - \beta^\prime) \times_{Y} F_{\beta^{\prime}}$ is the locus where
      $v \neq 0, \infty$ and $C$ has an irreducible component $C^{\prime}
      \cong \mathbb P^1$, such that under this isomorphism, $\infty$ is a node;
      all $y_j$ for $j\in J$ are located at $0$; $C^\prime$ contains no other
      markings or nodes of $C$; the restriction of the quasimap to $C^{\prime}$
      is in $F_{\beta^{\prime}}$, the distinguished fixed locus in the
      loop space $Q(\mathbb P^1, Y, \beta)$ defined in \eqref{eq:loop-space}.

      For $Q_{g, m + 1\mid n - |J|}(Y, \beta - \beta^\prime)$, the $n - |J|$
      light markings are those $y_j$ for $j\not \in J$.
      The $(m+1)$-th marking is a new heavy marking, and we denote it by
      $\bullet$.
      The fiber product is
      formed using the maps $\mathrm{ev}_{\bullet}$ and $\mathrm{ev}_{\infty}$. 

      Then $\alpha$ restricts to
      \[
        \textstyle 
        \prod_{i=1}^{m} \mathrm{ev}^*_{x_i}\gamma_i \psi_{x_i}^{a_i}
        \cdot
        \prod_{j\not \in J} \widehat{\mathrm{ev}}^*_{y_j} \eta_j \psi_{y_j}^{b_j}
        \boxtimes
        \mathrm{ev}_0^*(\prod_{j\in J} (-z)^{b_j}\eta_j).
      \]
      And
      \[
        \frac{1}{e_{\mathbb C^*}(N^{\mathrm{vir}}_{F_{\beta^\prime, J}/ \mathcal
          M})} =
        \frac{1}{- z - \psi_{\bullet}}
        \boxtimes \frac{z^{1-|J|}}{e_{T\times \mathbb
            C_z^*}(N^{\mathrm{vir}}_{F_{\beta^\prime}/ Q(\mathbb P^1, Y, \beta^{\prime})})},
      \]
      The term $- z - \psi_{\bullet}$ comes from the space of infinitesimal smoothing of the
      node on $C^\prime$, and $z^{1-{|J|}}$ coming from the deformation of $y_j$
      away from $y_1$, for $j\in J\setminus \{1\}$.
      Thus
      \[
        \begin{aligned}
          & \mathrm{Res}\Big(\int_{[F_{\beta^{\prime}, J}]^{\mathrm{vir, tw}}}
          \frac{\alpha|_{F_{\beta^{\prime}, J}}}{- z - \psi_{\bullet}}
          \boxtimes \frac{z^{1-|J|}}{e_{T\times \mathbb
            C_z^*}(N^{\mathrm{vir}}_{F_{\beta^\prime}/ Q(\mathbb P^1, Y, \beta^{\prime})})} \Big) \\
          = & - \int_{[Q_{g, m + 1\mid n - |J|}(Y, \beta
              - \beta^\prime)]^{\mathrm{vir, tw}}}
              \prod_{i=1}^{m} \mathrm{ev}^*_{x_i}\gamma_i \psi_{x_i}^{a_i}
              \cdot
              \prod_{j\not \in J} \widehat{\mathrm{ev}}^*_{y_j} \eta_j
              \psi_{y_j}^{b_j}  \\
          &\cdot
            \mathrm{Res}\Big(
            \frac{\prod_{j\in J}
            (-z)^{b_j}\mathrm{ev}_{\bullet}^*\Phi_{\beta^{\prime}} \big(\eta_j\big)
            }{z^{|J| - 1} (z + \psi_{\bullet})}
            \Big)
        \end{aligned}
      \]
      This gives $(\beta^\prime, J) \neq (0, \{1\})$ term on the right hand side, up to a sign, in the wall-crossing
      formula \eqref{eq:basic-wc}.
    \end{itemize}

    Putting all those results into \eqref{eq:sum-of-res}, we get the desired
    wall-crossing formula.
  \end{proof}

  We now put the wall-crossing formula in a more succinct form. We extend
  $\Phi_{\beta}$ linearly in $z, \rho$ and $q$, and set 
  \[
    \Phi =  \sum_{\beta} \Phi_\beta q^{\beta}: H_{T}^*([V/G], \mathbb Q)[z, z^{-1}]\!] [\rho, \rho^{-1}]\!] \longrightarrow H^*_T(Y, \mathbb Q)
    [z, z^{-1}]\!] [\rho, \rho^{-1}]\!]
  \]
  and $q$ is the
  Novikov variable. Finally set
  \[
    \Phi^{+} : H_{T}^*([V/G], \mathbb Q)[z, z^{-1}]\!] [\rho, \rho^{-1}]\!] \longrightarrow H^*_T(Y, \mathbb Q)
    [z] [\rho, \rho^{-1}]\!]
  \]
  to be the truncation of $\Phi$ keeping
  the $z^{\geq 0}$-terms, defined in \eqref{eq:truncation-def}.

  For $\mathbf t_1(z) ,\ldots, \mathbf t_m(z) \in H^*_T(Y, \mathbb Q)[z]$, 
  $\mathbf s_1(z) ,\ldots, \mathbf s_n(z) \in H^*_T([V/G], \mathbb Q)[z]$,
  we set
  \[
    \begin{aligned}
      \langle \mathbf t_1(\psi) ,\ldots, \mathbf t_m(\psi) \mid
      \mathbf s_1(\psi) ,\ldots, &  \mathbf s_n(\psi)
                                \rangle^{0^+, \mathrm{tw}}_{g,m|n} \\ = \sum_{\beta} q^\beta
                              & \langle \mathbf t_1(\psi) ,\ldots, \mathbf t_m(\psi) \mid
                                \mathbf s_1(\psi) ,\ldots,  \mathbf s_n(\psi)
                                \rangle^{0^+, \mathrm{tw}}_{g,m|n, \beta} \quad .
    \end{aligned}
  \]
  \begin{corollary}
    For $\mathbf s_1(z) ,\ldots, \mathbf s_n(z) \in H^*_T([V/G], \mathbb Q)[z]$,
    we have
    \[
      \langle - |  \mathbf{s}_1(\psi),\cdots, \mathbf{s}_n(\psi) \rangle^{0^+, \mathrm{tw}}_{g,m|n}
      = \sum_{k = 1}^{n}\sum_{\vec B} \frac{1}{k!} \langle - ,
      \Phi^+(\tfrac{\mathbf{s}_{B_1}(-z) 
      }{z^{|B_1|-1}})\big |_{z = -\psi},\cdots,
      \Phi^+(\tfrac{\mathbf{s}_{B_k}(-z) }{z^{|B_k|-1}})\big |_{z = - \psi}
      \rangle^{0+, \mathrm{tw}}_{g,m+k
        \mid 0}\quad ,
    \]
    where ``$-$'' denotes arbitrary insertions with descendants at the first $m$
    markings, the summation is over all partitions $\vec B = (B_1 ,\ldots,
    B_k)$ of the set $\{1 ,\ldots, n\}$, $|B_i| > 0$ is the size the $B_i$ and $\mathbf
    s_{B_i} (-z)= \prod_{j\in B_{i}} \mathbf s_j(-z)$ for $i = 1 ,\ldots, k$.
  \end{corollary}
  \begin{proof}
    This follows immediately from inductively applying Proposition~\ref{prop:basic-wc}.
  \end{proof}

  We set
  \[
    \mathbb I(\mathbf s,q,z) = z\Phi(e^{\mathbf s/z}) = \frac{z}{e_{T\times \mathbb C^*_{\rho}}(\mathbb E)} \sum_{\beta} q^{\beta}
    (\mathrm{ev}_{\infty})_*
    \big(
    \frac{\exp(\frac{\widehat{\mathrm{ev}}_{0}^*(\mathbf s)}{z}) \cap
      [F_\beta]^{\mathrm{vir, tw}}}{e_{\mathbb
        C^*}(N^{\mathrm{vir}}_{F_{\beta/ Q(\mathbb P^1, Y, \beta)}})}
    \big). 
  \]
  This is well-defined as a formal function near the origin of the infinite
  dimensional space $H^*_T([V/G], \mathbb Q)[z]$.

  \begin{proposition}
    \label{prop:wall-crossing-big-I-form}
    As a formal functions in $\mathbf s(z)$ defined near the origin of
    $H^*_T([V/G], \mathbb Q)[z]$, we have 
    \[
      \sum_{n\geq 0} \frac{1}{n!}	\langle  - |  \mathbf s(\psi) ,\ldots, \mathbf s(\psi) \rangle^{0^+,\mathrm{tw}}_{g,m|n}
      = \sum_{n\geq 0} \frac{1}{n!} \langle - ,  \mu(\mathbf s,-\psi) ,\ldots,  \mu(\mathbf s,-\psi)\rangle^{0^+, \mathrm{tw}}_{g,m+n}
    \]
    where $ \mu(\mathbf s,z):=[ \mathbb I(\mathbf s(-z),q,z)- \mathbb I({\mathbf 0},q,z) ]_{+}  =
    \mathbf s(z)+O(q)$.
  \end{proposition}
  \begin{proof}
    The $n = 0$ terms are trivially equal. 
    Applying the previous Corollary and using the linearity of $\Phi$, we have
    \begin{align*}
      & \sum_{n > 0} \frac{1}{n!} \langle - | \mathbf s(\psi) ,\ldots, \mathbf s(\psi)
        \rangle^{0^+, \mathrm{tw}}_{g,m|n}\\
      = & \sum_{n > 0} \frac{1}{n!} \sum_{k>0} \frac{1}{k!} \sum_{
          b_1+\cdots+b_k=n,b_i>0} \binom{n}{b_1 \cdots b_k}
          \Big\langle - ,
          \Phi^+(\tfrac{\mathbf s(-z)^{b_1} }{z^{b_1-1}})\big |_{z = - \psi},\cdots, \Phi^+(\tfrac{\mathbf s(-z)^{b_k} }{z^{b_k-1}})\big |_{z = - \psi}
          \Big\rangle^{0^+, \mathrm{tw}}_{g,m+k}\\
      = & \sum_{k>0} \frac{1}{k!} \sum_{b_1,\ldots,b_k>0} \prod_{i=1}^k \frac{1}{b_i!} \Big\langle - ,
          \Phi^+(\tfrac{\mathbf s(-z)^{b_1} }{z^{b_1-1}})\big |_{z = - \psi},\cdots, \Phi^+(\tfrac{\mathbf s(-z)^{b_k} }{z^{b_k-1}})\big |_{z = - \psi}
          \Big\rangle^{0^+, \mathrm{tw}}_{g,m+k}\\
      = & \sum_{k>0} \frac{1}{k!} \Big\langle - , \Phi^+(z e^{\frac{\mathbf s(-z)}{z}
          }-z),\cdots, \Phi^+(z e^{\frac{\mathbf s(-z)}{z} }-z) \Big\rangle^{0^+, \mathrm{tw}}_{g,m+k}.
    \end{align*}
    Finally, $\Phi^+(z e^{\frac{\mathbf s(-z)}{z}}-z) = [z\Phi(e^{\frac{\mathbf s(-z)}{z}})
    -z\Phi(\mathbf 1)]_{+} =  \mu(\mathbf s,z)$.
  \end{proof}

  For $\mathbf t(z) \in H^*_T(Y, \mathbb Q)[z]$, $\mathbf s(z)  \in H^*_T([V/G],
  \mathbb Q)[z]$, define the genus-$g$ potential functions 
	\[
  	\mathcal F_g^{0+,\tw}(\bt(\psi);  \mathbf s(\psi)) := \sum_{m,n \geq 0}
    \frac{1}{m!n!}\langle \mathbf t(\psi) ,\ldots, \mathbf t(\psi)    \mid
    \mathbf s(\psi) ,\ldots,     \mathbf s(\psi)
    \rangle^{0^+, \mathrm{tw}}_{g,m|n} \quad ,
  \]
  and
  \[
    \mathcal F_g^{\infty,\tw}(\bt(\psi)) := \sum_{m\geq 0} \frac{1}{m!}\langle
    \mathbf t(\psi)   ,\ldots, \mathbf t(\psi) 
    \rangle^{\infty, \mathrm{tw}}_{g,m} \quad .
  \]
  The latter is the generating series of usual twisted Gromov-Witten invariants,
  namely, $\epsilon = +\infty$ stable quasimap invariants with only heavy
  markings. In both summations, the unstable terms are taken to be $0$.

  \begin{theorem}
    \label{thm:wc-main}
    For $g\geq 1$, 
    \begin{equation}  \label{eqn:wc}
      \mathcal F_g^{0+,\tw}(\bt (\psi);  \mathbf s(\psi))=
      \mathcal F_g^{\infty,\tw}(\bt(\psi) + [\mathbb I(\mathbf s(-z),q,z)-z]_{+} \big|_{z = -\psi}  ) .
    \end{equation} 
    For $g=0$, the same formula holds modulo constant and linear terms in
    $\mathbf t$.
  \end{theorem}
  \begin{proof}
    This follows immediately from
    Proposition~\ref{prop:wall-crossing-big-I-form}, together with
    the $\epsilon$ wall-crossing formula \cite{ciocan2020quasimap,
      zhou2022quasimap}. Note that the $I$-functions in this paper defers by a
    factor of $z$ from those in \cite{ciocan2020quasimap,
      zhou2022quasimap}.
    Namely, the small $I$-function in \cite{ciocan2020quasimap,
      zhou2022quasimap} equals to $\mathbb I(0,q,z)/z$ of this paper.
\end{proof}

\subsection{One-point wall-crossing formula for Fano complete intersections}
\label{subsec:one-point-wc}
We specialize the wall-crossing formula to invariants with one light marking.
Let
\[
  X=X_{m_1,\ldots,m_r}\subset \mathbb P^{N-1}
\]
be a Fano complete intersection with $m_i\geq2$, cut out by a section of $\mathbb E = \bigoplus_{i = 1}^{r}\mathcal O_{\mathbb P^{N-1}}(m_i)$.  
Its Fano index is $c_X=N-\sum_{j=1}^r m_j$.
Let $\lh$ denote the hyperplane class on $[\mathbb C^N/\mathbb C^*]$,
and write $\hh$ for its restriction to $\mathbb P^{N-1}$ and further to $X$.

We restrict to inputs
$\mathbf s(z)=s=\sum_i s^i\phi_i$, where $\phi_i=\lh^i$.
The $I$-function of $X$ is
\[
\mathbb I^{X}(s,q,z)
=z\sum_{d\ge 0}q^d
\exp\!\left(\sum_i\frac{s^i}{z}(\hh+dz)^i\right)
\frac{\prod_{j=1}^r\prod_{l=1}^{m_jd}(m_j\hh+lz)} {\prod_{l=1}^d(\hh+lz)^N},
\]
while the equivariant $\mathbb E$-twisted $I$-function of
$\mathbb P^{N-1}$ is
\[
\mathbb I^{\mathbb P ,\tw}(s,q,z)
=z\sum_{d\ge 0}q^d
\exp\!\left(\sum_i\frac{s^i}{z}(\hh+dz)^i\right)
\frac{\prod_{j=1}^r\prod_{l=1}^{m_jd}(m_j\hh+lz-\rho)} {\prod_{l=1}^d \prod_i(\hh+lz-\lambda_i) },
\]
where $\rho$ and $\lambda_i$ are the equivariant parameters introduced in \S \ref{sec:setup}.

We apply Theorem~\ref{thm:wc-main} to the quasimap theory of $X$ and
to the equivariant $\mathbb E$-twisted quasimap theory of
$\mathbb P^{N-1}$.
We write
$$
\bigl\langle - \bigr\rangle^{\epsilon,X}_{g,n}
\qquad\text{and}\qquad
\bigl\langle - \bigr\rangle^{\epsilon,\mathbb P,\tw}_{g,n}
$$
for the corresponding correlators, respectively.
For $\epsilon=\infty$, we define the $t$-shifted correlator by
$$
\bigl\langle -
\bigr\rangle^{\infty,\bullet,t}_{g,n} :=  \sum_{m\geq 0} \frac{1}{m!}\bigl\langle -, t,\cdots,t\bigr\rangle^{\infty,\bullet}_{g,n+m},
$$
where $\bullet=X$ or $\mathbb P,\tw$.
For $g\geq1$, differentiating equation~\eqref{eqn:wc} with respect to
$s^k$ and then setting $\bt=\mathbf s=0$ gives
\begin{equation}
	\label{eqn:onepointwc}
	\bigl\langle \phi_k \bigr\rangle^{\epsilon=0+,X}_{g,|1}
	=
	\bigl\langle \widetilde \phi_k(\psi) \bigr\rangle^{\infty,X,\tau}_{g,1},\qquad 
	\bigl\langle \phi_k \bigr\rangle^{\epsilon=0+,\mathbb P, \tw}_{g,|1}
	=
	\bigl\langle \widetilde \phi^{\tw}_k(\psi) \bigr\rangle^{\infty,\mathbb P, \tw,\tau}_{g,1}. 
\end{equation}
Here $\tau=[\mathbb I^{\bullet}(0,q,z)-z]_+$ is the mirror map
of the corresponding theory. The two mirror maps agree and are given by
\beq\label{eqn:tau}
\tau=
\begin{cases}
	\mathbf m!\,q\cdot\mathbf 1, & c_X=1,\\
	0, & c_X>1,
\end{cases}
\eeq
where $\mathbf m!:=\prod_{j=1}^r(m_j!)$.
The transformed insertions appearing in~\eqref{eqn:onepointwc} are
defined by
\[
\widetilde \phi_k(z)
:=
\left[
\left.\frac{\partial}{\partial s^k}\mathbb I^X(s,q,-z)\right|_{s=0}
\right]_+,
\qquad
\widetilde \phi^{\tw}_k(z)
:=
\left[
\left.\frac{\partial}{\partial s^k}
\mathbb I^{\mathbb P,\tw}(s,q,-z)\right|_{s=0}
\right]_+,
\]
where $[\ \cdot\ ]_+$ denotes the projection onto nonnegative powers of
$z$.
Since differentiation with respect to $s^k$ multiplies the degree-$d$
summand of $\mathbb I^\bullet(s,q,z)$ by
$z^{-1}(\hh+dz)^k$, we have
\[
\left.\frac{\partial}{\partial s^k}\mathbb I^{\bullet}(s,q,z)\right|_{s=0}
=
z^{-1}(\hh+zD)^k \mathbb I^{\bullet}(0,q,z)=:I^{\bullet,(k)}(q,z),
\]
where $D=q\frac{\pd}{\pd q}$.
Consequently,
\[
\widetilde \phi_k(z)=\bigl[I^{X,(k)}(q,-z)\bigr]_+,
	\qquad 
\widetilde \phi^{\tw}_k(z)=\bigl[I^{\mathbb P,\tw,(k)}(q,-z)\bigr]_+.
\]

We now express the right-hand side of \eqref{eqn:onepointwc} in terms of ancestor invariants. 
The descendant--ancestor relation~\cite{KM98,Giv01a} gives, for any $t$,
\begin{equation}
	\label{eqn:onepointwc-ancestor}
	\bigl\langle \widetilde \phi_k(\psi) \bigr\rangle^{\infty,X,t}_{g,1}
	=
	\bigl\langle
	\bigl[S^t(\bar\psi)\widetilde \phi_k(\bar\psi)\bigr]_+
	\bigr\rangle^{\infty,X,t}_{g,1},
\end{equation}
where $S^t(z)\in\End(H^{*}(X,\mathbb Q))[[z^{-1}]]$ is the $S$-matrix at $t$ and $\bar\psi$ is the ancestor psi-class. 
At the mirror point $\tau=\tau(q)$ given by \eqref{eqn:tau}, the mirror theorem~\cite{Giv96} gives
$$
\mathbb I^{X}(0,q,-z)=-zS^{\tau}(z)^{-1}{\bf 1}.
$$
The equivariant twisted theory satisfies the analogous identity with
$S^{t}(z)$ replaced by $S^{\tw,t}(z)$.
The two $S$-matrices satisfy the quantum differential equations
\beq\label{eqn:QDE}
z {d} S^{t}(z)=dt*_{t}S^{t}(z),\qquad
z {d} S^{\tw,t}(z)=dt*_{\tw,t}S^{\tw,t}(z).
\eeq
Here $*_t$ and $*_{\tw,t}$ denote the quantum products at $t$ in the untwisted theory and the equivariant twisted theory, respectively.
At $t=\tau$, we abbreviate the corresponding quantum products to $*$
and $*_{\tw}$, respectively.

\begin{theorem}
	For $g\geq 1$ and $k\geq 0$, the one-point wall-crossing formulas take the following ancestor form:
	\begin{equation}
		\label{eqn:onepointwc-final}
		\bigl\langle \lh^k \bigr\rangle^{\epsilon=0+,X}_{g,|1}
		=
		\bigl\langle T_k(\bar\psi) \bigr\rangle^{\infty,X,\tau}_{g,1},\qquad
		\bigl\langle \lh^k \bigr\rangle^{\epsilon=0+,\mathbb P,\tw}_{g,|1}
		=
		\bigl\langle T^{\tw}_k(\bar\psi) \bigr\rangle^{\infty,\mathbb P,\tw,\tau}_{g,1}.
	\end{equation}
	Set $\widetilde\hh=\hh+D(\tau)=\hh+\tau$.
	The insertions $T_k$ and $T_k^{\tw}$ are determined by the recursions
	\begin{equation}
		\label{eqn:Tk-recursion}
		T_0(z)=1,\qquad
		T_k(z)=\widetilde\hh*T_{k-1}(z)-zD (T_{k-1}(z)),
	\end{equation}
	and
	\begin{equation}
		\label{eqn:Tk-recursiontw}
		T_0^{\tw}(z)=1,\qquad
		T_k^{\tw}(z)=\widetilde\hh*_{\tw}T^{\tw}_{k-1}(z)-zD (T^{\tw}_{k-1}(z)).
	\end{equation}
\end{theorem}
\begin{proof}
	We prove the untwisted case; the equivariant twisted case is analogous.
	The mirror theorem gives
	$I^{X,(0)}(q,-z)=S^\tau(z)^{-1}{\bf1}$.
	Assume inductively that
	\[
	I^{X,(k-1)}(q,-z)=S^\tau(z)^{-1}T_{k-1}(z),
	\]
	for some $k\geq 1$. Then
	\[
	I^{X,(k)}(q,-z)
	=(\hh-zD)I^{X,(k-1)}(q,-z)
	=(\hh-zD)S^\tau(z)^{-1}T_{k-1}(z).
	\]
	The divisor equation and the quantum differential equation imply that,
	for any $q$-dependent cohomology-valued series $v$,
	\[
	(\hh-zD)\bigl(S^\tau(z)^{-1}v\bigr)
	=
	S^\tau(z)^{-1}\bigl(\widetilde\hh*v-zD(v)\bigr).
	\]
	Applying this identity with $v=T_{k-1}(z)$ gives
	\[
	I^{X,(k)}(q,-z)
	=
	S^\tau(z)^{-1}
	\bigl(\widetilde\hh*T_{k-1}(z)-zD (T_{k-1}(z))\bigr)
	=S^\tau(z)^{-1}T_{k}(z).
	\]
	Thus, by induction, $I^{X,(k)}(q,-z)=S^\tau(z)^{-1}T_{k}(z)$ for all $k\geq 0$.
	Therefore,
	$$
	 \bigl[S^{\tau}(\bar\psi)\widetilde \phi_k(\bar\psi)\bigr]_{+}
	=\bigl[S^{\tau}(\bar\psi)I^{X,(k)}(q,-\bar\psi)\bigr]_+ 
	=T_{k}(\bar\psi).
	$$
	Combining this identity with equations~\eqref{eqn:onepointwc} and \eqref{eqn:onepointwc-ancestor} proves the theorem.
\end{proof}

\section{Genus-one Virasoro conjecture for Fano complete intersections}
\label{sec:virasoro}
In this section, we prove Theorem~\ref{mainthm:genus-one-virasoro}: the genus-one ambient Virasoro conjecture holds for the Fano complete intersections in projective spaces.

\subsection{Notation and basic properties}
We retain the notation $X$, $\hh$, and $\fanoind $ from
\S \ref{subsec:one-point-wc}.
Then $\dim X=N-1-r$ and $c_1(X)=\fanoind \hh$.
Throughout this section, ordinary cohomology is taken with rational coefficients, 
which we suppress from the notation. Whenever quantum products or generating series are involved, 
we implicitly extend scalars to the Novikov ring and use the same notation for the resulting modules.
Let $j:X\hookrightarrow \mathbb P^{N-1}$ be the inclusion. 
We denote by $H^{\rm amb}(X):={\rm Im}(j^*)$ the ambient state space, 
and by $H^{\rm pri}(X)$ its orthogonal complement with respect to the Poincar\'e
pairing $\eta$. Thus $H^*(X)=H^{\rm amb}(X)\oplus H^{\rm pri}(X)$. By abuse of
notation, $j^*(\hh^a)$ is still denoted by $\hh^a$, and
$H^{\rm amb}(X)=\sspan\{\phi_a=\hh^a\}_{a=0}^{\dim X}$. We fix a homogeneous
basis $\{\gamma_i\}$ of $H^{\rm pri}(X)$ and denote by $\{\gamma^i\}$ its dual
basis, i.e. $\eta(\gamma^i,\gamma_j)=\delta_{ij}$.
For an endomorphism $A$ of $H^*(X)$, we use the supertrace
$$
\str(A):=\sum_{a=0}^{\dim X}\eta(\phi_a,A\phi^a)
+\sum_i\eta(\gamma_i,A\gamma^i),
$$
where $\{\phi^a\}$ is the $\eta$-dual basis of $\{\phi_a\}$.

Let
$t=\sum_{a=0}^{\dim X}t^a\phi_a$ be an arbitrary point of
$H^{\rm amb}(X)$, and write
\[
\corr{-}_{g,n}(t)
:=
\bigl\langle-\bigr\rangle^{\infty,X,t}_{g,n}.
\]
We also write $F_g(t):=\corr{\ }_{g,0}(t)$ for the
genus-$g$ potential on the small ambient phase space. For $2g-2+n>0$, the
divisor equation gives
\beq\label{eqn:divisor}
\frac{\pd}{\pd t^1}\corr{-}_{g,n}
=q\frac{\pd}{\pd q}\corr{-}_{g,n}.
\eeq
The quantum product $*_{t}$ on $H^*(X)$, with
$t\in H^{\rm amb}(X)$, is defined by
$$
\eta(v_1*_{t}v_2,v_3)=\corr{v_1,v_2,v_3}_{0,3},
\qquad v_1,v_2,v_3\in H^*(X).
$$
The ambient subspace is closed under $*_{t}$, and we use the same
notation for the induced quantum product on $H^{\rm amb}(X)$.

In what follows, we identify the vector field $\pd/\pd t^a$ with the class
$\phi_a$. Introduce
$$
E=\sum_{a=0}^{\dim X}(1-a) t^a\phi_a+c_1(X),\qquad
\mu=1-\frac{\dim X}{2}-\nabla E,
$$
where $\nabla$ is the Levi--Civita connection of $\eta$. Then
$E|_{ t=0}=\fanoind\,\hh$ and
$\mu(\phi_a)=(a-\frac{\dim X}{2})\phi_a$. 
The operator $\mu$ is skew-adjoint:
$\eta(v_1,\mu(v_2))=-\eta(\mu(v_1),v_2)$.
More generally, for $\alpha\in H^{p,q}(X)$, one has
$\mu(\alpha)=(p-\frac{\dim X}{2})\alpha$.
In particular, $\mu$ preserves both $H^{\rm amb}(X)$ and
$H^{\rm pri}(X)$.

The dimension axiom gives the following homogeneity relation for $F_0$:
\beq\label{eqn:hom-Fg}
E(F_0( t))=(3-\dim X)F_0( t)
+\tfrac{1}{2}\eta( t,c_1(X)\cup t).
\eeq
Taking three derivatives in the ambient directions gives
\beq\label{eqn:hom-quantum-product}
E\big(\corr{\phi_a,\phi_b,\phi_c}_{0,3}\big)
=(a+b+c-\dim X)\corr{\phi_a,\phi_b,\phi_c}_{0,3}.
\eeq
We shall use the following reformulation.

\begin{lemma}\label{lem:hom-quantum-product}
	Let $\mathfrak D_E:=E+\mu+\frac{\dim X}{2}$, where $E$ acts on the
	coefficient functions of cohomology-valued vector fields. 
	Then for any ambient vector fields $v_1,v_2$, possibly depending on $ t$, one has
	$$
	\mathfrak D_E(v_1*_{ t}v_2)
	=
	\mathfrak D_E(v_1)*_{ t}v_2
	+
	v_1*_{ t}\mathfrak D_E(v_2).
	$$
\end{lemma}
\begin{proof}
	Set $\deg=\mu+\frac{\dim X}{2}$. Then $\deg(\phi_a)=a\, \phi_a$ and
	$\eta(v_1,\deg(v_2))+\eta(\deg(v_1),v_2)=\dim X\,\eta(v_1,v_2)$. Pairing
	with an arbitrary basis vector and using~\eqref{eqn:hom-quantum-product}, we
	obtain, for basis vector fields,
	$$
	E(\phi_a*_{ t}\phi_b)
	=
	\deg(\phi_a)*_{ t}\phi_b
	+\phi_a*_{ t}\deg(\phi_b)
	-\deg(\phi_a*_{ t}\phi_b).
	$$
	This is the desired identity for basis vector fields. The general case
	follows from linearity and the Leibniz rule for the coefficient functions.
\end{proof}

\subsection{Reconstruction of genus-one Virasoro constraints}
This subsection reduces the genus-one Virasoro constraints for Fano
complete intersections to a finite collection of genus-one one-point
identities at the origin of the small phase space. 

Firstly, by the genus-one topological recursion relation, the genus
one $L_{k-1}$-constraint can be reduced to the following small-phase-space
formulation:
\beq\label{eqn:vira-genus-one}
E^{k}(F_1( t))
=-\frac{1}{4}\sum_{i+j=k-1}\str\big(\E^i\mu \E^j\mu \big)
-\frac{1}{24}\sum_{i+j=k-1}\str\big((\E^i\mu\E^j{\bf 1})*_{ t}\big),
\eeq
where $E^k=E*_{t}\cdots*_{t}E$ denotes the $k$-fold quantum
power for $k\geq1$, and we set
$E^0={\bf 1}=\pd/\pd t^0$.
When applied to a function, $E^k$ denotes differentiation along this
vector field. We also write $\E=E*_{t}$ for quantum multiplication by
$E$ and set $\E^0={\rm Id}$.
\begin{remark}
	The small-phase-space formulation~\eqref{eqn:vira-genus-one} of the genus-one
	Virasoro constraints was obtained by Liu~\cite[Theorem~4.4]{Liu01}.
	The convention used here follows the formulation in~\cite{GZ25}.
	We also refer to~\cite{GZ25} for an extension of this formula to
	cohomological field theories with non-flat units.
\end{remark}

Secondly, the genus-one Virasoro conjecture can be further simplified by using the Getzler relation~\cite{Get97}. Introduce
$$
\cG_k( t)
:=E^k(F_1( t))+\frac{1}{24}\sum_{a+b=k-1}\str\big((\E^a\mu\E^b{\bf 1})*_{ t}\big)
+\frac{1}{4} \sum_{a+b=k-1}\str\big(\E^a\mu \E^b\mu\big).
$$
Then the genus-one $L_{k-1}$-constraint is equivalent to $\cG_k({ t})=0$.
Liu proved that\footnote{In~\cite{Liu01}, this identity was proved under
	the assumption $H^{\rm odd}(X)=0$, which does not hold for
	odd-dimensional complete intersections. For each odd cohomology class
	$v$, one may introduce a formal odd variable $s_v$ such that $s_vv$ is
	even. With this formal modification, Liu's argument applies without
	change when $H^{\rm odd}(X)\ne0$.}
for $m,n\geq0$ with $m+n>0$,
\beq\label{eqn:Get-Vir}
E^n(\cG_{m+1}( t))=\frac{m(m+1)}{m+n}\cG_{m+n}( t).
\eeq
In particular, by taking $m=1,n=k-1$ in equation~\eqref{eqn:Get-Vir}, one has for $k\geq 1$,
\beq\label{eqn:hk-h2}
\cG_{k}( t)=\frac{k}{2}\, E^{k-1}(\cG_2( t)).
\eeq
Therefore, together with the known cases $\cG_0( t)=\cG_1( t)=0$, the genus-one Virasoro conjecture holds true if and only if $\cG_2( t)\equiv 0$, i.e. the $L_1$-constraint holds.

Lastly, note that $\cG_2( t)$ is a power series in the variables $\{ t^i\}$. In general, to verify the single equation $\cG_2( t)\equiv 0$, one still needs to compute infinitely many invariants. For Fano complete intersections, we prove that this condition can be further reduced to finitely many initial values at the origin of the small phase space. Consequently, the genus-one Virasoro constraints are reduced to a finite collection of numerical identities.

\begin{proposition}\label{prop:virasoro-reconstruction}
	For the Fano complete intersections $X\subset \mathbb P^{N-1}$, the genus-one Virasoro constraints 
	$\cG_k( t)=0$, $k\geq 0$,
	hold for arbitrary $ t\in H^{\rm amb}(X)$ if and only if 
	$$
	\cG_k(0)=0,\qquad  {\text{for }}\quad k=2,\cdots,\dim X+1.
	$$ 
\end{proposition}
\begin{proof}
	Throughout the proof, all computations are restricted to the ambient space.
	We first derive a consequence of equations~\eqref{eqn:Get-Vir}
	and~\eqref{eqn:hk-h2}. By applying $E^n$ to the $k=m+1$ case of
	equation~\eqref{eqn:hk-h2}, we have
	$$
	E^{n}(\cG_{m+1}( t))
	=
	\frac{m+1}{2}\,E^n\big(E^{m}(\cG_2( t))\big).
	$$
	On the other hand, by equation~\eqref{eqn:Get-Vir} and by the $k=m+n$ case of
	equation~\eqref{eqn:hk-h2}, we also have
	$$
	E^{n}(\cG_{m+1}( t))
	=
	\frac{m(m+1)}{m+n}\cG_{m+n}( t)
	=
	\frac{m(m+1)}{2}E^{m+n-1}(\cG_2( t)).
	$$
	Comparing these two equations, we obtain
	\beq
  \label{eqn:Get-h2}
	E^n\big(E^{m}(\cG_2)\big)
	=
	m\,E^{m+n-1}(\cG_2),
	\qquad m,n\geq 0,\quad m+n>0.
	\eeq
	
	Now we consider the Taylor expansion of $\cG_2( t)$ at the origin. We assign
	${\rm ord}( t^a)=1$ and write
	$$
	\cG_2( t)=\sum_{l\geq 0}\cG_{2,l}( t),
	$$
	where $\cG_{2,l}$ is the homogeneous part of order $l$. Similarly, we write
	$$
	E^k=\sum_{l\geq 0}E^k_l,
	$$
	where $E^k_l$ has coefficients, with respect to the basis
	$\{\pd_{ t^i}\}_{i=0}^{\dim X}$, homogeneous of order $l$ in $ t$.
	Taking the order-$l$ part of~\eqref{eqn:Get-h2} acting on $\cG_{2,l+2}$ gives
	$$
	E^n_0\big(E^m_0(\cG_{2,l+2})\big)
	=
	m\sum_{i=0}^{l+1}E^{m+n-1}_i(\cG_{2,l+1-i})
	-
	\sum_{i=1}^{l+2}\sum_{a+b=i}
	E^n_a\big(E^m_b(\cG_{2,l+2-i})\big).
	$$
	The right-hand side only involves lower order terms $\cG_{2,0},\ldots,\cG_{2,l+1}$. 
	Since $E_0^k=\fanoind^k\hh^{*k}$ and $\hh^{*k}=\hh^k+O(q)$, the vectors
	$\{E_0^k\}_{k=0}^{\dim X}$ generate $H^{\rm amb}(X)$ over the Novikov ring.
	Thus the above identity reconstructs $\cG_{2,l+2}$ from lower order terms.
	Hence $\cG_2( t)$ is determined by its constant and linear parts. 
	In particular,
	$$
	\cG_{2,0}=0,\quad \cG_{2,1}=0   \quad\Longrightarrow\quad	\cG_2( t)\equiv 0.
	$$
	
	It remains to translate these two initial conditions into the stated numerical conditions.
	The condition $\cG_{2,0}=0$ is exactly $\cG_2(0)=0$. 
	Moreover, by taking the constant part of equation~\eqref{eqn:hk-h2}, we have
	$$
	\cG_{k}(0)=\frac{k}{2}\,E^{k-1}_0(\cG_{2,1}), \qquad k\geq 1.
	$$
	Using again that $\{E^k_0\}_{k=0}^{\dim X}$ generates $H^{\rm amb}(X)$, and the fact $\cG_{1}(0)=0$,
	we see that $\cG_{2,1}=0$ if and only if
	$$
	\cG_{k}(0)=0,\qquad k=2,\ldots,\dim X+1.
	$$
	Therefore $\cG_2( t)\equiv 0$ is equivalent to the above finite collection of
	initial values. By~\eqref{eqn:hk-h2}, this is further equivalent to the genus
	one Virasoro constraints for arbitrary $ t\in H^{\rm amb}(X)$. The proposition
	is proved.
\end{proof}
We next express these initial values in terms of genus-one one-point
invariants, in a form suitable for comparison with the twisted theory.
The initial values above are taken at $t=0$, whereas the comparison
formula below is written at the mirror point $t=\tau$. Since $\tau$
lies in the unit direction, the string equation identifies the quantum
products and the primary correlators at these two points. We therefore
use $*$ and $\langle-\rangle_{g,n}$ for their common values.
\begin{corollary}\label{cor:vira-initial}
	For the Fano complete intersections $X\subset \mathbb P^{N-1}$, the genus-one Virasoro constraints 
	hold for arbitrary $ t\in H^{\rm amb}(X)$ if and only if, for $k=2,\cdots,\dim X+1$,
	\begin{align}\label{eqn:genus-one-formula}
		\<\widetilde{\hh}^{*k}\>_{1,1}
		=&-\frac{1}{24\,\fanoind}\sum_{a+b=k-1}\str\big(\widetilde{\hh}^{*a}*\mu(\widetilde{\hh}^{*b})* \big)
		-\frac{1}{4\, \fanoind} \sum_{a+b=k-1}\str\big(\widetilde{\hh}^{*a}* \,\mu\, \widetilde{\hh}^{*b}*\,\mu\big).
	\end{align}
	Recall here that $\widetilde{\hh}=\hh+\tau=\hh+\delta_{\fanoind,1}{\bf m}!\, q\cdot {\bf 1}$.
\end{corollary}

\begin{proof}
	Following the notation used in the proof of
	Proposition~\ref{prop:virasoro-reconstruction}, we denote by $E^k_0$ the
	constant part of $E^k$ with respect to the coordinate $ t$. 
	Then the equation $\cG_k(0)=0$ is precisely
	\begin{align*}
		\<E^k_0\>_{1,1}
		=&-\frac{1}{24}\sum_{a+b=k-1}
		\str\big(E^a_0*\mu(E^b_0)* \big) 
		-\frac{1}{4} \sum_{a+b=k-1}
		\str\big(E^a_0* \,\mu\, E^b_0*\,\mu\big).
	\end{align*}
	Since $E|_{ t=0}=\fanoind \hh$, we have
	$E^k_0=\fanoind^{k}\hh^{*k}$. 
	Since $\tau=\delta_{\fanoind,1}{\bf m}!\,q\cdot{\bf 1}$, we identify
	$\tau$ with its scalar coefficient
	$\delta_{\fanoind,1}{\bf m}!\,q$ whenever it appears as a factor in a
	quantum product.
	Hence
	\beq\label{eqn:HinE}
	\widetilde \hh^{*k}
	=\sum_{l=0}^{k}\binom{k}{l}\tau^{k-l}\cdot\hh^{*l}
	=\frac{1}{\fanoind^{k}}\sum_{l=0}^{k}\binom{k}{l}
	\big(\fanoind \tau\big)^{k-l}\cdot E^{l}_0 \, .
	\eeq
	Now we compute
	\[
	\sum_{a+b=k-1}\widetilde\hh^{*a}\otimes \widetilde\hh^{*b}
	=\frac{1}{\fanoind^{k-1}}\sum_{a+b=k-1}\sum_{i=0}^a\sum_{j=0}^{b}
	\binom{a}{i}\binom{b}{j}
	(\fanoind \tau)^{k-1-i-j}E^{i}_0\otimes E^j_{0}.
	\]
	Exchanging the order of summation, for fixed $i,j$, we use the identity
	\[
	\sum_{a+b=k-1}\binom{a}{i}\binom{b}{j}
	=\binom{k}{i+j+1}.
	\]
	Therefore,
	\beq\label{eqn:HHinEE}
	\sum_{a+b=k-1}\widetilde\hh^{*a}\otimes \widetilde\hh^{*b}
	=\frac{1}{\fanoind^{k-1}}\sum_{l=0}^{k}\binom{k}{l}
	(\fanoind \tau)^{k-l}\sum_{i+j=l-1}E^{i}_0\otimes E^j_{0}.
	\eeq
	By equations~\eqref{eqn:HinE} and~\eqref{eqn:HHinEE},
	equation~\eqref{eqn:genus-one-formula} is equivalent to
	\[
	\sum_{l=0}^{k}\binom{k}{l}(\fanoind \tau)^{k-l}\cG_l(0)=0.
	\]
	Thus, for each $k=0,\cdots,\dim X+1$, the validity of
	$\cG_l(0)=0$ for $l=0,\dots,k$ implies the validity of
	equation~\eqref{eqn:genus-one-formula}. Conversely, since the above
	relation is triangular with leading coefficient $1$, the validity of
	equation~\eqref{eqn:genus-one-formula} for $k=0,\dots,l$ implies
	$\cG_l(0)=0$.
	
	Moreover, the $k=0,1$ cases of equation~\eqref{eqn:genus-one-formula}
	hold since $\cG_0(0)=\cG_1(0)=0$. Hence the conditions
	\eqref{eqn:genus-one-formula} for $k=2,\dots,\dim X+1$ are equivalent to
	$\cG_k(0)=0$ for $k=2,\dots,\dim X+1$. The corollary then follows from
	Proposition~\ref{prop:virasoro-reconstruction}.
\end{proof}

\subsection{Comparison formula}
\label{sec:comparison_formula}
In this subsection, we derive a comparison formula between the genus-one Gromov--Witten invariants of $X$ and the corresponding equivariant twisted invariants. 
The formula follows by combining the one-point wall-crossing formula with Quantum Lefschetz for all-light quasimaps.
The degree-zero equivariant twisted invariants need not be regular at $(\lambda,\rho)=(0,0)$. 
By convention, we set the non-equivariant limit of such an invariant to be zero whenever it is homogeneous of positive degree in $\lambda$ and $\rho$.

\begin{theorem}[Comparison formula]\label{thm:GW-tw}
	For $k\geq2$, we have
	\beq\label{eqn:GW-tw}
	\begin{split}
	&\, \<\widetilde\hh^{*k}\>_{1,1}
	-\frac{1}{24}\sum_{a+b=k-1}\str\bigl(\widetilde\hh^{*a}*q\tfrac{d}{dq}\big(\widetilde\hh^{*b}\big)*\bigr)\\
	=& \, \lim_{\substack{\rho,\lambda\to0}}
	\biggl(\<\widetilde\hh^{\twqp k}\>^{\rm tw}_{1,1}
	-\frac{1}{24}\sum_{a+b=k-1}\tr\bigl(\widetilde\hh^{\twqp a}\twqp 
	q\tfrac{d}{dq}\big(\widetilde\hh^{\twqp b}\big)\twqp\bigr)
	\biggr),
	\end{split}
	\eeq
	where $\twqp=*_{\tw}$.
\end{theorem}
\begin{proof}
	By the ancestor one-point wall-crossing formula
	\eqref{eqn:onepointwc-final}, only the terms of $T_k(z)$ of degree at
	most one in $z$ contribute in genus one with one ancestor marking,
	since $\bar\psi^2=0$ on $\overline M_{1,1}$. The recursion
	\eqref{eqn:Tk-recursion} gives
	\[
	T_k(z)=\widetilde\hh^{*k}
	-z\sum_{a+b=k-1}
	\widetilde\hh^{*a}*D(\widetilde\hh^{*b})+O(z^2).
	\]
	The genus-one ancestor topological recursion relation gives
	\(
	\<v\bar\psi\>_{1,1}=\frac{1}{24}\str(v*)
	\).
	It follows that
	\[
	\<\lh^k\>^{\epsilon=0+,X}_{1,|1}
	=
	\<\widetilde\hh^{*k}\>_{1,1}
	-\frac{1}{24}\sum_{a+b=k-1}
	\str\bigl(\widetilde\hh^{*a}*D(\widetilde\hh^{*b})*\bigr).
	\]
	The analogous identity holds in the equivariant twisted theory, with
	$*$ and $\str$ replaced by $\twqp$ and $\tr$, respectively.
	
	We compare the two all-light invariants degree by degree. Denote by
	$\<\lh^k\>^{\epsilon=0+,X}_{1,|1,d}$ and
	$\<\lh^k\>^{\epsilon=0+,\mathbb P,\tw}_{1,|1,d}$
	their degree-$d$ parts, respectively. For every $d>0$, the genus-one
	Quantum Lefschetz principle for all-light quasimaps, as explained in
	the Introduction, gives
	\[
	\<\lh^k\>^{\epsilon=0+,X}_{1,|1,d}
	=
	\lim_{\rho,\lambda\to0}
	\<\lh^k\>^{\epsilon=0+,\mathbb P,\tw}_{1,|1,d}.
	\]
	In degree zero, Quantum Lefschetz does not apply directly. 
	The dimension constraint shows that the untwisted degree-zero contribution vanishes for $k\geq2$. 
	On the twisted side, the degree-zero contribution is homogeneous of degree $k-1>0$ in $\lambda$ and $\rho$, 
	and hence its non-equivariant limit vanishes by the convention above.
	Summing over all degrees and combining with the two wall-crossing identities above proves~\eqref{eqn:GW-tw}.
\end{proof}

\subsection{Genus one computations for the twisted theory}
The comparison formula reduces the required Gromov--Witten computations 
to explicit calculations in the twisted theory. 
In practice, the latter can be carried out by equivariant localization and the $R$-matrix formalism. 
Before turning to the general consequence needed for the Virasoro constraints, 
we illustrate this comparison in the case of the cubic threefold.
The detailed twisted theory computations used below are explained in the Appendix.
\begin{example}[$X_3\subset \mathbb P^4$]
	Let $X=X_3\subset \mathbb P^4$. Then $\dim X=3$, $\fanoind=2$,
	$\widetilde\hh=\hh$, and
	$$
	H^{\rm amb}(X)=\sspan\{{\bf 1},\hh,\hh^2,\hh^3\},\qquad
	H^{\rm pri}(X)=H^{1,2}(X)\oplus H^{2,1}(X),
	$$
	with $\dim H^{1,2}(X)=\dim H^{2,1}(X)=5$. 
	The operator $\mu$ satisfies
	$\mu(\hh^a)=(a-\frac{3}{2})\hh^a$ for $a=0,\ldots,3$ on
	$H^{\rm amb}(X)$, while $\mu=-\frac{1}{2}$ on $H^{1,2}(X)$ and
	$\mu=\frac{1}{2}$ on $H^{2,1}(X)$.
	
	On the twisted side, take equivariant parameters $\rho$ and
	$\lambda_i=\xi^i\lambda$, $i=0,\cdots,4$, where $\xi=e^{2\pi i/5}$.
	The small quantum powers are
	$$
	\hh^{\twqp 0}={\bf 1},\quad
	\hh^{\twqp 1}=\hh,\quad
	\hh^{\twqp 2}=\hh^2+6q,\quad
	\hh^{\twqp 3}=\hh^3+21q\hh+11q\rho,
	$$
	and
	$$
	\hh^{\twqp 4}
	=\hh^4+27q\hh^2+25q\rho\hh+6q(\rho^2+27q).
	$$
	The trace term
	$-\frac{1}{24}\sum_{a+b=k-1}\tr\big(\hh^{\twqp a}\twqp
	D(\hh^{\twqp b})\twqp\big)$ is zero for $k=0,1,2$, and for $k=3$ equals
	$$
	-\frac{1}{24}\sum_{a+b=2}
	\tr\big(\hh^{\twqp a}\twqp D(\hh^{\twqp b})\twqp\big)
	=-\frac{5q}{4},
	$$
	where $D=q\frac{d}{dq}$. 
	Moreover, the genus-one twisted invariants are
	{
		\begin{align*}
			\<\hh^{\twqp 0}\>^{\rm tw}_{1,1}
			&=\frac{5\rho^4}{24(\rho^5+243\lambda^5)},\qquad
			\<\hh^{\twqp 1}\>^{\rm tw}_{1,1}
			=-\frac{10\rho^5+2025\lambda^5}{24(\rho^5+243\lambda^5)},\\
			\<\hh^{\twqp 2}\>^{\rm tw}_{1,1}
			&=-\frac{45\lambda^5\rho}{8(\rho^5+243\lambda^5)},\qquad
			\<\hh^{\twqp 3}\>^{\rm tw}_{1,1}
			=-\frac{62q\rho^5-15\lambda^5\rho^2+15066\lambda^5q}
			{8(\rho^5+243\lambda^5)}.
		\end{align*}
	}
	
	On the Gromov--Witten side, $\hh*\gamma=0$ for any
	$\gamma\in H^{\rm pri}(X)$, and the small quantum powers are
	$$
	\hh^{*0}={\bf 1},\quad
	\hh^{*1}=\hh,\quad
	\hh^{*2}=\hh^2+6q,\quad
	\hh^{*3}=\hh^3+21q\hh,\quad
	\hh^{*4}=27q\hh^2+162q^2.
	$$
	The corresponding supertrace term is zero for $k=0,1,2$. For $k=3$,
	using $\str({\bf 1}*)=4-10=-6$ (this differs from the twisted side, where
	$\tr({\bf 1}\twqp)=5$), it is
	$$
	-\frac{1}{24}\sum_{a+b=2}
	\str\big(\hh^{*a}*D(\hh^{*b})*\big)
	=\frac{3q}{2}.
	$$
	
	By the string equation, $\<{\bf 1}\>_{1,1}=0$. The $L_0$-constraint gives
	$$
	\<\hh\>_{1,1}
	=-\frac{1}{48}\cdot \Big(-\frac{3}{2}\Big)\cdot \str({\bf 1}*)
	-\frac{1}{8}\str(\mu^2)
	=-\frac{3}{16}-\frac{5}{16}
	=-\frac{1}{2}.
	$$
	We also have $\<\hh^2\>_{1,1}=0$ by dimension reasons.
	
	Finally, applying the comparison formula for $k=3$ gives
	$$
	\<\hh^{*3}\>_{1,1}+\frac{3q}{2}
	=
	\lim_{{\rho,\lambda\to 0}}
	\left(
	-\frac{62q\rho^5-15\lambda^5\rho^2+15066\lambda^5q}
	{8(\rho^5+243\lambda^5)}
	-\frac{5q}{4}
	\right)
	=-9q.
	$$
	Therefore
	$$
	\<\hh^3\>_{1,1}
	=\<\hh^{*3}\>_{1,1}-21q\<\hh\>_{1,1}
	=-9q-\frac{3q}{2}+\frac{21q}{2}
	=0.
	$$
	These computations show explicitly how the genus-one Gromov--Witten invariants
	are recovered from the twisted theory after taking the non-equivariant limit,
	and illustrate the verification of the initial conditions
	\eqref{eqn:genus-one-formula} in this example.
\end{example}

In the general proof of the genus-one Virasoro conjecture, however, only the following identity is needed.

\begin{proposition}\label{prop:lim-vira-tw}
	For $k=2,3,\cdots,\dim X+1$, 
	\begin{align}
		&\lim_{\substack{\rho,\lambda\to 0}}\bigg(\<\widetilde{\hh}^{\twqp k}\>^{\rm tw}_{1,1} -\frac{1}{24}\sum_{a+b=k-1}\tr\big(\widetilde{\hh}^{\twqp a}\twqp q\tfrac{d}{dq}(\widetilde{\hh}^{\twqp b})\twqp \big)\bigg)\nonumber\\
		=&-\frac{k\, (k-1-\dim X)}{48\,\fanoind}\tr_{\rm amb}\big(\widetilde{\hh}^{*(k-1)}*\big)
		-\frac{1}{4\,\fanoind}\sum_{a+b=k-1}\tr_{\rm amb}\big(\widetilde{\hh}^{*a}* \,\mu\, \widetilde{\hh}^{*b}*\,\mu\big).\label{eqn:tw-lim}
	\end{align}
	Here $\tr_{\rm amb}(A)$ denotes the trace of the operator $A$ on the ambient space.
\end{proposition}
\begin{proof}
	We give a sketch of the proof using the notation introduced in
	Appendix~\ref{sec:CohFT-comp-tw}; the details are provided there.
	The Givental--Teleman reconstruction theorem for semisimple
	cohomological field theories, together with the quasi-homogeneity
	of the equivariant twisted Gromov--Witten theory, gives
	\begin{align}
		\<\widetilde{\hh}^{\twqp k}\>^{\tw}_{1,1}
		={}&-\frac{1}{24\fanoind}
		\sum_{a+b=k-1}
		\tr\bigl(
		\widetilde{\hh}^{\twqp a}\twqp
		\mathcal V(\widetilde{\hh}^{\twqp b})\twqp
		\bigr)
		-\frac{1}{4\fanoind}
		\sum_{a+b=k-1}
		\tr\bigl(
		\widetilde{\hh}^{\twqp a}\twqp
		\mu\widetilde{\hh}^{\twqp b}\twqp\mu
		\bigr)\nonumber\\
		&+\frac{1}{4\fanoind}
		\sum_{a+b=k-1}
		\tr\bigl(
		\widetilde{\hh}^{\twqp a}\twqp
		(\mathcal V-\mu)\widetilde{\hh}^{\twqp b}\twqp
		(\mathcal V-\mu)
		\bigr)
		-\frac{1}{2}
		\tr\bigl(
		\widetilde{\hh}^{\twqp k}\twqp\Lambda(R_1^*)
		\bigr),
	\end{align}
	where $R_1^*$ is represented with respect to the flat basis.
	A further application of quasi-homogeneity gives
	\begin{align}
		&\<\widetilde{\hh}^{\twqp k}\>^{\tw}_{1,1}
		-\frac{1}{24}\sum_{a+b=k-1}
		\tr\bigl(
		\widetilde{\hh}^{\twqp a}\twqp
		q\tfrac{d}{dq}(\widetilde{\hh}^{\twqp b})\twqp
		\bigr)\nonumber\\
		={}&-\frac{k(k-1-\dim X)}{48\fanoind}
		\tr\bigl(
		\widetilde{\hh}^{\twqp(k-1)}\twqp
		\bigr)
		-\frac{1}{4\fanoind}
		\sum_{a+b=k-1}
		\tr\bigl(
		\widetilde{\hh}^{\twqp a}\twqp
		\mu\widetilde{\hh}^{\twqp b}\twqp\mu
		\bigr)\nonumber\\
		&+\frac{1}{4\fanoind}
		\sum_{a+b=k-1}
		\tr\bigl(
		\widetilde{\hh}^{\twqp a}\twqp
		(\mathcal V-\mu)\widetilde{\hh}^{\twqp b}\twqp
		(\mathcal V-\mu)
		\bigr)
		-\frac{1}{2}
		\tr\bigl(
		\widetilde{\hh}^{\twqp k}\twqp\Lambda(R_1^*)
		\bigr)\nonumber\\
		&+\frac{1}{24\fanoind}
		\sum_{a+b=k-1}\sum_\alpha
		\tr\bigl(
		\widetilde{\hh}^{\twqp a}\twqp
		\Lambda(x_\alpha^b\Psi^{\bar\alpha}_0)
		\bar e_\alpha\twqp
		\bigr).
		\label{eqn:vira-genus-one-w-tw}
	\end{align}
	By Proposition~\ref{prop:twisted-trace-limits}, the first two
	terms on the right-hand side have the corresponding ambient limits,
	while the remaining three terms vanish. Taking the non-equivariant
	limit in~\eqref{eqn:vira-genus-one-w-tw} therefore gives
	\eqref{eqn:tw-lim}.
\end{proof}
\begin{remark}
	The first two steps of the proof are formal consequences of
	semisimple reconstruction and quasi-homogeneity.
	Only the final non-equivariant limit argument uses explicit computations in the
	twisted theory.
\end{remark}

\subsection{Finishing the proof of Theorem~\ref{mainthm:genus-one-virasoro}}

We now complete the proof of the genus-one ambient Virasoro conjecture.
\begin{theorem}[=Theorem~\ref{mainthm:genus-one-virasoro}]
	Let $X\subset\mathbb P^{N-1}$ be a smooth Fano complete intersection.
	Then the genus-one Virasoro conjecture
	\eqref{eqn:vira-genus-one} holds for every
	$t\in H^{\rm amb}(X)$.
\end{theorem}

\begin{proof}
	If $X$ is exceptional, its big quantum cohomology is generically semisimple~\cite{BM04,CMP10,Hu21}, 
	and the theorem follows immediately. 
	We may therefore assume that $X$ is non-exceptional.
	
	By Corollary~\ref{cor:vira-initial}, it remains to prove the cases
	$k=2,\cdots,\dim X+1$ of equation~\eqref{eqn:genus-one-formula}.
	Combining the comparison formula \eqref{eqn:GW-tw} with
	Proposition~\ref{prop:lim-vira-tw}, we obtain
	\begin{align}
		&\<\widetilde {\hh}^{*k}\>_{1,1}
		-\frac{1}{24} \sum_{a+b=k-1}
		\str\big(\widetilde {\hh}^{*a}* q\tfrac{d}{dq} (\widetilde {\hh}^{*b})*\big)\nonumber \\
		=&-\frac{k\, (k-1-\dim X)}{48\,\fanoind}
		\tr_{\rm amb}\big(\widetilde{\hh}^{*(k-1)}*\big)
		-\frac{1}{4\,\fanoind}\sum_{a+b=k-1}
		\tr_{\rm amb}\big(\widetilde{\hh}^{*a}* \,\mu\, \widetilde{\hh}^{*b}*\,\mu\big),
		\label{eqn:gw-vira-tw}
	\end{align}
	for $k=2,\cdots,\dim X+1$.
	It remains to show that \eqref{eqn:gw-vira-tw} is equivalent to the
	required identity \eqref{eqn:genus-one-formula}.
	
	Set $\deg=\mu+\frac{\dim X}{2}$.  
	Combining Lemma~\ref{lem:hom-quantum-product} with
	$E|_{t=0}=\fanoind D$, which follows from the divisor equation, we obtain
	$$
	(\fanoind D+\deg)(v_1*v_2)
	=(\fanoind D+\deg)(v_1)*v_2
	+v_1*(\fanoind D+\deg)(v_2).
	$$
	Since $(\fanoind D+\deg)\widetilde\hh=\widetilde \hh$, induction on $k$
	gives
	$
	(\fanoind D+\deg)(\widetilde\hh^{*k})=k\, \widetilde\hh^{*k}.
	$
	This is equivalent to
	\beq\label{eqn:hom-small-quantum-powers}
	q\frac{\pd}{\pd q}(\widetilde{\hh}^{*k})
	=\frac{1}{\fanoind}
	\Big(k-\frac{\dim X}{2}-\mu\Big)(\widetilde{\hh}^{*k}).
	\eeq
	It follows from \eqref{eqn:hom-small-quantum-powers} that
	\beq\label{eqn:phia-D-phib}
	\sum_{a+b=k-1}\widetilde{\hh}^{*a}*q\tfrac{d}{dq}(\widetilde{\hh}^{*b})
	=\frac{k(k-1-\dim X)}{2\, \fanoind}\,
	\widetilde{\hh}^{*(k-1)}
	-\frac{1}{\fanoind}\sum_{a+b=k-1}
	\widetilde{\hh}^{*a}*\mu(\widetilde{\hh}^{*b}).
	\eeq
	Therefore equation~\eqref{eqn:genus-one-formula} can be rewritten in the
	following equivalent form:
	\begin{align}
		&\, \<\widetilde {\hh}^{*k}\>_{1,1}
		-\frac{1}{24} \sum_{a+b=k-1}
		\str\big(\widetilde {\hh}^{*a}*
		q\tfrac{d}{dq} (\widetilde {\hh}^{*b})*\big)\nonumber \\
		=&\, -\frac{k(k-1-\dim X)}{48\,\fanoind}
		\str\big(\widetilde{\hh}^{*(k-1)}*\big)
		-\frac{1}{4\, \fanoind} \sum_{a+b=k-1}
		\str\big(\widetilde{\hh}^{*a}* \,\mu\, \widetilde{\hh}^{*b}*\,\mu\big).
		\label{eqn:vira-genus-one-w}
	\end{align}
	
	Comparing \eqref{eqn:gw-vira-tw} with
	\eqref{eqn:vira-genus-one-w}, it remains to check that the primitive part
	does not contribute to the supertraces appearing on the right-hand side of \eqref{eqn:vira-genus-one-w}.
	Equivalently, we need to prove, for $k=2,\cdots,\dim X+1$, that
	\beq\label{eqn:full-amb-tr-1}
	\str\big(\widetilde{\hh}^{*(k-1)}*\big)
	= \tr_{\rm amb}\big(\widetilde{\hh}^{*(k-1)}*\big),
	\eeq
	and
	\beq\label{eqn:full-amb-tr-2}
	\sum_{a+b=k-1}
	\str(\widetilde{\hh}^{*a}*\, \mu\,  \widetilde{\hh}^{*b}*\, \mu)
	=
	\sum_{a+b=k-1}
	\tr_{\rm amb}(\widetilde{\hh}^{*a}*\, \mu\,  \widetilde{\hh}^{*b}*\, \mu).
	\eeq
	
	We now prove these two identities. By monodromy invariance
	and~\cite{hu2015big}, genus-zero invariants with exactly one primitive
	insertion vanish. Moreover, for
	$\gamma_1,\gamma_2\in H^{\rm pri}(X)$ and $j\geq 0$, one has
	$$
	\<\widetilde{\hh}^{*j},\gamma_1,\gamma_2\>_{0,3}
	=\delta_{j,0}\,\eta(\gamma_1,\gamma_2).
	$$
	These two facts give $\widetilde{\hh}*\gamma=0$ for every
	$\gamma\in H^{\rm pri}(X)$.
	Since $k-1\geq 1$, the primitive contribution to $\str\big(\widetilde{\hh}^{*(k-1)}*\big)$ is zero, and
	\eqref{eqn:full-amb-tr-1} follows.
	
	It remains to prove \eqref{eqn:full-amb-tr-2}. 
	Using the primitive basis $\{\gamma_i\}$ and its dual basis
	$\{\gamma^i\}$ fixed above, the primitive contribution is 
	$$
	\str_{\rm pri}(\widetilde{\hh}^{*a}*\, \mu\,  \widetilde{\hh}^{*b}*\, \mu)
	=\sum_i
	\eta(\gamma_i,
	\widetilde{\hh}^{*a}*\, \mu\,  \widetilde{\hh}^{*b}*\, \mu\gamma^i).
	$$
	The right-hand side can be rewritten as
	$$
	\sum_i
	\eta(\widetilde{\hh}^{*a}*\gamma_i,
	\mu\,  \widetilde{\hh}^{*b}*\, \mu\gamma^i).
	$$
	Since $\mu(H^{\rm pri}(X))\subset H^{\rm pri}(X)$ and
	$a+b=k-1\geq 1$, at least one of $a$ and $b$ is positive. If $a\geq 1$,
	then $\widetilde{\hh}^{*a}*\gamma_i=0$; if $b\geq 1$, then
	$\widetilde{\hh}^{*b}*(\mu\gamma^i)=0$. Therefore the primitive
	contribution vanishes in every summand, proving
	\eqref{eqn:full-amb-tr-2}.
	
	This proves \eqref{eqn:vira-genus-one-w}, and hence
	\eqref{eqn:genus-one-formula}. The theorem now follows from
	Corollary~\ref{cor:vira-initial}.
\end{proof}

\begin{remark}
	By considering monodromy invariants, Hu~\cite{hu2015big,Hu22} obtained a reduction
	structure for the genus-one potential $F_1$.
	This reduction structure naturally applies to the genus-one Virasoro constraints.
	Together with an argument similar to that used in the proof of
	Proposition~\ref{prop:virasoro-reconstruction}, it gives a reconstruction
	procedure for the genus-one Virasoro conjecture on the full cohomology.
	
	The resulting initial conditions are not, in general, purely ambient; they may
	also involve primitive insertions. Nevertheless, in the following cases these
	remaining initial conditions can be reduced to the ambient result:
	\begin{enumerate}
		\item $X_{3}\subset \mathbb P^{n}$, $n\geq 4$ and $n\ne 5$;
		\item $X_{2,2}\subset\mathbb P^{2n+1}$, $n\geq 2$;
		\item Fano complete intersections $X_{\bf m}$ satisfying
		$\gcd(\dim X-2,\fanoind)\ne1$.
	\end{enumerate}
	Therefore, our result implies that the genus-one Virasoro constraints hold
	for the full cohomology in all the above cases. Details will appear in a
	forthcoming paper.
\end{remark}

\begin{appendices}
	
	\addtocontents{toc}{\protect\setcounter{tocdepth}{0}}
	
	\section{Cohomological field theory computations for the twisted theory }
	\label{sec:CohFT-comp-tw}
	In this appendix, we collect the CohFT computations for the
	equivariant twisted theory of $\mathbb P^{N-1}$ associated with the
	Fano complete intersection $X$, and use them to prove
	Proposition~\ref{prop:lim-vira-tw}. 
	
	\noindent {\bf Notation convention.} Except in the final subsection,
	all correlators, the quantum product, the Poincar\'e pairing, and $S$- and $R$-matrices
	are understood in this twisted theory. 
	For simplicity, we omit the corresponding subscript and superscript ``$\tw$".
	
	\subsection{Twisted $I$-function and Picard--Fuchs equation}
	We retain the notation for the equivariant twisted theory introduced in \S \ref{subsec:one-point-wc}. Throughout this appendix, we
	specialize the equivariant parameters of $\mathbb P^{N-1}$ to
	$\lambda_i=\xi^i\lambda$, $i=0,\ldots,N-1$, where
	$\xi=e^{2\pi i/N}$.
	The relation $\prod_i(\hh-\lambda_i)=0$ then specializes to
	$\hh^N=\lambda^N$.
	
	The corresponding twisted $I$-function is
	\begin{align*}
		I^{\lambda,\rho}(q,z):=
		&\ z\sum_{d\geq 0}q^d
		\frac{\prod_{j=1}^r\prod_{k=1}^{m_jd}
			(m_j\hh+kz-\rho)}
		{\prod_{k=1}^{d}((\hh+kz)^N-\lambda^N)}
		=z+\delta_{\fanoind,1}{\bf m}!q+O(z^{-1}).
	\end{align*}
	This $I$-function satisfies the following Picard--Fuchs equation
	\beq\label{eqn:PF-I}
	\Big(D_{\hh}^N-\lambda^N
	-q\prod_{j=1}^r\prod_{k=1}^{m_j}(m_jD_{\hh}+kz-\rho)\Big)I^{\lambda,\rho}(q,z)=0,
	\eeq
	where $D_{\hh}=zD+\hh$ with $D=q\frac{d}{dq}$.
	
	We use the following notation throughout this appendix:
	$$
	W(x)=\prod_{i=1}^r(m_ix-\rho)^{m_i},\qquad
	P(x,q)=x^N-\lambda^N-qW(x).
	$$
	For a function of $x$, primes always denote derivatives with respect to $x$.
	
	\subsection{S-matrix and quantum product}
	The mirror theorem~\cite{Giv96,lian1997mirror} gives
	\(
	I^{\lambda,\rho}(q,-z)
	=-zS^\tau(z)^{-1}{\bf1},
	\)
	where
	$\tau=\delta_{\fanoind,1}{\bf m}!q\cdot{\bf1}$.
	Using the symplectic condition
	$S^{\tau,*}(-z)S^\tau(z)=\id $, where
	$S^{\tau,*}(z)$ denotes the adjoint of $S^\tau(z)$ with respect to
	$\eta$, we obtain
	\[
	I^{\lambda,\rho}(q,z)=zS^{\tau,*}(z){\bf1}.
	\]
	
	The quantum differential equation gives, for any vector field $v$,
	possibly depending on $q$,
	\beq\label{eqn:QDE-S}
	D_{\hh}(S^{\tau,*}(z)v)=S^{\tau,*}(z)\big(\widetilde{\hh}*v+zD(v)\big),
	\eeq
	where \(\widetilde\hh=\hh+\tau\). 
	Starting from $S^{\tau,*}(z){\bf1}=z^{-1}I^{\lambda,\rho}(q,z)$,
	the quantum product can be determined recursively as follows~\cite{chang2018polynomial}.
	Applying \eqref{eqn:QDE-S} to \(v={\bf1}\), we obtain
	$$
	S^{\tau,*}(z){\hh}=D_{\hh}(S^{\tau,*}(z){\bf1})-\tau\cdot S^{\tau,*}(z){\bf1}.
	$$
	Suppose that the formulas for $\widetilde{\hh}*{\hh}^{i}$ and $S^{\tau,*}(z){\hh}^{i+1}$ are known for $i=0,\cdots,k-1$, where $1\leq k\leq N-1$.
	Then we consider
	$$
	S^{\tau,*}(z)(\widetilde{\hh}*{\hh}^k)=D_{\hh}(S^{\tau,*}(z){\hh}^k).
	$$
	Taking the coefficient of \(z^0\) on both sides determines the formula for $\widetilde{\hh}*{\hh}^{k}$. 
	It has the form $\widetilde{\hh}*{\hh}^{k}={\hh}^{k+1}+\sum_{i=0}^{k}a_{k,i}{\hh}^i$ for some coefficients $a_{k,i}=a_{k,i}(q,\lambda,\rho)$.
	Consequently,
	$$
	S^{\tau,*}(z){\hh}^{k+1}
	=D_{\hh}(S^{\tau,*}(z){\hh}^k)
	-\sum_{i=0}^{k}a_{k,i}S^{\tau,*}(z){\hh}^i.
	$$
	By induction, one obtains the matrix of quantum multiplication by \(\widetilde\hh\) with respect to the basis $\{{\hh}^k\}_{k=0}^{N-1}$.
	
	Having determined the quantum multiplication by $\widetilde\hh$,
	taking successive quantum powers gives the transition matrix
	$\mathcal B$ between quantum powers and cup powers:
	\beq\label{eqn:Trans-QH-H}
	\{1,\widetilde\hh,\widetilde\hh^{*2},\cdots,
	\widetilde\hh^{*(N-1)}\}
	=\{1,\hh,\hh^2,\cdots,\hh^{N-1}\}\mathcal B.
	\eeq
	The triangular form of the quantum multiplication shows that
	$\mathcal B$ is unipotent and hence invertible. Under the specialization
	$\lambda_i=\xi^i\lambda$, the $I$-function, and hence the quantum
	product, depends on $\lambda$ only through $\lambda^N$. Moreover,
	$a_{k,i}$ is homogeneous of degree $k+1-i$, where
	$\deg\hh=\deg\lambda=\deg\rho=1$ and $\deg q=\fanoind$.
	Thus $a_{k,i}$ is independent of $\lambda$ for $k\leq N-2$.
	Since only these multiplication formulas are needed to construct
	$\widetilde\hh^{*j}$ for $j=0,\ldots,N-1$, the matrix $\mathcal B$ is
	independent of $\lambda$.

	After localizing the equivariant parameters, the twisted quantum product is semisimple at \(q=0\), and hence in a formal neighborhood of \(q=0\).
	We fix the canonical basis $\{e_\alpha\}_{\alpha=0}^{N-1}$ by requiring the following specialization at $q=0$:
	$$
	e_\alpha|_{q=0}=\frac{\prod_{\beta\ne\alpha}(\hh-\lambda_\beta)}
	{\prod_{\beta\ne\alpha}(\lambda_\alpha-\lambda_\beta)}
	=\frac{\prod_{\beta\ne\alpha}(\hh-\xi^\beta\lambda)}
	{N\xi^{(N-1)\alpha}\lambda^{N-1}},
	\qquad \alpha=0,\ldots,N-1,
	$$
	where $\lambda_\alpha=\xi^\alpha\lambda$.
	It follows from the relation $\prod_\beta(\hh-\lambda_\beta)=0$ that
	$\hh\cup(e_\alpha|_{q=0})=\lambda_\alpha e_\alpha|_{q=0}$, and hence
	$(e_\alpha\cup e_\beta)|_{q=0}=\delta_{\alpha,\beta}e_\alpha|_{q=0}$.
	Recall that
	$$
	\eta(\hh^a,\hh^b)=\int_{\mathbb P^{N-1}}\hh^a\cup\hh^b\cup\Big(\prod_{i=1}^r(m_i\hh-\rho)\Big).
	$$
	Therefore
	\beq\label{eqn:eta-at-q=0}
	\eta(e_\alpha,e_\beta)|_{q=0}
	=\delta_{\alpha,\beta}\cdot
	\frac{\prod_{i=1}^r(m_i\lambda_\alpha-\rho)}
	{\prod_{\gamma\ne\alpha}(\lambda_\alpha-\lambda_\gamma)}.
	\eeq
	Set $\Delta_\alpha^{-1}=\eta(e_\alpha,e_\alpha)$ and $\bar e_\alpha=\Delta_\alpha^{1/2}e_\alpha$.
	Then $\{\bar e_\alpha\}_{\alpha=0}^{N-1}$ is the normalized canonical basis, 
	i.e., $\eta(\bar e_\alpha,\bar e_\beta)=\delta_{\alpha,\beta}$.
	We also write $\Delta=\diag\{\Delta_0,\ldots,\Delta_{N-1}\}$.

	\subsection{$\Psi$-matrix and $R$-matrix}
	\label{sec:PF-R}
	By Givental~\cite{Giv01b}, at a generic point
	$t=\sum_i t^i\hh^i$, the modified $S$-matrix admits the asymptotic
	factorization
	\beq\label{eqn:S-R}
	{}^{\mathrm{new}}S^{t,*}(z)
	:=C_1^*(\lambda,z)C_2^*(\rho,z)S^{t,*}(z)
	\asymp e^{U(t)/z}R^{t,*}(z)\Psi^t.
	\eeq
	Here $\asymp$ denotes the asymptotic expansion as $z\to0$, and
	$U(t)=\diag\{u^0,\ldots,u^{N-1}\}$ is the diagonal matrix of
	canonical coordinates. The matrix $\Psi^t$ is the transformation
	matrix from a flat basis to the normalized canonical basis, while
	$R^{t,*}$ is the adjoint of the $R$-matrix $R^t$ with respect to
	$\eta$ and has the expansion
	$R^{t,*}(z)=\sum_{k\geq0}R_k^{t,*}z^k$.
	The correction factors in~\eqref{eqn:S-R} are
	\begin{align*}
		C_1^*(\lambda,z)=&\,
		\diag\Big(\Big\{e^{\sum_{k\geq1}\sum_{j\ne i}
			\frac{B_{2k}}{2k(2k-1)}
			\frac{z^{2k-1}}{(\lambda_j-\lambda_i)^{2k-1}}}\Big\}_{i=0,\cdots,N-1}\Big),\\
		C_2^*(\rho,z)=&\,
		e^{\sum_{k\geq1}\sum_{a=1}^r
			\frac{B_{2k}}{2k(2k-1)}
			\frac{z^{2k-1}}{(m_a\hh-\rho)^{2k-1}}},
	\end{align*}
	where $B_{2k}$ are the Bernoulli numbers, and
	$C_2^*(\rho,z)$ is understood as the operator of multiplication by
	the displayed class.
	Moreover, the $R$- and $\Psi$-matrices satisfy the string equations
	$\pd_{t^0}R^{t,*}(z)=0$ and $\pd_{t^0}\Psi^t=0$, as well as the
	divisor equations
	$\pd_{t^1}R^{t,*}(z)=q\pd_qR^{t,*}(z)$ and
	$\pd_{t^1}\Psi^t=q\pd_q\Psi^t$.
	The $R$-matrix also satisfies the symplectic condition
	$R^{t,*}(-z)R^t(z)=\id$.

	At $t=0$, we denote the corresponding $\Psi$- and $R$-matrices by $\Psi$ and $R^{*}(z)$, respectively.
	The first column of $R^{*}(z)\Psi$ is uniquely determined by the Picard--Fuchs equation and
	the classical limit condition: 
	\beq\label{eqn:initial-R}
	R^{*}(z)|_{q=0}=e^{\diag(b_0(z),\cdots,b_{N-1}(z))},
	\eeq
	where
	$$
	b_i(z)=\sum_{k\geq1}\frac{B_{2k}}{2k(2k-1)}
	\left(\sum_{j\ne i}\frac{z^{2k-1}}{(\lambda_j-\lambda_i)^{2k-1}}
	+\sum_{a=1}^r\frac{z^{2k-1}}{(m_a\lambda_i-\rho)^{2k-1}}\right).
	$$
	The remaining columns of $R^*(z)\Psi$ can be computed from the first
	column by evaluating the quantum differential equation
	\beq\label{eqn:QDE-R}
	z\pd_{t^i}\left(e^{U(t)/z}R^{t,*}(z)\Psi^t\right)
	=
	e^{U(t)/z}R^{t,*}(z)\Psi^t(\phi_i*_t)
	\eeq
	at $t=0$.
	In this subsection, we explicitly compute the matrices $\Psi$ and $R_1^*$.
	\subsubsection{Picard--Fuchs equation for $R^*\Psi {\bf 1}$}
	Since $\tau$ lies in the unit direction, the string equation gives
	$R^{\tau,*}(z)=R^*(z)$ and $\Psi^{\tau}=\Psi$.
	Moreover, the correction factors $C_1^*(\lambda,z)$ and
	$C_2^*(\rho,z)$ are independent of $q$.
	Therefore, applying the asymptotic relation~\eqref{eqn:S-R} at
	$t=\tau$ and using
	$I^{\lambda,\rho}(q,z)=zS^{\tau,*}(z){\bf1}$, we see that
	\[
	\widetilde I_\alpha^{\lambda,\rho}(q,z)
	:=
	e^{u^\alpha(\tau)/z}
	R^*(z)_{\bar\beta}^{\bar\alpha}\Psi^{\bar\beta}_0
	\]
	satisfies the Picard--Fuchs equation
	\beq\label{eqn:PF-I-cano-u}
	\Big((zD+\lambda_\alpha)^N-\lambda^N
	-q\prod_{j=1}^r\prod_{k=1}^{m_j}
	(m_j(zD+\lambda_\alpha)+kz-\rho)\Big)
	\widetilde I_{\alpha}^{\lambda,\rho}(q,z)=0.
	\eeq
	To compute the $R$-matrix and $\Psi$-matrix, we rewrite the above equation as follows:
	\beq\label{eqn:PF-I-cano}
	\Big((zD+x_\alpha)^N-\lambda^N
	-q\prod_{j=1}^r\prod_{k=1}^{m_j}
	(m_j(zD+x_\alpha)+kz -\rho)\Big)R^{*}(z)^{\bar\alpha}_{0}=0,
	\eeq
	where
	$x_\alpha=\lambda_\alpha+D(u^\alpha(\tau))$,
	and $R^{*}(z)^{\bar\alpha}_{0}$ is the simplified notation of $R^{*}(z)^{\bar\alpha}_{\bar\beta}\Psi^{\bar\beta}_{0}$.
	
	Now we write $D_{\alpha}=zD+x_\alpha$ and define
	$$
	\PF(D_\alpha):=D_\alpha^N-\lambda^N
	-q\prod_{i=1}^r\prod_{k=1}^{m_i}(m_iD_\alpha+kz-\rho),
	$$
	then the Picard--Fuchs equation for $R^{*}$ is
	\beq\label{eqn:PF-R}
	\PF(D_\alpha)(R^{*}(z)^{\bar\alpha}_{0})=0.
	\eeq
	We denote
	$$
	\PF(D_\alpha)=\PF_{\alpha,0}+z\PF_{\alpha,1}+z^2\PF_{\alpha,2}+O(z^3),
	$$
	where $\PF_{\alpha,k}$ is the coefficient of $z^k$.
	It is easy to see $\PF_{\alpha,0}=P(x_\alpha,q)$.
	\begin{lemma}
		With respect to the canonical basis, the vector field $\widetilde\hh$ is given by
		$$
		\widetilde\hh=\sum_\alpha x_\alpha e_\alpha,
		$$
		where $x_\alpha$ is
		uniquely determined by
		\beq\label{eqn:PF-0}
		P(x_\alpha,q)=x_\alpha^N-\lambda^N-q\prod_{i=1}^r(m_ix_\alpha-\rho)^{m_i}=0,
		\eeq
		and the initial condition $x_\alpha|_{q=0}=\lambda_\alpha$.
		Consequently,
		\beq\label{eqn:qp-relation}
		\widetilde\hh^{*N}-\lambda^N-q(m_1\widetilde\hh-\rho)^{*m_1}*
		\cdots*(m_r\widetilde\hh-\rho)^{*m_r}=0.
		\eeq
	\end{lemma}
	\begin{proof}
		By the divisor equation and the definition of the canonical
		coordinates, $x_\alpha$ is the eigenvalue of quantum multiplication
		by $\widetilde\hh$ corresponding to $e_\alpha$. Hence
		\(
		\widetilde\hh=\sum_\alpha x_\alpha e_\alpha.
		\)
		The coefficient of $z^0$ in~\eqref{eqn:PF-R} is
		\(
		P(x_\alpha,q)\Psi^{\bar\alpha}_0=0.
		\)
		Since $\Psi^{\bar\alpha}_0$ is invertible after localization, this
		gives $P(x_\alpha,q)=0$. Moreover,
		$P'(\lambda_\alpha,0)=N\lambda_\alpha^{N-1}\ne0$, so the initial
		condition $x_\alpha|_{q=0}=\lambda_\alpha$ determines $x_\alpha$
		uniquely. Finally,
		\(
		P(\widetilde\hh)
		=\sum_\alpha P(x_\alpha,q)e_\alpha=0,
		\)
		which gives~\eqref{eqn:qp-relation}.
	\end{proof}
	
	\subsubsection{Computation of $\Psi{\bf 1}$}
	Taking the coefficient of $z^1$ in~\eqref{eqn:PF-R} gives
	$\PF_{\alpha,1}(\Psi^{\bar\alpha}_0)=0$.
	To compute $\PF_{\alpha,1}$, we introduce
	\beq\label{eqn:K1}
	K_1(x)=\sum_{i=1}^r\frac{m_i(m_i+1)}{2(m_ix-\rho)}.
	\eeq
	A direct expansion gives
	\begin{align*}
		\PF_{\alpha,1}
		=&\, P'(x_\alpha,q)D
		+\Big(\frac{1}{2}P''(x_\alpha,q)D(x_\alpha)
		-qW(x_\alpha)K_1(x_\alpha)\Big).
	\end{align*}
	Furthermore, since $D(P(x_\alpha,q))=0$, we have
	\beq\label{eqn:Dx-alpha}
	P'(x_\alpha,q)D(x_\alpha)=qW(x_\alpha)=x_\alpha^N-\lambda^N.
	\eeq
	Therefore
	$$
	\PF_{\alpha,1}
	=P'(x_\alpha,q)D+
	\Big(\frac{1}{2}P''(x_\alpha,q)D(x_\alpha)
	-P'(x_\alpha,q)D(x_\alpha)K_1(x_\alpha)\Big).
	$$
	\begin{lemma}
		We have the following formula for $\Psi^{\bar\alpha}_0$:
		\beq\label{eqn:Psi0}
		(\Psi_{0}^{\bar\alpha})^2
		=\frac{\prod_{i=1}^r(m_ix_\alpha-\rho)}{P'(x_\alpha,q)}.
		\eeq
	\end{lemma}
	\begin{proof}
		The equation $\PF_{\alpha,1}(\Psi^{\bar\alpha}_0)=0$ is equivalent to
		\beq\label{eqn:Dlog_Psi}
		\frac{D\Psi_0^{\bar\alpha}}{\Psi_0^{\bar\alpha}}
		=D(x_\alpha)\left(K_1(x_\alpha)-\frac{P''(x_\alpha,q)}{2P'(x_\alpha,q)}\right).
		\eeq
		Solving this equation gives
		$$
		(\Psi_{0}^{\bar\alpha})^2
		=c_\alpha\cdot\frac{\prod_{i=1}^{r}(m_ix_\alpha-\rho)}{P'(x_\alpha,q)}.
		$$
		The constant $c_\alpha=1$ is determined by the initial condition~\eqref{eqn:eta-at-q=0}.
	\end{proof}
	We note here that $(\Psi^{\bar\alpha}_0)^2=\eta(e_\alpha,e_\alpha)=\Delta_{\alpha}^{-1}$.
	
	\subsubsection{Computation of $R^{*}_1{\bf 1}$}
	To compute the $R$-matrix, we introduce the notation $R^*(z)^{\alpha}_k$ which should be distinguished from $R^*(z)^{\bar\alpha}_k$:
	$$
	R^*(z)^{\alpha}_{k}:=\eta(e^\alpha,R^*(z){\hh}^k),\qquad
	R^*(z)^{\bar\alpha}_{k}:=\eta(\bar e_\alpha,R^*(z){\hh}^k),
	$$
	where $\{e^\alpha\}_{\alpha=0}^{N-1}$ denotes the basis dual to
	$\{e_\alpha\}_{\alpha=0}^{N-1}$ with respect to $\eta$.
	We have $R^*(z)^{\bar\alpha}_k=\Psi^{\bar\alpha}_0R^*(z)^\alpha_k$.
	By considering the commutation relation between $\PF(D_\alpha)$ and $\Psi^{\bar\alpha}_0$, and by using the equations for $x_\alpha$ and $\Psi^{\bar\alpha}_0$, equation~\eqref{eqn:PF-R} gives
	\beq\label{eqn:PF-R-c}
	P'(x_\alpha,q)D((R_1^*)^\alpha_0)
	+\PF_{\alpha,2}(\Psi_0^{\bar\alpha})/\Psi_0^{\bar\alpha}=0.
	\eeq
	For further use, we introduce the following functions (recall $K_1(x)$ is defined by \eqref{eqn:K1}):
	\begin{align*}
		K_0(x):=&\,\sum_{i=1}^r\frac{1}{m_ix-\rho},\qquad 
		K_2(x):=\frac{1}{2}K_1(x)^2
		-\frac{1}{12}\sum_{i=1}^r\frac{m_i(m_i+1)(2m_i+1)}{(m_ix-\rho)^2}.
	\end{align*}
	A direct expansion gives
	\begin{align*}
		\PF_{\alpha,2}=&\, \frac{1}{2}P''(x_\alpha,q)D^2
		+\bigg(\frac{1}{2}P^{(3)}(x_\alpha,q)D(x_\alpha)-q(WK_1)'(x_\alpha)\bigg)D
		+\frac{1}{6}P^{(3)}(x_\alpha,q)D^2(x_\alpha) \\
		&\, +\frac{1}{8}P^{(4)}(x_\alpha,q)D(x_\alpha)^2
		-q\bigg(\frac{1}{2}(WK_1)''(x_\alpha)D(x_\alpha)+W(x_\alpha)K_2(x_\alpha)\bigg).
	\end{align*}
	
	\begin{lemma}\label{lem:formula-R1}
		We have the following formula:
		\beq\label{eqn:formula-R1}
		\begin{split}
			(R_1^*)^\alpha_0
			={}&\frac{K_0(x_\alpha)}{12}
			-\frac{P''(x_\alpha,q)}{24P'(x_\alpha,q)}
			+D(x_\alpha)\bigg[
			K_2(x_\alpha)
			-\frac{K_1(x_\alpha)P''(x_\alpha,q)}{2P'(x_\alpha,q)}
			-\frac{K'_0(x_\alpha)}{12}\\
			&\hspace{1.5cm}
			+\frac{5P''(x_\alpha,q)^2}{24P'(x_\alpha,q)^2}
			-\frac{P^{(3)}(x_\alpha,q)}{8P'(x_\alpha,q)}
			+\frac{W'(x_\alpha)P''(x_\alpha,q)}{24W(x_\alpha)P'(x_\alpha,q)}
			-\frac{W''(x_\alpha)}{24W(x_\alpha)}
			\bigg].
		\end{split}
		\eeq
	\end{lemma}
	\begin{proof}
		The equation for $(R_1^*)^\alpha_0$ can be rewritten as follows:
		\begin{align*}
			D\big((R_1^*)^\alpha_0\big)
			={}&
			-\frac12\frac{P''}{P'}
			\left(D\left(\frac{D\Psi_0^{\bar\alpha}}{\Psi_0^{\bar\alpha}}\right)
			+\left(\frac{D\Psi_0^{\bar\alpha}}{\Psi_0^{\bar\alpha}}\right)^2\right)
			-\left(\frac12\frac{P^{(3)}}{P'}D(x_\alpha)
			-q\frac{(WK_1)'}{P'}\right)
			\frac{D\Psi_0^{\bar\alpha}}{\Psi_0^{\bar\alpha}}\\
			&-\frac16\frac{P^{(3)}}{P'}D^2(x_\alpha)
			-\frac18\frac{P^{(4)}}{P'}D(x_\alpha)^2
			+\frac{q}{P'}
			\left(\frac12(WK_1)''D(x_\alpha)+WK_2\right),
		\end{align*}
		where all functions in the last display are evaluated at $(x,q)=(x_\alpha,q)$.
		Using~\eqref{eqn:Dx-alpha},~\eqref{eqn:Dlog_Psi}, and
		$$
		-K'_0(x)=6K_1(x)^2-12K_2(x)+6K'_1(x)-\frac{d}{dx}\Big(\frac{W'(x)}{W(x)}\Big),
		$$
		one checks by direct simplification that the right-hand side is the $D$-derivative of the right-hand side of equation~\eqref{eqn:formula-R1}.
		Moreover, at $q=0$ one has $x_\alpha=\lambda_\alpha$ and $D(x_\alpha)=0$, so
		$$
		(R_1^*)^\alpha_0\big|_{q=0}
		=\frac{K_0(\lambda_\alpha)}{12}
		-\frac{P''(\lambda_\alpha,0)}{24P'(\lambda_\alpha,0)}.
		$$
		Since $P(x,0)=x^N-\lambda^N=\prod_{\beta=0}^{N-1}(x-\lambda_\beta)$, we have
		$\frac{P''(\lambda_\alpha,0)}{2P'(\lambda_\alpha,0)}
		=\sum_{\beta\ne\alpha}\frac{1}{\lambda_\alpha-\lambda_\beta}$.
		Therefore
		$$
		(R_1^*)^\alpha_0\big|_{q=0}
		=\frac{1}{12}\left(
		\sum_{i=1}^r\frac{1}{m_i\lambda_\alpha-\rho}
		+\sum_{\beta\ne\alpha}\frac{1}{\lambda_\beta-\lambda_\alpha}\right),
		$$
		which agrees with the initial condition~\eqref{eqn:initial-R}. This fixes the integration constant.
	\end{proof}
	
	\subsubsection{Computation of the remaining entries}
	Now we compute the remaining columns of $\Psi$ and $R_1^*$.
	For the matrix $\Psi$, recall that $\{1,\hh,\cdots,\hh^{N-1}\}=\{e_0,\cdots,e_{N-1}\}\Delta^{1/2}\Psi$,
	and 
	$$
	\{1,\widetilde\hh,\cdots,\widetilde\hh^{*(N-1)}\}=\{e_0,\cdots,e_{N-1}\}\check\Psi,
	$$
	where $\check\Psi^\alpha_j=x_\alpha^j$, $\alpha,j=0,\cdots,N-1$.
	Combining these two changes of basis gives
	\beq\label{eqn:Psi-check-Psi}
	\Psi=\Delta^{-1/2}\check\Psi\mathcal B^{-1},
	\eeq
	where $\mathcal B$ is defined by~\eqref{eqn:Trans-QH-H}.
	Thus, $\Psi$ is explicitly determined.
	
	We next compute the remaining entries of $R_1^*$.
	Expanding the QDE~\eqref{eqn:QDE-R} gives
	\begin{align*}
		&\, e^{U(t)/z}\left(\pd_{t^i}(U(t))R^{t,*}(z)\Psi^{t}
		+z\pd_{t^i}(R^{t,*}(z))\Psi^{t}
		+R^{t,*}(z)z\pd_{t^i}(\Psi^{t})\right)\\
		=&\, e^{U(t)/z}R^{t,*}(z)\Psi^{t}(\phi_i*_{t}).
	\end{align*}
	Canceling $e^{U(t)/z}$ and comparing the coefficients of $z^0$ and
	$z^1$, we obtain 
	$$
	\pd_{t^i}U(t)=\Psi^{t} (\phi_i*_{t}) (\Psi^{t})^{-1}
	$$ 
	and
	$$
	[R^{t,*}_1,\Psi^{t} (\phi_i*_{t}) (\Psi^{t})^{-1}]=(\pd_{t^i}\Psi^{t})(\Psi^{t})^{-1},
	$$
	respectively.
	Taking $i=1$, setting $t=0$, and using the divisor equation, we obtain
	\beq\label{eqn:R1}
	[R^{*}_1,\hh*]=D(\Psi)\Psi^{-1}.
	\eeq
	Under the same specialization, the coefficient of $z^2$ gives
	\beq\label{eqn:R2}
	[R^{*}_2,\hh*]
	=D(R_1^{*})+R^{*}_1 \cdot [R^{*}_1,\hh*].
	\eeq
	
	We regard each $R_k^*$ as an endomorphism of the state space and use
	the same notation for its matrix representation when the basis is
	clear. Equations~\eqref{eqn:R1} and~\eqref{eqn:R2} are written with
	respect to the normalized canonical basis. Passing to the flat basis
	$\{\hh^k\}_{k=0}^{N-1}$, we obtain
	\beq\label{eqn:R1-2}
	[R_1^*,\hh*]=\Psi^{-1}D(\Psi)
	\eeq
	and
	\beq\label{eqn:R2-2}
	[R_2^*,\hh*]
	=
	D(R_1^*)+[R_1^*,\hh*]R_1^*,
	\eeq
	respectively. Here $D(R_1^*)$ denotes the derivative of the matrix of
	$R_1^*$ with respect to the flat basis. In deriving
	\eqref{eqn:R2-2}, we have used
	$\Psi^{-1}D(\Psi)=[R_1^*,\hh*]$.
	
	We introduce the endomorphism
	\beq\label{eqn:operator-V}
	\mathcal V:=\fanoind[R_1^*,\hh*].
	\eeq
	Thus, its matrices with respect to the normalized canonical basis and
	the flat basis are given by
	$\mathcal V=\fanoind D(\Psi)\Psi^{-1}$
	and $\mathcal V=\fanoind\Psi^{-1}D(\Psi)$,
	respectively. Since $\widetilde\hh-\hh=\tau$ is a scalar multiple of
	the unit, we also have
	$[R_1^*,\widetilde\hh*]=[R_1^*,\hh*]
	=\frac{1}{\fanoind}\mathcal V$.
	Iterating this identity, for $k\geq1$ we obtain
	\beq\label{eqn:R1k}
	R_1^*\widetilde\hh^{*k}
	=
	\widetilde\hh^{*k}*R_1^*{\bf1}
	+\frac{1}{\fanoind}
	\sum_{a+b=k-1}
	\widetilde\hh^{*a}*
	\mathcal V(\widetilde\hh^{*b}).
	\eeq
	Since
	$\{{\bf1},\widetilde\hh,\ldots,\widetilde\hh^{*(N-1)}\}$
	is a basis by~\eqref{eqn:Trans-QH-H},
	this determines all the remaining entries of $R_1^*$ from $R_1^*{\bf1}$.

	\subsection{Quasi-homogeneity conditions}
	In the Gromov--Witten theory of $X$, there is an Euler vector field $E$ such that $*$, ${\bf 1}$ and $\eta$ are eigenvectors of the Lie derivative along $E$ with eigenvalues $0$, $-1$, and $2-\dim X$, respectively.
	For the equivariant twisted theory, the equivariant parameters
	$\lambda_i$ and $\rho$ carry nontrivial degrees, and the corresponding
	grading is expressed by quasi-homogeneity conditions.
	We summarize the relations needed below and refer the reader
	to~\cite{Giv96,Giv98,Giv01a,Giv01b} for details.
	
	The state space of the equivariant twisted theory is
	$\Frob=H_{T\times\mathbb C^{*}_{\rho}}^*(\mathbb P^{N-1})$.
	We fix a flat basis $\{\phi_a=\hh^a\}_{a=0}^{N-1}$.
	Set $\delta:=\dim\mathbb P^{N-1}-\operatorname{rank}\mathbb E=N-1-r=\dim X$.
	The Euler vector field is
	$$
	E=\sum_a(1-a)t^a\phi_a+
	c_1(\mathbb P^{N-1})-c_1(\mathbb E)
	=\sum_a(1-a)t^a\phi_a+\fanoind\,\hh.
	$$
	We define the degree operator
	$$
	\mu:=1-\tfrac{\delta}{2}-\nabla E,
	$$
	where $\nabla$ is the Levi--Civita connection of $\eta$.
	Equivalently, $\mu(\phi_a)=(a-\frac{\delta}{2})\phi_a$.
	Let
	$$
	\Lambda=\rho\pd_{\rho}+\lambda\pd_{\lambda},\qquad
	E_{\Lambda}:=E+\Lambda.
	$$ 
	Here $E_\Lambda$ acts on the coefficient functions of vector fields.
	Then, for any vector fields $v,w$, the Frobenius structure satisfies
	\begin{align}
		E_{\Lambda}(\eta(v,w))-\eta((E_{\Lambda}+\mu)(v),w)-\eta(v,(E_{\Lambda}+\mu)(w))&=0,\\
		(E_{\Lambda}+\mu)(v*_{t}w)-(E_{\Lambda}+\mu)(v)*_{t}w-v*_{t}(E_{\Lambda}+\mu)(w) &=\tfrac{\delta}{2}(v*_{t}w).\label{eqn:hom-qp}
	\end{align}
	Moreover, 
	\beq\label{eqn:hom-u-Psi}
	E_{\Lambda}(u^\alpha)=u^\alpha,\qquad
	E_{\Lambda}(\Psi^{t})=\Psi^{t}\mu,
	\eeq
	and
	\beq\label{eqn:hom-R}
	(z\pd_z+E_{\Lambda})(R^{t,*}(z))=0.
	\eeq
	
	At $t=0$, we have $E|_{t=0}=\fanoind\hh$, and the divisor equation identifies its action on coefficient functions with $\fanoind D$. 
	We next record some consequences of the homogeneity conditions and the QDE.
	Combining~\eqref{eqn:R1} and~\eqref{eqn:hom-u-Psi}, we obtain,
	with respect to the normalized canonical basis,
	$\mathcal V=\Psi \mu\Psi^{-1}-\Lambda(\Psi)\Psi^{-1}$.
	Here $\mathcal V$ and $\mu$ are regarded as endomorphisms, so
	$\Psi\mu\Psi^{-1}$ is the matrix representation of $\mu$ with
	respect to the normalized canonical basis. Passing to the flat basis gives
	$\mathcal V=\mu-\Psi^{-1}\Lambda(\Psi)$.
	Since the flat basis vectors $\hh^k$ are independent of $q$, we have $D(\hh^k)=0$. 
	Differentiating $\hh^k=\sum_\alpha\bar e_\alpha
	\Psi^{\bar\alpha}_k$ gives
	$$
	0=\sum_\alpha D(\bar e_\alpha)\Psi^{\bar\alpha}_k
	+\sum_\alpha\bar e_\alpha D(\Psi^{\bar\alpha}_k).
	$$
	Together with $\fanoind D(\Psi)=\mathcal V\Psi$ and the invertibility of $\Psi$, this implies, for each $\alpha$,
	\beq\label{eqn:hom-bar-e}
	(\fanoind D+\mathcal V)(\bar e_\alpha)=0.
	\eeq
	Finally, taking the coefficient of $z^1$ in~\eqref{eqn:hom-R} gives
	\beq\label{eqn:hom-R1}
	\left(\fanoind D+\Lambda\right)(R_1^{*})=-R_1^{*}.
	\eeq
	Here $R_1^*$ is represented with respect to the normalized canonical basis.
	
	\subsection{Genus-one formula}
	We now derive the genus-one formula from Givental's reconstruction formula and the $R$-matrix.
	\begin{proposition}
		For $k\geq0$, we have
		\begin{align}
			\<\widetilde{\hh}^{*k}\>_{1,1}
			={}&-\frac{1}{24\fanoind}
			\sum_{a+b=k-1}
			\tr\bigl(
			\widetilde{\hh}^{*a}*
			\mathcal V(\widetilde{\hh}^{*b})*
			\bigr)
			-\frac{1}{4\fanoind}
			\sum_{a+b=k-1}
			\tr\bigl(
			\widetilde{\hh}^{*a}*
			\mu\widetilde{\hh}^{*b}*\mu
			\bigr)\nonumber\\
			&+\frac{1}{4\fanoind}
			\sum_{a+b=k-1}
			\tr\bigl(
			\widetilde{\hh}^{*a}*
			(\mathcal V-\mu)\widetilde{\hh}^{*b}*(\mathcal V-\mu)
			\bigr)
			-\frac{1}{2}
			\tr\bigl(
			\widetilde{\hh}^{*k}*
			\Lambda(R_1^*)
			\bigr),
			\label{eqn:equiv-vira-genus1}
		\end{align}
		where $ R_1^*$ is represented with respect to the flat basis $\{\hh^{k}\}_{k=0}^{N-1}$.
	\end{proposition}
	
	\begin{proof}
		We first derive the formula with respect to the normalized canonical basis.
		Givental's reconstruction formula (see also~\cite{guo2016genus}) gives,
		for any $v$,
		\[
		\<v\>_{1,1}
		=-\frac{1}{24}\sum_\alpha
		\eta(e^\alpha,[R_1^*,v*]{\bf1})
		+\frac{1}{2}\tr(v*R_1^*).
		\]
		For any $v$, we have
		$\sum_\alpha\eta(e^\alpha,v)=\tr(v*)$.
		Applying this identity together with~\eqref{eqn:R1k} and
		\eqref{eqn:hom-R1}, we obtain
		\begin{align}
			\<\widetilde{\hh}^{*k}\>_{1,1}
			={}&-\frac{1}{24\fanoind}
			\sum_{a+b=k-1}
			\tr\bigl(
			\widetilde{\hh}^{*a}*
			\mathcal V(\widetilde{\hh}^{*b})*
			\bigr)\nonumber\\
			&-\frac{\fanoind}{2}
			\tr\bigl(
			\widetilde{\hh}^{*k}*D(R_1^*)
			\bigr)
			-\frac{1}{2}
			\tr\bigl(
			\widetilde{\hh}^{*k}*\Lambda(R_1^*)
			\bigr).
			\label{eqn:genus-one-canonical}
		\end{align}
		By~\eqref{eqn:R2},
		$D(R_1^*)=[R_2^*,\hh*]-R_1^*[R_1^*,\hh*]$.
		Since $\widetilde\hh-\hh$ is a scalar multiple of the unit,
		$[R_i^*,\hh*]=[R_i^*,\widetilde\hh*]$. The cyclicity of the trace
		therefore gives
		\beq\label{eqn:trace-DR1}
		\tr\bigl(
		\widetilde{\hh}^{*k}*D(R_1^*)
		\bigr)
		=
		\frac{1}{2\fanoind^2}
		\sum_{a+b=k-1}
		\tr\bigl(
		\widetilde{\hh}^{*a}*
		\mathcal V\widetilde{\hh}^{*b}*\mathcal V
		\bigr).
		\eeq
		
		We then pass from the normalized canonical basis to the flat basis. Let
		\[
		R_{1}^{*,\mathrm{flat}}:=\Psi^{-1}R_1^*\Psi,
		\]
		then we have
		\[
		\Psi^{-1}\Lambda(R_1^*)\Psi
		=
		\Lambda(R_{1}^{*,\mathrm{flat}})
		-[R_{1}^{*,\mathrm{flat}},\Psi^{-1}\Lambda(\Psi)].
		\]
		By the cyclicity of the trace,
		\[
		\tr\bigl(
		\widetilde\hh^{*k}*[R_{1}^{*,\mathrm{flat}},\Psi^{-1}\Lambda(\Psi)]
		\bigr)
		=
		-\frac{1}{\fanoind}
		\sum_{a+b=k-1}
		\tr\bigl(
		\widetilde\hh^{*a}*
		(\Psi^{-1}\Lambda(\Psi))\, 
		\widetilde\hh^{*b}*\mathcal V
		\bigr).
		\]
		Notice that $\Psi^{-1}\Lambda(\Psi)$ is the matrix representation of $\mu-\mathcal V$ under the flat basis. We obtain
		\beq\label{eqn:Lambda-R1-flat}
		\tr\bigl(\widetilde{\hh}^{*k}*\Lambda(R_1^*)
		\bigr)
		=\tr\bigl(\widetilde{\hh}^{*k}*\Lambda(R_{1}^{*,\mathrm{flat}})\bigr)
		-\frac{1}{\fanoind}\sum_{a+b=k-1}
		\tr\bigl(\widetilde\hh^{*a}*(\mathcal V-\mu)
		\widetilde\hh^{*b}*\mathcal V\bigr).
		\eeq
		Substituting~\eqref{eqn:trace-DR1} and~\eqref{eqn:Lambda-R1-flat} into
		\eqref{eqn:genus-one-canonical}, and using
		$\mathcal V=\mu+(\mathcal V-\mu)$ together with the cyclicity of the
		trace, we see that the mixed terms cancel. This proves the
		proposition.
	\end{proof}
	\begin{corollary}
		For $k\geq0$, we have
		\begin{align*}
			&\<\widetilde{\hh}^{*k}\>_{1,1}
			-\frac{1}{24}\sum_{a+b=k-1}
			\tr\bigl(
			\widetilde{\hh}^{*a}*
			q\tfrac{d}{dq}(\widetilde{\hh}^{*b})*
			\bigr)\\
			={}&-\frac{k(k-1-\delta)}{48\fanoind}
			\tr\bigl(\widetilde{\hh}^{*(k-1)}*\bigr)
			-\frac{1}{4\fanoind}\sum_{a+b=k-1}
			\tr\bigl(
			\widetilde{\hh}^{*a}*
			\mu\widetilde{\hh}^{*b}*\mu
			\bigr)\\
			&+\frac{1}{4\fanoind}\sum_{a+b=k-1}
			\tr\bigl(
			\widetilde{\hh}^{*a}*
			(\mathcal V-\mu)\widetilde{\hh}^{*b}*
			(\mathcal V-\mu)
			\bigr)
			-\frac{1}{2}
			\tr\bigl(
			\widetilde{\hh}^{*k}*\Lambda(R_1^*)
			\bigr)\\
			&+\frac{1}{24\fanoind}
			\sum_{a+b=k-1}\sum_\alpha
			\tr\bigl(
			\widetilde{\hh}^{*a}*
			\Lambda(x_\alpha^b\Psi^{\bar\alpha}_0)
			\bar e_\alpha*
			\bigr).
		\end{align*}
		Here $R_1^*$ is represented with respect to the flat basis, as in
		the proposition above.
	\end{corollary}
	
	\begin{proof}
		Recall that
		$\widetilde{\hh}^{*b}
		=\sum_\alpha x_\alpha^b\Psi^{\bar\alpha}_0\bar e_\alpha$.
		The homogeneity relations
		$(\fanoind D+\Lambda)(x_\alpha)=x_\alpha$ and
		$(\fanoind D+\Lambda)(\Psi^{\bar\alpha}_0)
		=-\frac{\delta}{2}\Psi^{\bar\alpha}_0$, together
		with~\eqref{eqn:hom-bar-e}, give
		\[
		\mathcal V(\widetilde{\hh}^{*b})
		=\bigl(b-\tfrac{\delta}{2}\bigr)\widetilde{\hh}^{*b}
		-\fanoind D(\widetilde{\hh}^{*b})
		-\sum_\alpha
		\Lambda(x_\alpha^b\Psi^{\bar\alpha}_0)\bar e_\alpha.
		\]
		Substituting this expression into the first sum on the right-hand
		side of~\eqref{eqn:equiv-vira-genus1}, and using
		$\widetilde{\hh}^{*a}*\widetilde{\hh}^{*b}
		=\widetilde{\hh}^{*(k-1)}$ and
		$\sum_{a+b=k-1}(b-\frac{\delta}{2})
		=\frac{k(k-1-\delta)}{2}$, proves the Corollary.
	\end{proof}

	\subsection{Non-equivariant limit}
	We now compute the non-equivariant limits appearing in the proof of
	Proposition~\ref{prop:lim-vira-tw}. In order to distinguish the untwisted
	quantum product $*$ from the equivariant twisted quantum product
	$*_{\tw}$, we denote the latter by $\twqp$ in this subsection.
	For convenience, we first set $\rho=0$ and then let $\lambda\to0$
	throughout the computation. For the positive-degree contributions, the
	non-equivariant limits exist and are independent of the order of
	specialization. For the degree-zero contributions, this order agrees
	with the convention introduced in \S \ref{sec:comparison_formula}.
	
	\begin{lemma}\label{lem:QP-ext}
		For $j=0,\ldots,r-1$, we have
		\beq\label{eqn:H*extH}
		\lim_{\substack{\lambda,\rho\to0}}\widetilde\hh\twqp{\hh}^{N-r+j}={\hh}^{N-r+j+1}.
		\eeq
	\end{lemma}
	\begin{proof}
		We first set $\lambda=\rho=0$ in $I$-function $I^{\lambda,\rho}$, then by direct computations,
		we have the following reduced Picard--Fuchs equation:
		$$
		\left[
		D_{\hh}^{N-r}
		-q\left(\prod_{i=1}^r m_i\right)
		\prod_{i=1}^r\prod_{k=1}^{m_i-1}(m_iD_{\hh}+kz)\right]
		I^{0,0}(q,z)=z\hh^{N-r}.
		$$
		By repeatedly applying the QDE~\eqref{eqn:QDE-S}, we can write the left-hand side of above equation as
		$zS^{\tau,*}(z)v(q,z)$ for some
		$v(q,z)\in H^*(\mathbb P^{N-1})[[q]][z]$.
		Hence,
		$$
		v(q,z)=S^{\tau,*}(z)^{-1}{\hh}^{N-r}
		={\hh}^{N-r}+O(z^{-1}).
		$$
		Since $v(q,z)$ contains only non-negative powers of $z$, it follows that
		$v(q,z)={\hh}^{N-r}$. Therefore,
		\beq\label{eqn:S-extH}
		S^{\tau,*}(z){\hh}^{N-r}\big|_{\lambda=\rho=0}
		={\hh}^{N-r}.
		\eeq
		Suppose $S^{\tau,*}(z){\hh}^{k}\big|_{\lambda=\rho=0}
		={\hh}^{k}$ for some $k\geq N-r$, then by using QDE,
		\begin{align*}
			S^{\tau,*}(z)(\widetilde\hh\twqp{\hh}^k)\big|_{\lambda=\rho=0}
			=D_{\hh}\big(S^{\tau,*}(z){\hh}^k\big)\big|_{\lambda=\rho=0}
			=D_{\hh}\left(\hh^k\right)
			={\hh}^{k+1}.
		\end{align*}
		Comparing the coefficients of $z^0$, we obtain
		$$
		\widetilde\hh\twqp{\hh}^k\big|_{\lambda=\rho=0}
		={\hh}^{k+1}.
		$$
		The preceding identity then also gives
		$S^{\tau,*}(z){\hh}^{k+1}|_{\lambda=\rho=0}
		={\hh}^{k+1}$.
		Starting from~\eqref{eqn:S-extH}, the induction proves the Lemma.
	\end{proof}
	\begin{proposition}\label{prop:twisted-trace-limits}
		For $1\leq k-1\leq\delta$ and $0\leq a\leq k-1$, set
		$b=k-1-a$. We have
		\begin{align}
			\lim_{\lambda\to0}\lim_{\rho\to0}
			\tr\bigl(\widetilde\hh^{\twqp(k-1)}\twqp\bigr)
			=&\, \tr_{\rm amb}\bigl(\widetilde\hh^{*(k-1)}*\bigr),
			\label{eqn:full-amb-tr-tw-1} \\
			\lim_{\lambda\to0}\lim_{\rho\to0}
			\tr\bigl(\widetilde\hh^{\twqp a}\twqp \mu\widetilde\hh^{\twqp b}\twqp\mu\bigr)
			=&\, \tr_{\rm amb}\bigl(\widetilde\hh^{*a}*\mu\widetilde\hh^{*b}*\mu\bigr).
			\label{eqn:full-amb-tr-tw-2}
		\end{align}
		Moreover,
		\begin{align}
			&\lim_{\lambda\to0}\lim_{\rho\to0}
			\tr\bigl(
			\widetilde\hh^{\twqp a}\twqp
			(\mathcal V-\mu)\widetilde\hh^{\twqp b}\twqp
			(\mathcal V-\mu)
			\bigr)=0,
			\label{eqn:full-amb-tr-tw-3}\\
			&\lim_{\lambda\to0}\lim_{\rho\to0}
			\sum_\alpha
			\tr\bigl(
			\widetilde\hh^{\twqp a}\twqp
			\Lambda(x_\alpha^b\Psi^{\bar\alpha}_0)
			\bar e_\alpha\twqp
			\bigr)=0,
			\label{eqn:full-amb-tr-tw-5}\\
			&\lim_{\lambda\to0}\lim_{\rho\to0}
			\tr\bigl(
			\widetilde\hh^{\twqp k}\twqp\Lambda(R_1^*)
			\bigr)=0.
			\label{eqn:full-amb-tr-tw-4}
		\end{align}
		Here $R_1^*$ is represented with respect to the flat basis.
	\end{proposition}
	
	\begin{proof}
		By the convention fixed at the beginning of this subsection, we first set $\rho=0$.
		In particular, we have $\Lambda=\lambda\pd_{\lambda}$.
		
		Consider the flat basis $\{\hh^j\}_{j=0}^{N-1}$. 
		After taking the limit $\lambda\to0$, the twisted quantum product on $\sspan\{1,\hh,\ldots,\hh^\delta\}$ specializes to the untwisted quantum product on the ambient cohomology. 
		For $j=\delta+1,\ldots,N-1$, Lemma~\ref{lem:QP-ext} shows that successive
		multiplication by $\widetilde\hh$ produces no $\hh^j$-component.
		This proves~\eqref{eqn:full-amb-tr-tw-1}. Since $\mu$ is diagonal
		with respect to the flat basis, the same argument proves
		\eqref{eqn:full-amb-tr-tw-2}.
		
		We prove~\eqref{eqn:full-amb-tr-tw-3} by a direct computation in the canonical basis. 
		Recall that with respect to the normalized canonical basis, $\mathcal V-\mu$ has the matrix form $-\Lambda(\Psi)\Psi^{-1}$.
		By using equation~\eqref{eqn:Psi-check-Psi} and the independence of $\mathcal B$ with $\lambda$,
		we see that with respect to the canonical basis $\{e_\alpha\}$, $\mathcal V-\mu$ has the matrix form
		$$
		-\operatorname{diag}\left(\Lambda(\Psi^{\bar\alpha}_0)/\Psi^{\bar\alpha}_0\right)
		-\Lambda(\check\Psi)\check\Psi^{-1}.
		$$
		Recall that $\check\Psi^\alpha_j=x_\alpha^j$. 
		Let $M=\sum_i m_i$ and ${\bf m}^{\bf m}=\prod_i m_i^{m_i}$. 
		Then equations~\eqref{eqn:PF-0} and~\eqref{eqn:Psi0} give
		$$
		\Lambda(x_\alpha)
		=
		\frac{Nx_\alpha(x_\alpha^{\fanoind}-{\bf m}^{\bf m}q)}
		{Nx_\alpha^{\fanoind}-M{\bf m}^{\bf m}q},
		\qquad
		\frac{\Lambda(\Psi^{\bar\alpha}_0)}
		{\Psi^{\bar\alpha}_0}
		=
		\frac{\Lambda(x_\alpha)}{2}
		\left(
		\frac{r}{x_\alpha}
		-\frac{P''(x_\alpha,q)}{P'(x_\alpha,q)}
		\right).
		$$
		Furthermore, by the Lagrange interpolation formula,
		$$
		\bigl(\Lambda(\check\Psi)\check\Psi^{-1}\bigr)^\alpha_\beta
		=
		\begin{cases}
			\displaystyle
			\frac{\Lambda(x_\alpha)P'(x_\alpha,q)}
			{(x_\alpha-x_\beta)P'(x_\beta,q)},
			&\alpha\ne\beta,\\[8pt]
			\displaystyle
			\frac{\Lambda(x_\alpha)P''(x_\alpha,q)}
			{2P'(x_\alpha,q)},
			&\alpha=\beta.
		\end{cases}
		$$
		It follows that the entries of $\mathcal V-\mu$ in the canonical basis
		are
		\[
		(\mathcal V-\mu)^\alpha_\beta
		=
		\begin{cases}
			\displaystyle
			-\frac{\Lambda(x_\alpha)P'(x_\alpha,q)}
			{(x_\alpha-x_\beta)P'(x_\beta,q)},
			&\alpha\ne\beta,\\[8pt]
			\displaystyle
			-\frac{r}{2}\frac{\Lambda(x_\alpha)}{x_\alpha},
			&\alpha=\beta.
		\end{cases}
		\]
		In particular, for $\alpha\ne\beta$,
		\[
		(\mathcal V-\mu)^\alpha_\beta
		(\mathcal V-\mu)^\beta_\alpha
		=
		-\frac{\Lambda(x_\alpha)\Lambda(x_\beta)}
		{(x_\alpha-x_\beta)^2}.
		\]
		Since quantum multiplication by $\widetilde\hh$ is diagonal in the
		canonical basis, the $\alpha$-th diagonal entry of the endomorphism in
		\eqref{eqn:full-amb-tr-tw-3} is
		\[
		\frac{r^2}{4}x_\alpha^{k-1}
		\left(\frac{\Lambda(x_\alpha)}{x_\alpha}\right)^2
		-
		\sum_{\beta\ne\alpha}
		x_\alpha^a x_\beta^b
		\frac{\Lambda(x_\alpha)\Lambda(x_\beta)}
		{(x_\alpha-x_\beta)^2}.
		\]
		For the $M$ roots converging to zero, the defining equation gives
		$x_\alpha^M\sim-\lambda^N/({\bf m}^{\bf m}q)$, and their leading
		terms differ by distinct $M$-th roots of unity. Hence
		$\Lambda(x_\alpha)/x_\alpha$ and
		$x_\alpha x_\beta/(x_\alpha-x_\beta)^2$ are bounded for distinct roots
		converging to zero. Since $a+b=k-1\geq1$, all the corresponding terms
		in the preceding diagonal entry tend to zero.
		For the remaining roots,
		$x_\alpha^{\fanoind}\to{\bf m}^{\bf m}q$, and the formula for $\Lambda(x_\alpha)$ gives $\Lambda(x_\alpha)\to0$.  
		The same conclusion is immediate when
		one of $x_\alpha,x_\beta$ converges to zero and the other does not.
		Thus every diagonal entry tends to zero, and
		\eqref{eqn:full-amb-tr-tw-3} follows.
		Equation~\eqref{eqn:full-amb-tr-tw-5} follows similarly.
		
		It remains to prove~\eqref{eqn:full-amb-tr-tw-4}. 
		As in the preceding computations, it is enough to consider the possible nonregular terms of $R_1^*$ with respect to the quantum-power basis $\{\widetilde\hh^{\twqp k}\}_{k=0}^{N-1}$.
		Since $\mathcal B$ is independent of $\lambda$, $\Lambda$ may be applied directly to the matrix coefficients with respect to this basis. 
		Recall that $(R_1^*)^\alpha_0$ is the $e_\alpha$-component of $R_1^*{\bf1}$. Formula~\eqref{eqn:formula-R1} shows that its only possible nonregular part is a $\lambda$-independent scalar multiple of $x_\alpha^{-1}$. 
		Indeed, after setting $\rho=0$, a direct computation reduces formula~\eqref{eqn:formula-R1} to
		$$
		(R_1^*)^\alpha_0=\frac{a_0+a_1\cX+a_2\cX^2+a_3\cX^3}{24Mx_\alpha},
		$$
		where $\cX=\frac{Nx_\alpha^{\fanoind}}{Nx_\alpha^{\fanoind}-M{\bf m}^{\bf m}q}$ 
		and $a_i$, $i=0,1,2,3$, are constants that independent on $\lambda$. 
		Except for the term $\frac{a_0}{24Mx_\alpha}$, the same argument as above shows that Lagrange interpolation yields regular quantum-power coefficients.
		The nonregular term in $R_1^*{\bf1}$ has behavior
		$$
		\sum_\alpha\frac{1}{x_\alpha}e_\alpha
		=\frac{\widetilde\hh^{\twqp(N-1)}-{\bf m}^{\bf m}q\,\widetilde\hh^{\twqp(M-1)}}{\lambda^N}
		$$
		up to a constant factor $\frac{a_0}{24M}$.
		Applying $\Lambda$ to these coefficients, writing the result back in the canonical basis, and multiplying by $\widetilde\hh^{\twqp k}$ gives $-N\sum_\alpha x_\alpha^{k-1}e_\alpha$.
		After transforming back by $\check\Psi^{-1}$, its
		$\widetilde\hh^{\twqp i}$-component is
		$-N\sum_\alpha (\check\Psi^{-1})^i_\alpha x_\alpha^{k-1}=-N\delta_{i,k-1}$.
		Since $k-1\geq1$, its $\widetilde\hh^{\twqp0}$-component is zero.
		Thus it makes no contribution to the diagonal entry of the column indexed by ${\bf1}$.
		The remaining part of this column has regular quantum-power coefficients and tends to zero after applying $\Lambda$.
		For $R_1^*{\widetilde\hh^{\twqp j}}$, $j=1,\cdots,N-1$, equation~\eqref{eqn:R1k} and the same computation as above show that, 
		after applying $\Lambda$ and multiplying by $\widetilde\hh^{\twqp k}$ with $k\geq 2$, the contributions of the possible nonregular terms to the corresponding diagonal entries tend to zero.
		This proves~\eqref{eqn:full-amb-tr-tw-4}.
	\end{proof}
	
\end{appendices}

\bibliographystyle{alpha}
\bibliography{biblio}

@article{ciocan2014wall,
  title={Wall-crossing in genus zero quasimap theory and mirror maps},
  author={Ciocan-Fontanine, Ionut and Kim, Bumsig},
  journal={Algebraic Geometry},
  volume={1},
  number={4},
  pages={400--448},
  year={2014},
  publisher={European Mathematical Society Publishing House}
}

@article{ciocan2017higher,
  title={Higher genus quasimap wall-crossing for semipositive targets},
  author={Ciocan-Fontanine, Ionu{\c{t}} and Kim, Bumsig},
  journal={Journal of the European Mathematical Society},
  volume={19},
  number={7},
  pages={2051--2102},
  year={2017}
}

@article{hassett2003moduli,
  title={Moduli spaces of weighted pointed stable curves},
  author={Hassett, Brendan},
  journal={Advances in Mathematics},
  volume={173},
  number={2},
  pages={316--352},
  year={2003},
  publisher={Elsevier}
}

@article{guo2016genus,
  title={Genus-One Mirror Symmetry in the {Landau--Ginzburg} Model},
  author={Guo, Shuai and Ross, Dustin},
  journal={Algebraic Geometry},
  year={2019},
  volume={6},
  number={3},
  pages={260-301}
}

@article{Clader2024HigherGenus,
  title= {Higher-genus quasimap wall-crossing via localization},
  author  = {Clader, Emily and Janda, Felix and Ruan, Yongbin},
  journal = {Algebraic Geometry},
  volume  = {11},
  number  = {5},
  pages   = {712--736},
  year    = {2024},
  doi     = {10.14231/AG-2024-021},
  url     = {https://doi.org/10.14231/AG-2024-021}
}

@article{clader2017higherLG,
author = {Emily Clader and Felix Janda and Yongbin Ruan},
title = {{Higher-genus wall-crossing in the gauged linear sigma model}},
volume = {170},
journal = {Duke Mathematical Journal},
number = {4},
publisher = {Duke University Press},
pages = {697 -- 773},
year = {2021},
doi = {10.1215/00127094-2020-0053},
URL = {https://doi.org/10.1215/00127094-2020-0053}
}

@article{ciocan2020quasimap,
  title={Quasimap wall-crossings and mirror symmetry},
  author={Ciocan-Fontanine, Ionu{\c{t}} and Kim, Bumsig},
  journal={Publications math{\'e}matiques de l'IH{\'E}S},
  pages={1--60},
  year={2020},
  publisher={Springer}
}

@article{ciocan2014stable,
  title={Stable quasimaps to {GIT} quotients},
  author={Ciocan-Fontanine, Ionu{\c{t}} and Kim, Bumsig and Maulik, Davesh},
  journal={Journal of Geometry and Physics},
  volume={75},
  pages={17--47},
  year={2014},
  publisher={Elsevier}
}

@article{zhou2020higher-with-pg-no,
  title={Higher-genus wall-crossing in {Landau--Ginzburg} theory},
  author={Zhou, Yang},
  journal={Advances in Mathematics},
  volume={361},
  pages={Article 106914, 1-42},
  year={2020},
  publisher={Elsevier}
}

@article{zinger2009reduced,
  title={The reduced genus 1 {Gromov--Witten} invariants of {Calabi--Yau} hypersurfaces},
  author={Zinger, Aleksey},
  journal={Journal of the American Mathematical Society},
  volume={22},
  number={3},
  pages={691--737},
  year={2009}
}

@article{popa2013genus,
  title={The genus one {Gromov--Witten} invariants of {Calabi--Yau} complete intersections},
  author={Popa, Alexandra},
  journal={Transactions of the American Mathematical Society},
  volume={365},
  number={3},
  pages={1149--1181},
  year={2013}
}

@article{cheong2015orbifold,
  title={Orbifold quasimap theory},
  author={Cheong, Daewoong and Ciocan-Fontanine, Ionu{\c{t}} and Kim, Bumsig},
  journal={Mathematische Annalen},
  volume={363},
  number={3-4},
  pages={777--816},
  year={2015},
  publisher={Springer}
}

@article{kim2018mirror,
  title={Mirror theorem for elliptic quasimap invariants},
  author={Kim, Bumsig and Lho, Hyenho},
  journal={Geometry \& Topology},
  volume={22},
  number={3},
  pages={1459--1481},
  year={2018},
  publisher={Mathematical Sciences Publishers}
}

@inproceedings{ciocan-fontanine2016,
address = "Tokyo, Japan",
author = "Ciocan-Fontanine, Ionu{\c{t}} and Kim, Bumsig",
booktitle = "Development of Moduli Theory — Kyoto 2013",
doi = "10.2969/aspm/06910323",
pages = "323--347",
publisher = "Mathematical Society of Japan",
title = "{Big $I$-functions}",
url = "https://doi.org/10.2969/aspm/06910323",
year = "2016"
}

@article{lian1997mirror,
  title={Mirror principle {I}},
  author={Lian, Bong and Kefeng Liu and Shing-Tung Yau},
  journal={Asian J. Math.},
  volume={1},
  number={4},
  pages={729--763},
  year={1997}
}

@article{Wang_2025,
title={A mirror theorem for Gromov-Witten theory without convexity},
volume={13},
DOI={10.1017/fms.2025.34},
journal={Forum of Mathematics, Sigma},
author={Wang, Jun},
year={2025},
pages={e72}}

@article{chang2018polynomial,
	author  = {Chang, Huai-Liang and Guo, Shuai and Li, Jun},
	title   = {Polynomial structure of {Gromov--Witten} potential of quintic {$3$}-folds},
	journal = {Annals of Mathematics},
	volume  = {194},
	number  = {3},
	pages   = {585--645},
	year    = {2021},
	doi     = {10.4007/annals.2021.194.3.1}
}

@article{lee2001quantum,
  title={Quantum {Lefschetz} hyperplane theorem},
  author={Lee, Y-P},
  journal={Inventiones mathematicae},
  volume={145},
  number={1},
  pages={121--149},
  year={2001},
  publisher={Springer}
}

@article{coates2007quantum,
  title={Quantum {Riemann--Roch}, {Lefschetz} and {Serre}},
  author={Coates, Tom and Givental, Alexander},
  journal={Annals of Mathematics},
  pages={15--53},
  year={2007},
  publisher={JSTOR}
}

@article{zhou2022quasimap,
  title={Quasimap wall-crossing for {GIT} quotients},
  author={Zhou, Yang},
  journal={Inventiones mathematicae},
  volume={227},
  number={2},
  pages={581--660},
  year={2022},
  publisher={Springer}
}

@phdthesis{pinharry2020weighted,
  title={Weighted quasimap wall-crossing via localization},
  author={Pinharry, Michelle},
  year={2020},
  school={University of Minnesota}
}

@article{hu2015big,
  title={Big quantum cohomology of {Fano} complete intersections},
  author={Hu, Xiaowen},
  journal={arXiv preprint arXiv:1501.03683},
  year={2015}
}

@article{Wit91,
author = {Witten, Edward},
title = {Two-dimensional gravity and intersection theory on the moduli space},
journal = {Surveys in Differential Geometry},
volume = {1},
pages = {243--310},
year = {1991},
doi = {10.4310/SDG.1990.v1.n1.a5}
}

@article{Kont92,
	author = {Kontsevich, Maxim},
	title = {Intersection theory on the moduli space of curves and the matrix {Airy} function},
	journal = {Communications in Mathematical Physics},
	volume = {147},
	number = {1},
	pages = {1--23},
	year = {1992},
	doi = {10.1007/BF02099526}
}

@article{DZ98,
	author = {Dubrovin, Boris and Zhang, Youjin},
	title = {Bihamiltonian hierarchies in {2D} topological field theory at one-loop approximation},
	journal = {Communications in Mathematical Physics},
	volume = {198},
	number = {2},
	pages = {311--361},
	year = {1998},
	doi = {10.1007/s002200050480}
}

@article{DZ99,
	author = {Dubrovin, Boris and Zhang, Youjin},
	title = {Frobenius manifolds and {Virasoro} constraints},
	journal = {Selecta Mathematica},
	volume = {5},
	number = {4},
	pages = {423--466},
	year = {1999},
	doi = {10.1007/s000290050053}
}

@incollection{Get99,
	author = {Getzler, Ezra},
	title = {The {Virasoro} conjecture for {Gromov--Witten} invariants},
	booktitle = {Algebraic geometry: Hirzebruch 70},
	series = {Contemporary Mathematics},
	volume = {241},
	pages = {147--176},
	publisher = {American Mathematical Society},
	address = {Providence, RI},
	year = {1999}
}

@incollection{Get04,
author = {Getzler, Ezra},
title = {The jet-space of a {Frobenius} manifold and higher-genus {Gromov--Witten} invariants},
booktitle = {Frobenius Manifolds: Quantum Cohomology and Singularities},
pages = {45--89},
year = {2004},
publisher = {Springer}
}

@article{LT98,
	author = {Liu, Xiaobo and Tian, Gang},
	title = {{Virasoro} constraints for quantum cohomology},
	journal = {Journal of Differential Geometry},
	volume = {50},
	number = {3},
	pages = {537--590},
	year = {1998}
}

@article{Liu01,
	author = {Liu, Xiaobo},
	title = {Elliptic {Gromov--Witten} invariants and {Virasoro} conjecture},
	journal = {Communications in Mathematical Physics},
	volume = {216},
	number = {3},
	pages = {705--728},
	year = {2001},
	doi = {10.1007/s002200000360}
}

@article{Liu02,
	author = {Liu, Xiaobo},
	title = {Quantum product on the big phase space and the {Virasoro} conjecture},
	journal = {Advances in Mathematics},
	volume = {169},
	number = {2},
	pages = {313--375},
	year = {2002},
	doi = {10.1006/aima.2001.2046}
}

@article{Tel12,
	author = {Teleman, Constantin},
	title = {The structure of {2D} semi-simple field theories},
	journal = {Inventiones Mathematicae},
	volume = {188},
	number = {3},
	pages = {525--588},
	year = {2012},
	doi = {10.1007/s00222-011-0352-5}
}

@article{Giv01a,
	author = {Givental, Alexander},
	title = {{Gromov--Witten} invariants and quantization of quadratic {Hamiltonians}},
	journal = {Moscow Mathematical Journal},
	volume = {1},
	number = {4},
	pages = {551--568},
	year = {2001}
}

@article{Giv01b,
	author = {Givental, Alexander},
	title = {Semisimple {Frobenius} structures at higher genus},
	journal = {International Mathematics Research Notices},
	volume = {2001},
	number = {23},
	pages = {1265--1286},
	year = {2001},
	doi = {10.1155/S1073792801000605}
}

@incollection{Giv04,
	author = {Givental, Alexander},
	title = {Symplectic geometry of {Frobenius} structures},
	booktitle = {Frobenius Manifolds: Quantum Cohomology and Singularities},
	pages = {91--112},
	year = {2004},
	publisher = {Springer}
}

@article{Giv96,
	author = {Givental, Alexander},
	title = {Equivariant {Gromov--Witten} invariants},
	journal = {International Mathematics Research Notices},
	volume = {1996},
	number = {13},
	pages = {613--663},
	year = {1996},
	doi = {10.1155/S1073792896000414}
}

@article{OP06,
	author = {Okounkov, Andrei and Pandharipande, Rahul},
	title = {{Virasoro} constraints for target curves},
	journal = {Inventiones Mathematicae},
	volume = {163},
	number = {1},
	pages = {47--108},
	year = {2006},
	doi = {10.1007/s00222-005-0455-y}
}

@article{vakil-zinger,
author = {Ravi Vakil and Aleksey Zinger},
title = {{A desingularization of the main component of the moduli space of genus-one stable maps into $\mathbb P^n$}},
volume = {12},
journal = {Geometry \& Topology},
number = {1},
publisher = {MSP},
pages = {1 -- 95},
year = {2008},
doi = {10.2140/gt.2008.12.1},
URL = {https://doi.org/10.2140/gt.2008.12.1}
}

@article{li2009genus,
  title={On the genus-one {Gromov-Witten} invariants of complete intersections},
  author={Li, Jun and Zinger, Aleksey},
  journal={Journal of Differential Geometry},
  volume={82},
  number={3},
  pages={641--690},
  year={2009},
  publisher={Lehigh University}
}

@article{Hu22,
	author = {Hu, Xiaowen},
	title = {On genus 1 {Gromov--Witten} invariants of {Fano} complete intersections},
	journal = {arXiv preprint arXiv:2203.08091},
	year = {2022}
}

@article{GZ25,
	author  = {Guo, Shuai and Zhang, Qingsheng},
	title   = {Cohomological field theory with vacuum and its {Virasoro} constraints},
	journal = {Mathematische Annalen},
	volume  = {396},
	number  = {1},
	year    = {2026},
	doi     = {10.1007/s00208-026-03552-z},
	url     = {https://doi.org/10.1007/s00208-026-03552-z}
}

@article{Get97,
	author = {Getzler, Ezra},
	title = {Intersection theory on \({\Mbar}_{1,4}\) and elliptic {Gromov--Witten} invariants},
	journal = {Journal of the American Mathematical Society},
	volume = {10},
	number = {4},
	pages = {973--998},
	year = {1997},
	doi = {10.1090/S0894-0347-97-00246-4}
}

@article{CMP10,
	author  = {Chaput, Pierre-Emmanuel and Manivel, Laurent and Perrin, Nicolas},
	title   = {Quantum cohomology of minuscule homogeneous spaces {III}:
	semi-simplicity and consequences},
	journal = {Canadian Journal of Mathematics},
	volume  = {62},
	number  = {6},
	pages   = {1246--1263},
	year    = {2010},
	doi     = {10.4153/CJM-2010-050-9}
}

@incollection{BM04,
	author    = {Bayer, Arend and Manin, Yuri I.},
	title     = {{(Semi)simple exercises in quantum cohomology}},
	booktitle = {The Fano Conference},
	pages     = {143--173},
	publisher = {Universit{\`a} di Torino},
	address   = {Turin},
	year      = {2004},
	eprint    = {math/0103164},
	archivePrefix = {arXiv},
	primaryClass  = {math.AG}
}

@article{Hu21,
	author  = {Hu, Xiaowen},
	title   = {Big quantum cohomology of even dimensional intersections
	of two quadrics},
	journal = {arXiv preprint arXiv:2109.11469},
	year    = {2021}
}

@article{EHX97,
	author  = {Eguchi, Tohru and Hori, Kentaro and Xiong, Chuan-Sheng},
	title   = {Quantum cohomology and {Virasoro} algebra},
	journal = {Physics Letters B},
	volume  = {402},
	number  = {1--2},
	pages   = {71--80},
	year    = {1997},
	doi     = {10.1016/S0370-2693(97)00401-2},
	eprint  = {hep-th/9703086}
}

@article{EJX98,
	author  = {Eguchi, Tohru and Jinzenji, Masao and Xiong, Chuan-Sheng},
	title   = {Quantum cohomology and free-field representation},
	journal = {Nuclear Physics B},
	volume  = {510},
	number  = {3},
	pages   = {608--622},
	year    = {1998},
	doi     = {10.1016/S0550-3213(97)00730-X},
	eprint  = {hep-th/9709152}
}

@article{KM98,
	author  = {Kontsevich, Maxim and Manin, Yuri I.},
	title   = {Relations between the correlators of the topological
	sigma-model coupled to gravity},
	journal = {Communications in Mathematical Physics},
	volume  = {196},
	number  = {2},
	pages   = {385--398},
	year    = {1998},
	doi     = {10.1007/s002200050426},
	eprint  = {alg-geom/9708024}
}

@incollection{Giv98,
	author    = {Givental, Alexander B.},
	title     = {Elliptic {Gromov--Witten} invariants and the generalized
	mirror conjecture},
	booktitle = {Integrable Systems and Algebraic Geometry
	(Kobe/Kyoto, 1997)},
	editor    = {Saito, M.-H. and Shimizu, Y. and Ueno, K.},
	pages     = {107--155},
	publisher = {World Scientific},
	address   = {River Edge, NJ},
	year      = {1998},
	eprint    = {math/9803053}
}

@article{ABPZ23,
	author  = {Arg{\"u}z, H{\"u}lya and Bousseau, Pierrick and
	Pandharipande, Rahul and Zvonkine, Dimitri},
	title   = {{Gromov--Witten} theory of complete intersections via nodal invariants},
	journal = {Journal of Topology},
	volume  = {16},
	number  = {1},
	pages   = {264--343},
	year    = {2023},
	doi     = {10.1112/topo.12284}
}
\end{document}